\documentclass[aop,MSNbibl,seceqn,dvips]{arximspdf}
\usepackage{graphicx}

\doi{10.1214/12-AOP831} %kopijuoti is PTS
\volume{41}
\issue{6}
\pubyear{2013}
\firstpage{4002}
\lastpage{4049}

\makeatletter

\newcommand{\rrvert}{\vert}
\newcommand{\llvert}{\vert}
\newcommand{\eqref}[1]{(\ref{#1})}
\newtheorem{theorem}{Theorem}[section]
\newtheorem{lemma}[theorem]{Lemma}
\newtheorem{problem}[theorem]{Problem}
\newtheorem{conjecture}[theorem]{Conjecture}

\def\prt{\partial}
\def\eps{\varepsilon}
\def\R{\mathbb{R}}
\def\Z{\mathbb{Z}}
\def\P{\mathbb{P}}
\def\E{\mathbb{E}}
\def\bbT{\mathbb{T}}
\def\bn{\mathbf{n}}
\def\bv{\mathbf{v}}
\def\bw{\mathbf{w}}
\def\bz{\mathbf{z}}
\def\bQ{\mathbf{Q}}
\def\T{\mathcal{T}}
\def\B{\mathcal{B}}
\def\C{\mathcal{C}}
\def\EE{\mathcal{E}}
\def\S{\mathcal{S}}
\def\I{\mathcal{I}}

\def\bDelta{\bolds{\Delta}}

\def\A{{\mathcal A}}
\def\tngt{\T}
\def\torus{\bbT}
\def\sh{\S}
\newcommand{\<}{\langle}
\renewcommand{\>}{\rangle}
\def\F{\mathcal{F}}
\def\G{\mathcal{G}}

\def\n{\mathbf{n}}
\def\ol{\overline}
\def\bone{\mathbf{1}}
\def\wh{\widehat}
\def\wt{\widetilde}

\newcommand{\dist}{\mathrm{dist}}

\makeatother

\begin{document}
\begin{frontmatter}

\title{Brownian earthworm\thanksref{T1}}
\runtitle{Brownian earthworm}
\thankstext{T1}{Supported in part by NSF
Grants DMS-09-06743, DMS-10-07563, DMR-10-35196
and by Grant N N201 397137, MNiSW, Poland.}

\begin{aug}
\author[A]{\fnms{Krzysztof} \snm{Burdzy}\corref{}\ead[label=e1]{burdzy@math.washington.edu}},
\author[A]{\fnms{Zhen-Qing} \snm{Chen}\ead[label=e2]{zchen@math.washington.edu}}
\and
\author[A]{\fnms{Soumik} \snm{Pal}\ead[label=e3]{soumik@math.washington.edu}}
\runauthor{K. Burdzy, Z.-Q. Chen and S. Pal}
\affiliation{University of Washington}
\address[A]{Department of Mathematics\\
University of Washington\\
Box 354350\\
Seattle, Washington 98195\\
USA\\
\printead{e1}\\
\hphantom{E-mail:\ }\printead*{e2}\\
\hphantom{E-mail:\ }\printead*{e3}} %adresu isvedimo komanda gale!
\end{aug}

% HISTORY:
\received{\smonth{9} \syear{2011}}
\revised{\smonth{12} \syear{2012}}

% ABSTRACT
%
\begin{abstract}
We prove that the distance between two reflected Brownian
motions, driven by the same white noise, outside a sphere in a
3-dimensional flat torus does not
converge to 0, a.s., if the radius of the sphere is
sufficiently small, relative to the size of the torus.
\end{abstract}

% KEYWORDS
\begin{keyword}[class=AMS]
\kwd{60J65}
\end{keyword}
\begin{keyword}
\kwd{Reflected Brownian motion}
\end{keyword}

\end{frontmatter}

%s1 #&#
\section{Introduction}\label{secintro}

This article is partly motivated by a natural phenome\-non. We
would like to analyze the effect of a randomly moving earthworm
on the soil. The soil is pushed aside by the earthworm. What is
the asymptotic distribution of soil particles when time goes to
infinity? Is the soil compacted, or are soil particles more or
less evenly spread over the region, especially when the
earthworm is small compared to the size of the region?
The answer seems to depend on the shape of the earthworm;
for example, we believe that
the soil is compacted if the ``earthworm'' is cubical.
In our toy model, the earthworm is represented by a sphere following a
Brownian path. We conjecture that in this model, the soil particles
will be more or less evenly spread over the region.
Our rigorous results in this paper
partly justify these heuristic claims.
We will next state the model in rigorous terms
and then present a theorem and some conjectures. We will also
briefly review related results.
The earthworm picture will be mathematically interpreted
after Conjecture~\ref{prj182}.

Let $\torus_1$ be the flat $d$-dimensional torus with side length
2, that is, $\torus_1$ is the cube $\{(x_1,\ldots,x_d) \in\R^d\dvtx |x_k| \leq1 \mbox{ for } k=1,\ldots,d\}$,
with the opposite sides
identified in the usual way. Let $\B(x,r)$ denote the open ball
with center $x$ and radius $r$.
For $0<r<1$, let $D = \torus_1 \setminus\ol{\B(0,r)}$.
Let $\n(x)$ denote the unit inward normal vector
at $x\in\prt D = \prt\B(0,r)$.
Let $B$ be a standard
$d$-dimensional Brownian motion, $x_0, y_0 \in\ol D$, $x_0 \ne
y_0$ and consider the following Skorokhod equations:
%
%e1.1 #&#
%e1.2 #&#
%
\begin{eqnarray}
\label{eqj131} X_t &=& x_0 + B_t + \int
_0^t \n(X_s)
\,dL^X_s \qquad\mbox{for } t\geq0,
\\
Y_t &=& y_0 + B_t + \int
_0^t \n(Y_s)
\,dL^Y_s \qquad\mbox{for } t\geq0.\label{eqj132}
\end{eqnarray}
Here $L^X$ is the local time of $X$ on $\prt D$. In other
words, $L^X$ is a nondecreasing continuous process which does
not increase when $X$ is in $D$, that is, $\int_0^\infty
\bone_{D}(X_t) \,dL^X_t = 0$, a.s. Equation \eqref{eqj131} has
a unique pathwise solution $(X,L^X)$ such that $X_t \in\ol D$
for all $t\geq0$; see \cite{LS}. The reflected Brownian
motion $X$ is a strong Markov process. The same remarks apply
to \eqref{eqj132}, so $(X, Y)$ is also strong Markov. Note
that on
any time
interval $(s,t)$ such that $X_u \in D$ and $Y_u \in
D$ for all $u \in(s,t)$, we have $X_u - Y_u = X_s - Y_s$ for
all $u \in(s,t)$.

For $x, y\in\torus_1$, we use $\dist(x, y)$ to denote the
geodesic distance between $x$ and~$y$ in the torus $\torus_1$.

%
%th1.1 #&#
\begin{theorem}\label{eqj135}
When the dimension $d=3$,
there is
$r_0>0$ such that for every $r\leq r_0$ and every
$x_0\ne y_0$, we have $\limsup_{t\to
\infty} \dist(X_t,Y_t) >0$, a.s.
\end{theorem}

An analogous problem was considered in \cite{BCJ} for
planar domains $D$.
It was proved that if $D$ is a bounded domain
with a smooth boundary and at most one hole, then $\lim_{t\to
\infty} \dist(X_t,Y_t) = 0$, a.s. It is not known whether there exists
a two-dimensional domain $D$ such that we have
$\limsup_{t\to\infty}\dist(X_t,Y_t)
>0$ with positive probability.

Note that by the pathwise uniqueness of the solutions to
\eqref{eqj131}--\eqref{eqj132}, $0$ is an absorbing state
for the distance process $\dist(X_t, Y_t)$; that is,
if $\dist(X_{t_0}, Y_{t_0})=0$, then $\dist(X_t, Y_t)=0$
for all $t\geq t_0$. Theorem \ref{eqj135} says that
$\dist(X_t, Y_t)$ never enters the absorbing state $0$
nor converges to 0 as $t\to\infty$. Since $D$ is compact,
this suggests that $\dist(X_t, Y_t)$ fluctuates and is a
``recurrent'' process. We suspect
that $(X_t,Y_t)$ has a
stationary probability distribution
but this does not follow from
recurrence alone. Hence, we propose the following

%
%co1.2 #&#
\begin{conjecture}\label{prd71}
When the dimension $d=3$,
there is
$r_0>0$ such that for $r\leq r_0$
the process $(X,Y)$ has a stationary
distribution $Q$ which does not charge the diagonal
$\{(x,x)\dvtx x\in\ol D \}$.
There is only one stationary
distribution for $(X,Y)$ which does not charge the diagonal.
\end{conjecture}

Since \eqref{eqj131}--\eqref{eqj132} have a unique pathwise
solution, if $x_0=y_0$, then $X_t = Y_t$ for all $t\geq0$, a.s.
It follows that $(X,Y)$ has a unique stationary distribution
$Q'$ supported on the diagonal, characterized by the fact that
the distribution of $X$ under $Q'$ is uniform in $D$.\eject

Our state space $D$ for reflected Brownian motion is a subset
of a torus because three-dimensional Brownian motion is
transient so the result analogous to Theorem \ref{eqj135} for
the complement of a ball in $\R^3$ is not interesting.
Moreover, the boundary of $D$ has no other component besides
$\prt\B(0,r)$ so the relative position of
$X$ and $Y$ is determined solely by the interaction
of the processes with $\prt\B(0,r)$.

%
%pr1.3 #&#
\begin{problem}\label{pro51}
Is Theorem \ref{eqj135} valid when the dimension $d=2$?
\end{problem}

The reader may find it paradoxical that we can prove Theorem
\ref{eqj135} in 3 dimensions, but the analogous result in 2
dimensions is stated as an open problem. The reason is that the
proof depends in a crucial way on the sign of a certain
``Lyapunov exponent'' $\lambda^*_\rho=1+\lambda_\rho$ where
$\rho:=1/r$ and
$\lambda_\rho$ is defined in Theorem~\ref{tho53}(ii)
relative to the domain $D$. We prove in Lemma \ref{lemj271}
that $\lambda^*_\rho$ is positive for $D$ if $d=3$ and $\rho$
is large. In the 2-dimensional case, the analogous exponent is
equal to 0 \cite{BCJ}, Proposition~2.3, and this critical value
makes the problem harder.
We could have defined the domain $D$ as
$\torus_1 \setminus A$, with $A$ being not necessarily a ball. It is
easy to see that for many sets $A$, for example, those that are
bounded, smooth and close to a polyhedron, $\lambda^*$ is
negative. It was shown in \cite{BCJ} that in 2-dimensional
space, negative $\lambda^*$ implies that
$\lim_{t\to\infty}\dist(X_t,Y_t) =0$, a.s. In such a case,
$(X,Y)$ does not have a stationary distribution with some mass
outside the diagonal. It is not known whether there is a
2-dimensional domain, bounded or unbounded, with positive
$\lambda^*$.
This is related to another open problem that we have already
mentioned---it is not known whether there exists a two-dimensional
domain $D$ such that $\limsup_{t\to\infty}\dist(X_t,Y_t)>0$ with
positive probability. Theorem~\ref{eqj135} shows that this is
the case for a subset of a three-dimensional torus. We believe
that the theorem also holds in some bounded subsets of $\R^3$,
but we will not provide a rigorous proof. We make this claim
more precise in the following conjecture.

%
%co1.4 #&#
\begin{conjecture}
Suppose that $\B(x_j, r) \subset\B(0,1)$ for $j=1,\ldots, k$,
and let $D_1=\B(0, 1) \setminus\bigcup_{j=1}^k \ol{\B(x_j, r)}
\subset\R^3$. If $k$ is sufficiently large and\break
$ (\min_{1\leq j\leq k}( 1-|x_j|) +\min_{1\leq i<j\leq k}
|x_i-x_j| )/r$
is sufficiently
large,
then Theorem \ref{eqj135} holds for~$D_1$.
\end{conjecture}

Suppose that Conjecture \ref{prd71} is true, that is, for some
$r_0>0$ and all $r\leq r_0$,
the process $(X,Y)$ has a stationary
distribution $Q$ which does not charge the diagonal. This stationary
measure $Q$ depends on
$r$, the radius of the ball deleted from the torus $\torus_1$,
so we can write $Q_{r}$ to emphasize this dependence.

%
%co1.5 #&#
\begin{conjecture}\label{prj181}
The measures $Q_{r}$ converge to the uniform probability
distribution on $(\torus_1)^2$ as $r\to0$.
\end{conjecture}

Next, we consider the flow $X^x_t$ of reflected Brownian
motions, defined for $x\in\ol D$ by
%
%e1.3 #&#
%
\begin{equation}
\label{eqj133} X^x_t = x + B_t + \int
_0^t \n\bigl(X^x_s
\bigr) \,dL^x_s\qquad \mbox{for } t\geq0.
\end{equation}
Here $L^x$ is the local time of $X^x$ on $\prt D$. Equation
\eqref{eqj133} have unique pathwise solutions $(X^x,L^x)$ for
all $x$ simultaneously because the construction of the solution
to the Skorokhod equation given in \cite{LS} is deterministic.
Let $|A|$ denote the Lebesgue measure of a set $A$
and $\bQ_{r,t}(A) = |\{x\in D\dvtx X^x_t \in A\}| $.
We note that $\bQ_{r,t}$ is a random measure. For the definitions
of a random measure and weak convergence of random measures, see,
for example, \cite{Daw}; we will not review these notions here as
they are not used in the core of our paper.

%
%co1.6 #&#
\begin{conjecture}\label{prj182}
The measures $\bQ_{r,t} $ converge to a random measure $\bQ_{r}$
on $\torus_1\setminus\B(0, r)$ when $ t \to\infty$, in the sense of
weak convergence of random
measures. Random measures $\bQ_{r} $ converge weakly to the
uniform measure on $\torus_1$ when $r \to0$, in probability.
\end{conjecture}

In the context of \eqref{eqj133}, the earthworm picture is
obtained by interpreting
$\B(0, r)-B_t$
as a Brownian earthworm and $X^x_t - B_t$ as the location of a
displaced soil
particle.

For an extensive review of related results, see \cite{B3}.
Some of those
results will be recalled in Section~2.4.
The present article is,
philosophically speaking, a mirror image of \cite{BCJ}. That
article analyzed domains where $\dist(X_t,Y_t)$ converged to~0,
while the present article analyzes domains where the opposite
is true. It was proved in \cite{CLJ1,CLJ2} that, under mild
technical assumptions on the domain, reflected Brownian motions
$X$ and $Y$ do not coalesce in a finite time. A series of
papers by Pilipenko~\cite{P1,P2,P3} discuss stochastic flows of
reflected processes. The article \cite{P4} is posted on Math
ArXiv; it is a review and discussion of Pilipenko's previously
published results.

We will now outline the idea of the proof of our main result, Theorem
\ref{eqj135}. When the distance between the two solutions to the
Skorokhod problem $X$ and $Y$ is small, it changes in two distinct
ways. It increases at a rate proportional to the local time spent by
the processes on $\prt D$, due to the fact that $\prt D$ is curved and,
therefore, the directions in which $X$ and $Y$ are pushed are slightly
different. The distance between the two processes has negative jumps at
the ends of excursions of $X$ and $Y$ from $\prt D$ because the
difference between the two processes is not (approximately) parallel to
$\prt D$ at the ends of excursions; hence the local time push has a
different effect on the two trajectories. A discrete version of these
ideas is expressed in a formal way in \eqref{defvr} below. The origin
of these ideas goes back at least to the paper by Airault \cite{Ai}.
The continuous rate of increase of the distance between $X$ and $Y$ is
greater than the combined effect of negative jumps over long periods of
time, on average, for the domain $D$---this is the main estimate of
this paper, derived in Section~\ref{secexponent}. The main body of the
paper is devoted to detailed arguments showing that all modes of
behavior of the two processes not captured by the above description but
theoretically possible (such as coupling of the two processes at a
finite time) have negligibly small probability.

The rest of the paper is organized as follows. Section~\ref{secprelim} is a review of known results needed in this
paper, including a review of excursion theory in Section~\ref{secexc}, some technical estimates from \cite{BL,B3} in
Section~\ref{secdiffrbm} and preliminary analysis of the coupling. The
paper is based in an essential way on the exact and
explicit evaluation of the Lyapunov
exponent $\lambda_\rho$. The calculation is presented in Section~\ref{secexponent}.
The proof of Theorem \ref{eqj135}
is given in Section~\ref{secrec}; it consists of several
lemmas.

%s2 #&#
\section{Preliminaries}
\label{secprelim}

%s2.1 #&#
\subsection{General}\label{secj16}

For a process $Z$, a set $A$ and a point $a$ in the state space of $Z$,
let $T^Z_A =
\inf\{t\geq0\dvtx Z_t \in A\}$, $T^Z_a = \inf\{t\geq0\dvtx Z_t =a \}$
and $\tau^Z_A = \inf\{t\geq0\dvtx Z_t \notin A\}$.
By the Brownian scaling, if $\{X_t; t\geq0\}$
is the reflecting Brownian motion
on $\torus_1\setminus\overline{\B(0, r)}$ driven by Brownian motion $B_t$,
then $\{ r^{-1}X_{r^2t}; t\geq0\}$ is the reflecting
Brownian motion
on $(r^{-1} \torus_1)\setminus\overline{\B(0, 1)}$
driven by Brownian motion $r^{-1}B_{r^2t}$.
For notational convenience, throughout the remaining
part of this paper, we fix $\rho=1/r>1$ and take $\torus_\rho$ to be
the flat $3$-dimensional torus with side length
$2\rho>2$, that is, $\torus_\rho$ is the cube $\{(x_1,x_2,x_3) \in
\R^3\dvtx |x_k| \leq\rho, k=1,2,3\}$,
with the opposites sides
identified in the usual way, and let $D = \torus_\rho\setminus\ol
{\B(0,1)}$.

%s2.2 #&#
\subsection{Linear structure in torus}

In Section~\ref{secintro}, we used notation normally reserved
for elements of linear spaces, such as vector sum (e.g., $X_s -
Y_s$) and norm (e.g., $|X_t - Y_t|$). We will now make this
convention precise. Note that the torus $\torus_\rho$ can be
represented as the quotient $(\R/ (2 \rho\Z) )^3$. For $x\in
\torus_\rho$, let $A_x$ denote the set of all points in $\R^3$ which
correspond to $x$. For $x,y \in\torus_\rho$, we choose $x_1 \in A_x$
and $y_1\in A_y$ with the minimal distance $|x_1-y_1|$ among
all such pairs. Then we let $x-y = x_1 -y_1$ and $\dist(x,y)= |x-y| = |x_1
-y_1|$.

%s2.3 #&#
\subsection{Review of excursion theory}\label{secexc}

This section contains a brief review of excursion theory needed
in this paper. See, for example, \cite{M} for the foundations of the
theory in the abstract setting and \cite{B2} for the special
case of excursions of Brownian motion. Although Burdzy \cite{B2} does
not discuss reflected Brownian motion, all results we need from
his book readily apply in the present context. We will use two
different, but closely related, ``exit systems.'' The first one,
presented below, is a simple exit system representing
excursions of a single reflected Brownian motion from $\prt D$.
The second exit system encodes the information about both processes
$X$ and $Y$, but it is essentially equivalent to the first exit system.
We will introduce and use the second exit system in step 2.3 of the
proof of Lemma \ref{lemn81}.
Our review applies to general
domains $D$ with smooth boundaries, but we will assume that $D$
is the torus with the unit ball removed,
as in Section~\ref{secj16}.

Let $\P^{x_0}$ denote the distribution of the process $X$
defined by \eqref{eqj131}, and let $\E^{x_0}$ be the
corresponding expectation. Let $\P^x_D$ denote the distribution
of Brownian motion starting from $x\in D$ and killed upon
exiting $D$.

An ``exit system'' for excursions of the reflected Brownian
motion $X$ from $\prt D$ is a pair $(L^*_t, H^x)$ consisting of
a positive continuous additive functional $L^*_t$ of $X$
and a family
of ``excursion laws'' $\{H^x\}_{x\in\prt D}$.
Let $\bDelta$ denote the ``cemetery''
point outside $\ol D$, and let $\C$ be the space of all
functions $f\dvtx [0,\infty) \to\ol D\cup\{\bDelta\}$ which are
continuous and take values in $\ol D$ on some interval
$[0,\zeta)$, and are equal to $\bDelta$ on $[\zeta,\infty)$.
For $x\in\prt D$, the excursion law $H^x$ is a $\sigma$-finite
(positive) measure on $\C$, such that the canonical process is
strong Markov on $(t_0,\infty)$, for every $t_0>0$, with the
transition probabilities
$\P^{^{\centerdot}}_D$.
Moreover, $H^x$ gives zero
mass to paths which do not start from $x$. We will be concerned
only with the ``standard'' excursion laws; see Definition 3.2
of \cite{B2}. For every $x\in\prt D$ there exists a unique
standard excursion law $H^x$ in $D$, up to a multiplicative
constant.

Excursions of $X$ from $\prt D$ will be denoted $e$ or $e_s$,
that is, if $s< u$, $X_s,X_u\in\prt D$, and $X_t \notin\prt D$
for $t\in(s,u)$, then $e_s = \{e_s(t) = X_{t+s},
t\in[0,u-s)\}$ and $\zeta(e_s) = u -s$. By convention, $e_s(t)
= \bDelta$ for $t\geq\zeta(e_s)$, so $e_t \equiv\bDelta$ if
$\inf\{s> t\dvtx X_s \in\prt D\} = t$.

Let $\sigma_t = \inf\{s\geq0\dvtx L^*_s \geq t\}$ and
$\EE_u = \{e_s\dvtx s < \sigma_u\}$.
Let $I$ be
the set of left endpoints of all connected components of $(0,
\infty)\setminus\{t\geq0\dvtx X_t\in\partial D\}$. The following
is a special case of the exit system formula of \cite{M}. For
every $x\in\ol D$, every bounded predictable process $V_t$ and
every universally measurable function $f\dvtx \C\to[0,\infty)$
that vanishes on
excursions $e_t$ identically equal to $\bDelta$, we have
%
%e2.1 #&#
%
\begin{eqnarray}
\label{exitsyst} \E^x \biggl[ \sum_{t\in I}
V_t \cdot f ( e_t) \biggr] &=& \E^x \int
_0^\infty V_{\sigma_s} H^{X(\sigma_s)}(f)
\,ds
\nonumber
\\[-8pt]
\\[-8pt]
\nonumber
 &=& \E^x \int_0^\infty
V_t H^{X_t}(f) \,dL^*_t.
\end{eqnarray}
Here and
elsewhere $H^x(f) = \int_\C f \,dH^x$. Intuitively speaking,
\eqref{exitsyst} says that the right continuous version
$\EE_{t+}$ of the process of excursions is a Poisson
point process on the local time scale with variable intensity
$H^{^{\centerdot}}(f)$.

The normalization of the exit system is somewhat arbitrary. For
example, if $(L^*_t, H^x)$ is an exit system, and
$c\in(0,\infty)$ is a constant, then $(cL^*_t, (1/c)H^x)$ is
also an exit system. One can even make $c$ dependent on
$x\in\prt D$. Theorem 7.2 of \cite{B2} shows how to choose a
``canonical'' exit system; that theorem is stated for the usual
planar Brownian motion, but it is easy to check that both the
statement and the proof apply to the reflected Brownian motion.
According to that result, we can take $L^*_t$ to be the
continuous additive functional whose Revuz measure is a
constant multiple of the surface area measure $dx$
on $\prt D$ and
$H^x$'s to be standard excursion laws normalized so that
%
%e2.2 #&#
%
\begin{equation}
\label{eqM52} H^x (A) = \lim_{\delta\downarrow0} \frac1 \delta
\P_D^{x + \delta\n(x)} (A)
\end{equation}
for any event $A$ in a $\sigma$-field generated by the process
on an interval $[t_0,\infty)$, for any $t_0>0$. The Revuz
measure of $L^X$ is the measure $dx/(2|D|)$ on $\prt D$, that is,
if the initial distribution of $X$ is the uniform probability
measure $\mu$ on $D$,
then $\E^\mu\int_0^1 \bone_A (X_s) \,dL^X_s
= \int_A dx/(2|D|)$ for any Borel set $A\subset\prt D$. It has
been shown in \cite{BCJ} that $L^*_t=L^X_t$.

%s2.4 #&#
\subsection{Differentiability of stochastic flow of reflected
Brownian motions}\label{secdiffrbm}

It was proved in \cite{A,B3,P4}, in somewhat different
settings, that the stochastic flow of reflected Brownian
motions is differentiable in the initial condition. We will use
this result, and we will also need a key estimate from \cite{B3}
that was partly developed in~\cite{BL}. First, we will recall
some notation from \cite{B3}. The notation may seem somewhat
awkward in the present context because it was developed for
complicated arguments. We leave most of this notation unchanged
to help the reader consult the results in \cite{B3}.

We consider $ \prt D$ to be a smooth, properly embedded,
orientable hypersurface (i.e., submanifold of codimension $1$)
in $\R^3$, endowed with a smooth unit normal inward vector
field $\n$. We consider $\prt D$ as a Riemannian manifold with
the induced metric. We use the notation $\<\cdot,\cdot\>$ for
both the Euclidean inner product on $\R^3$ and its restriction
to the tangent space $\tngt_x\prt D$ for any $x\in\prt D$,
and $|\cdot|$ for the associated norm. For any $x\in\prt D$,
let $\pi_x\colon\R^{3}\to\tngt_x \prt D$ denote the
orthogonal projection onto the tangent space $\tngt_x \prt D$,
so $\pi_x \bz= \bz- \<\bz,\n(x)\>\n(x)$, and let $\sh
(x)\colon\tngt_x\prt D\to\tngt_x\,\prt D$ denote the shape
operator (also known as the Weingarten map), which is the
symmetric linear endomorphism of $\tngt_x\prt D$ associated
with the second fundamental form. It is characterized by $\sh
(x) \bv= - \partial_\bv\n(x)$ for $ \bv\in\tngt_x\prt
D$, where $\partial_\bv$ denotes the ordinary Euclidean
directional derivative in the direction of~$\bv$.

Recall that $\bDelta$ is an extra ``cemetery point'' outside
$\ol D$, so that we can send processes killed at a finite time
to $\bDelta$. For $s\geq0$ such that $X_s \in\prt D$ we let
$\zeta(e_s) = \inf\{t>0\dvtx X_{s+t} \in\prt D\}$. Here $e_s$ is
an excursion starting at time $s$, that is, $e_s = \{e_s(t) =
X_{t+s}, t\in[0,\zeta(e_s))\}$. We let $e_s(t) = \bDelta$
for $t\geq\zeta(e_s)$, so $e_t \equiv\bDelta$ if $\zeta(e_s)
=0$.

Let $\sigma^X_t$ be the inverse of local time $L^X_t$, that is,
$\sigma^X_t = \inf\{s \geq0\dvtx L^X_s \geq t\}$, and $\EE_b =
\{e_s\dvtx s < \sigma^X_b\}$. For $b,\eps>0$, let $\{e_{u_1},
e_{u_2}, \ldots, e_{u_m}\}$ be the set of all excursions $e\in
\EE_b$ with
$|e(0) -e(\zeta-)| \geq\eps$.
We assume that
excursions are labeled so that $u_k < u_{k+1}$ for all $k$, and
we let $\ell_k = L^X_{u_k}$ for $k=1,\ldots, m$. We also let
$u_0 =\inf\{t\geq0\dvtx X_t \in\prt D\}$, $\ell_0 =0 $,
$\ell_{m+1} = b$ and $\Delta\ell_k = \ell_{k+1} - \ell_k$. Let
$x_k =
e_{u_k} (\zeta-)$
be the right endpoint of excursion
$e_{u_k}$ for $k=1,\ldots, m$ and $x_0=X_{u_0}$.

For $\bv_0\in\R^3$, let
%
%e2.3 #&#
%
\begin{equation}
\label{defvr} \qquad\bv_b = \exp\bigl(\Delta\ell_m
\sh(x_m)\bigr) \pi_{x_m} \cdots \exp\bigl(\Delta
\ell_1 \sh(x_1)\bigr) \pi_{x_1} \exp\bigl(
\Delta\ell_0 \sh(x_0)\bigr) \pi_{x_0}
\bv_0.
\end{equation}
Note that all concepts based on excursions $e_{u_k}$ depend
implicitly on $\eps>0$, which is often suppressed in the
notation. Let $\A_b^\eps$ denote the linear mapping $\bv_0 \to
\bv_b$.

It was proved in Theorem 3.2 in \cite{BL} that for every $b>0$,
a.s., the limit $\A_b:= \lim_{\eps\to0} \A_b^\eps$ exists and
it is a linear mapping of rank $2$. For any $\bv_0$, with
probability~1, $\A_b^\eps\bv_0\to\A_b \bv_0$ as $\eps\to0$,
uniformly in $b$ on compact sets.

Recall the stochastic flow $X^x_t$ of reflected Brownian
motions defined in \eqref{eqj133}. By Theorem 3.1 of
\cite{B3}, for every $x\in D$, $b>0$ and compact set $K \subset
\R^3$, we have a.s.,
%
%e2.4 #&#
%
\begin{equation}
\label{n61} \lim_{\eps\to0} \sup_{\bv\in K} \bigl
\llvert \bigl(X^{x +
\eps\bv} _{\sigma^x_b} - X^{x}_{\sigma^x_b}
\bigr)/\eps- \A_{b} \bv \bigr\rrvert =0,
\end{equation}
where $\sigma^x_b = \inf\{t\geq0\dvtx L^{X^x}_t \geq b\}$.
Informally speaking, the last formula says that
$y \to X^y_{\sigma^x_b}$ is differentiable, that is,
the stochastic flow $X$ is differentiable in the space variable.
Formula \eqref{defvr} represents
a discrete approximation to the derivative~$\A_b$.
According to that formula, the approximation to the derivative is a
composition of two types of linear mappings.
After the $k$-th excursion, the projection on the tangent plane to
$\prt D$
at the endpoint of the $k$th excursion is added to the composition.
Between excursions, the
derivative expands or contracts (in the sense of the exponential function
of a linear mapping) at the rate proportional to the
curvature of $\prt D$ at the point where the most recent excursion
ended.

Consider some $b>0$, and let $\sigma_* = \inf\{t\geq0\dvtx L^X_t
\lor L^Y_t \geq b\}$. Thus defined $\sigma_*$ is different from
the random variable denoted by the same symbol
in \cite{B3}.
Article \cite{B3} is concerned with a stochastic flow,
and $\sigma_*$ denotes in that paper, roughly speaking, the time when
at least one of the local times corresponding to reflected
Brownian motions in the flow exceeds a certain level. The
results and arguments given in \cite{B3}
can be applied in our paper with our definition of $\sigma_*$
because we are concerned only with two reflected Brownian
motions $X$ and $Y$.

For $\eps_* >0$, let
%
%e2.5 #&#
%
\begin{equation}
\label{eqd111} \{e_{t^*_1}, e_{t^*_2}, \ldots, e_{t^*_{m^*}}
\} =\bigl\{e_t\in\EE_b\dvtx \bigl|e_t(0) -
e_t(\zeta-)\bigr| \geq\eps_*, t < \sigma_*\bigr\}.
\end{equation}
These excursions are labeled so that $t^*_k < t^*_{k+1}$ for all
$k$. We
let $\ell^*_k = L^X_{t^*_k}$ for $k=1,\ldots, m^*$.
We also let $t^*_0 = \inf\{t\geq0\dvtx X_t \in\prt D\}$,
$\ell^*_0 = 0 $, $\ell^*_{m^*+1} = L^X_{\sigma_*}$ and $\Delta
\ell^*_k = \ell^*_{k+1} - \ell^*_k$. Let $x^*_k =
e_{t^*_k}(\zeta-)$ for $k=1,\ldots, m^*$, and $x^*_0=X_{t^*_0}$.
Let
\[
\I_k = \exp\bigl(\Delta\ell^*_k \sh\bigl(x^*_k
\bigr)\bigr) \pi_{x^*_k}.
\]

The arguments in \cite{B3} were given only for $b=1$, but it is
easy to see that they apply equally to any fixed value of
$b>0$.

Let $\P^{x_0, y_0}$ denote the distribution of the solution
$(X,Y)$ to \eqref{eqj131}--\eqref{eqj132}, and let $\E^{x_0,
y_0}$ denote the corresponding expectation.

Fix an arbitrarily small $c_3>0$. By (3.161) and (3.167) of
\cite{B3}, there exist $c_4,c_5,c_6,\eps_0>0$, $\beta_1 \in
(1,4/3)$ and $\beta_2 \in(0, 4/3 - \beta_1)$ such that if
$X_0=x$, $Y_0=y$, $|x-y| =\eps< \eps_0$ and $\eps_* = c_4
\eps$, then
%
%e2.6 #&#
%
\begin{equation}
\label{EM311}\bigl |(Y_{\sigma_*} - X_{\sigma_*}) - \I_{m^*} \circ
\cdots\circ\I_0 (Y_{0} - X_{0})\bigr| \leq|
\Lambda| + \Xi,
\end{equation}
where $|\Lambda| < c_3 \eps$, $\P^{x,y}$-a.s., and
%
%e2.7 #&#
%
\begin{equation}
\label{M316} \P^{x,y}\bigl(|\Xi| > c_5 \eps^{\beta_1}
\bigr) \leq c_6 \eps^{\beta_2}.
\end{equation}
The meaning of $\Lambda$ and $\Xi$ is not important in the present
paper. These random variables arise in the decomposition of the
difference on the left-hand side of \eqref{EM311}. The random variable
$\Lambda$ is ``large'' because it is bounded by a constant multiple
of~$\eps$ to power 1; on the positive side, this bound is
deterministic. The random
variable $\Xi$ is ``small'' because it is (typically) smaller than
$\eps
^{\beta_1}$
with $\beta_1>1$, but this bound does not hold with
probability 1.

%s2.5 #&#
\subsection{Some path properties of couplings}
If no confusion may
arise, $x_0$ and $y_0$ will be suppressed in the notation $\P^{x_0, y_0}$,
$\E^{x_0, y_0}$ and $\P^{x_0, y_0}$-a.s.,
and we will use the notation ``$\P$,'' ``$\E$'' and ``a.s.''

The next lemma says that if the two processes $X$ and $Y$ are close to
each other and almost parallel to $\prt D$ then they will stay almost
parallel to $\prt D$ as long as they do not move far away from the
current position.
The proof is based on an idea that will be used several times in this
article; see steps 2.1, 2.2, 2.4 and 2.6 of the proof of Lemma \ref
{lemn81}. The argument is concerned with an interval where only one
of the processes can have some local time push. The analysis of the
relative positions of the two processes at the beginning and the end of
the interval, and the direction of the local time push, leads to a
(desired) contradiction. The idea is graphically illustrated in Figure~\ref{figearth1} below (step 2.2 of the proof of Lemma \ref{lemn81})
because that implementation yields the most convincing picture.

%
%le2.1 #&#
\begin{lemma}\label{lemn31}
Suppose that $x_1\in\prt D$, $c_1 \in(0,1/100)$, and let $D_1 =\break
\ol{D \cap\B(x_1, c_1/4)}$. Assume that $x_0, y_0 \in D_1$ and
$|\<x_0-y_0, \bn(x_1)\>| \leq c_1 |x_0-y_0| $. Let $T_1 =
\tau^X_{D_1} \land\tau^Y_{D_1}$. Suppose that $X$ and $Y$
solve \eqref{eqj131}--\eqref{eqj132} with $X_0 = x_0$ and
$Y_0=y_0$. Then a.s.,
$|\< X_t - Y_t, \bn(x_1)\>| \leq c_1 |X_t -
Y_t|$ for all $t\leq T_1$.
\end{lemma}

\begin{pf}
Observe that for
$x_2\in\prt D \cap D_1$ and $y_2 \in D_1$ we have
$\<x_2-y_2, \bn(x_1)\> \leq c_1 |x_2-y_2|/2 $.
Moreover, for
any $x_3\in
\prt D \cap D_1$, the angle between $\bn(x_1)$ and $\bn(x_3)$
is less than $c_1/2$ radians.

Assume that $|\< X_t - Y_t, \bn(x_1)\>| > c_1 |X_t - Y_t|$ for
some $t\leq T_1$. We will show that this assumption leads to a
contradiction. Let
\[
T_2 = \inf\bigl\{t\geq0\dvtx \bigl|\bigl\< X_t - Y_t,
\bn(x_1)\bigr\>\bigr| > c_1 |X_t - Y_t|
\bigr\}.
\]
By assumption and the pathwise uniqueness of solutions to
\eqref{eqj131}--\eqref{eqj132}, $T_2<T_1$ and $|X_{T_2}-Y_{T_2}|>0$.
We have $|\< X_{T_2} - Y_{T_2}, \bn(x_1)\>| = c_1 |X_{T_2} - Y_{T_2}|$
so at most one of the
points $X_{T_2}$ and $Y_{T_2}$ belongs to the boundary of $D$.
At least one of these points belongs to $\prt D$ because $t\to
|\< X_t - Y_t, \bn(x_1)\>| / |X_t - Y_t|$ is constant over
intervals where neither $X$ nor $Y$ visit $\prt D$. Suppose
without loss of generality that $X_{T_2}\in\prt D$. Then, by
the opening remarks, $\< X_{T_2} - Y_{T_2}, \bn(x_1)\> \leq
c_1 |X_{T_2} - Y_{T_2}|/2$, and therefore,
%
%e2.8 #&#
%
\begin{equation}
\label{en31} T_2= \inf\bigl\{t\geq0\dvtx \bigl\< X_t -
Y_t, \bn(x_1)\bigr\><- c_1 |X_t -
Y_t|\bigr\}.
\end{equation}
In particular,
$\< X_{T_2} - Y_{T_2}, \bn(x_1)\> = -c_1 |X_{T_2} - Y_{T_2}|$.
Let
\[
T_3= \inf\{s>T_2\dvtx Y_s \in\partial D\}
\wedge T_1.
\]
Then $T_2<T_3$ and $L^Y_{T_3} =L^Y_{T_2}$.
Hence, for $t\in[T_2, T_3]$,
we have
\begin{eqnarray*}
\bigl\<X_t-Y_t, \bn(x_1)\bigr\> &=&\biggl \<
X_{T_2} - Y_{T_2}+ \int^t_{T_2}
\bn(X_s) \,dL^X_s, \bn(x_1)
\biggr\>
\\
&\geq& -c_1 |X_{T_2} - Y_{T_2}|+c_1
\bigl(L^X_t-L^X_{T_2}\bigr)
\\
&\geq& -c_1 |X_{T_2} - Y_{T_2}|+c_1
\biggl|\int_{T_2}^t \bn(X_s)
\,dL^X_s \biggr|
\\
&\geq& -c_1 \biggl| X_{T_2}-Y_{T_2} +\int
_{T_2}^t \bn(X_s)
\,dL^X_s \biggr|
\\
& = & -c_1 |X_t-Y_t|,
\end{eqnarray*}
contradicting the definition of
$T_2$ in view of \eqref{en31}.
This completes the proof of the lemma.
\end{pf}

%
%le2.2 #&#
\begin{lemma}\label{lemo311}
$\!\!$If $x,y \in\ol D$ and $x\ne y$, then \mbox{$\P^{x,y}(X_t \ne Y_t,
\mbox{ for every } t\geq0) = 1$}.
\end{lemma}

\begin{pf}
The proof of the lemma consists of two main steps. The first step uses
a result on differentiability of the stochastic flow of reflected
Brownian motions. According to this result, under some assumptions, the
derivative of the stochastic flow is a nontrivial linear mapping.
Hence, different trajectories in the stochastic flow do not collide.
This argument applies directly only when the starting points of $X$ and
$Y$ are ``almost parallel'' to $\prt D$. The general case, presented in
step 2 below, is dealt with by reducing it to the first case at an
appropriate stopping time.

Assume that for some distinct $x,y \in\ol D$, $X_t = Y_t$ for
some $t<\infty$, with positive probability. A standard
application of the strong Markov property shows that there must exist
$r\in(0,1/200)$, $x_1 \in\prt D$ and $y_1 \in\ol D$ such
that if we write $D_1 = \ol D\cap\B(x_1, r/8)$ and $T_1 =
\tau^X_{D_1} \land\tau^Y_{D_1}$, then $\P^{x_1,y_1} (\exists
t\in[0,T_1]\dvtx X_t = Y_t)>0$. Note that necessarily $y_1 \in D_1$.

\textit{Step} 1.
Suppose that
$r\in(0,1/100)$, $x_1 \in\prt D$, $y_1 \in D_1$ and $x_1\ne y_1$.
In this step, we will consider the case when
$|\<x_1-y_1, \bn(x_1)\>| \leq(r/2) |x_1-y_1|$.

Let $K_{\delta} = (x_1 + \T_{x_1} \prt D) \cap\prt
\B({x_1}, \delta)$ and $K^0_{\delta} = \T_{x_1} \prt D \cap\prt
\B(0, \delta)$. Recall the stochastic flow $X^x_t$ of
reflected Brownian motions defined in \eqref{eqj133}, and note
that $(X_t,Y_t) = (X^{x_1}_t, X^{y_1}_t)$ under $\P^{x_1,y_1}$.
Let $\wh\sigma_b = \inf\{t\geq0\dvtx L^{X^{x_1}}_t \geq b\}$.
According to Theorem 3.2 of \cite{BL} and its proof, for any
fixed $b>0$, $\A_b$ has rank 2. In fact, the proof shows more
than that, namely, $\P^{x_1}$-a.s., $\inf_{\bv\in K^0_{\delta}}
|\A_b(\bv)|>0$. This and \eqref{n61} imply that for any $b>0$,
\[
\lim_{\delta\to0} \P^{x_1} \Bigl(\inf_{\bv\in K^0_\delta}
\bigl\llvert X^{{x_1} + \bv} _{\wh\sigma_b} - X^{{x_1}}_{\wh\sigma_b}
\bigr\rrvert /|\bv| > 0 \Bigr) =1.
\]
Since the stochastic differential equation \eqref{eqj131} has
a unique strong solution, if $X^x_t = X^y_t$ for some $t$, then
$X^x_s = X^y_s$ for all $s\geq t$, a.s. Hence, the last formula
can be strengthened as follows:
\[
\lim_{\delta\to0} \P^{x_1} \Bigl(\inf_{\bv\in K^0_\delta}
\inf_{0\leq t \leq\wh\sigma_b} \bigl\llvert X^{{x_1} + \bv} _{t} -
X^{{x_1}}_{t} \bigr\rrvert /|\bv| > 0 \Bigr) =1.
\]
For every $k\geq1$ find $\delta_k>0$ such that
%
%e2.9 #&#
%
\begin{equation}
\label{eqd107} \P^{x_1} \Bigl(\inf_{\bv\in K^0_{\delta_k}} \inf
_{0\leq t \leq\wh\sigma_b} \bigl\llvert X^{{x_1} + \bv} _{t} -
X^{{x_1}}_{t} \bigr\rrvert /|\bv| > 0 \Bigr)
\geq1-2^{-k}.
\end{equation}

It follows from Lemmas 3.3 and 3.4 of \cite{B3} and their
proofs that there exist stopping times $S_k$ such that $S_k\to
\infty$ as $k\to\infty$, and $|X^x_{t} - X^y_{t}| \leq k
|X^x_0 - X^y_0|$ for al $x, y \in\ol D$ and $t\in[0, S_k]$,
a.s. We can assume without loss of generality that $\delta_k\to
0$ as $k\to\infty$. We make $\delta_k>0$ smaller, if
necessary, so that $|X^{x_1}_{t} - X^z_{t}| \leq r/8$, for all $k\geq1$,
$z\in K_{\delta_k}$ and $t\in[0, S_k]$, a.s. By passing to a
subsequence, if necessary, we may assume that
%
%e2.10 #&#
%
\begin{equation}
\label{eqd106} \P(S_k > \wh\sigma_b) \geq1 -
2^{-k}.
\end{equation}
If we let $T_2 = T_1 \land\wh\sigma_b$,
\begin{eqnarray*}
F^1_k &=& \bigl\{\bigl|X^{x_1}_{t} -
X^z_{t}\bigr| \leq r/8, \forall z\in K_{\delta_k}, t
\in[0, T_2]\bigr\},
\\
F^2_k &=& \Bigl\{\inf_{\bv\in K^0_{\delta_k}} \inf
_{0\leq t \leq T_2} \bigl\llvert X^{{x_1} + \bv} _{t} -
X^{{x_1}}_{t} \bigr\rrvert /|\bv| > 0 \Bigr\},
\\
F_k &=& F^1_k \cap F^2_k,
\end{eqnarray*}
then, by \eqref{eqd107} and \eqref{eqd106}, $\P(F_k) \geq
1- 2^{-k+1}$.\eject

We will argue that if $F_k$ holds, then for all
$t\in[0, T_2]$ and $z\in K_{\delta_k}$,
%
%e2.11 #&#
%e2.12 #&#
%e2.13 #&#
%
\begin{eqnarray}
\label{j152} \bigl|\bigl\<X_{t} - Y_{t}, \bn(x_1)\bigr\>\bigr|
&\leq&(r/2) |X_{t} - Y_{t}|,
\\
\label{j153} \bigl|\bigl\<X_{t} - X^z_{t},
\bn(x_1)\bigr\>\bigr| &\leq& r \bigl|X_{t} - X^z_{t}\bigr|,
\\
\label{j154} \bigl|\bigl\<X^z_{t} - Y_{t},
\bn(x_1)\bigr\>\bigr| &\leq& r \bigl|X^z_{t} -
Y_{t}\bigr|.
\end{eqnarray}
We obtain \eqref{j152} from our assumption
that $|\<x_1-y_1, \bn(x_1)\>| \leq(r/2) |x_1-y_1|$ and
Lemma \ref{lemn31}.
If
$F^1_k$ holds, then $X^z_t \in\B(x_1, r/4)$ for all $t\in[0,
T_2]$ and $z\in K_{\delta_k}$. Hence, \eqref{j153} follows from
Lemma \ref{lemn31} applied with $c_1 = r$. The claim holds
for all $z\in K_{\delta_k}$ simultaneously because Lemma
\ref{lemn31} is deterministic.
We can make $\delta_k>0$ smaller, if necessary, so that
$|\<z-y_1, \bn(x_1)\>| \leq r |z-y_1|$ for all $k$ and all $z\in
K_{\delta_k}$. Once again, we apply
Lemma \ref{lemn31} with $c_1 = r$ and conclude that \eqref{j154}
holds true.

Estimates \eqref{j152}--\eqref{j154} have the following topological
consequences. Recall that $\pi_{x_1} \bz$ denotes the
projection of $\bz$ on $\T_{x_1} \prt D$. Assuming that $F_k$
holds and $t\leq T_2$, the set $\Gamma_t=\pi_{x_1}\{X^x_t,
x\in K_{\delta_k}\}$ is a closed loop that contains
$\pi_{x_1}X_t$ inside. When $t$ goes from $0$ to $T_2$, $\pi_{x_1}X_t$,
$\pi_{x_1}Y_t$ and $\Gamma_t$ evolve
continuously. If $X_t = Y_t$ for some $t\leq T_2$, then we must have
$\pi_{x_1}Y_s = \pi_{x_1}X^x_s$ for some $k\geq1$, $x\in K_{\delta_k}$
and $0\leq s\leq t$. This and \eqref{j154} imply that $Y_s = X^x_s$.
Hence, $X_t = Y_t = X^x_t$. But this
means that $F^2_k$ does not hold. Since $\P(F_k) \geq1-
2^{-k+1}$, we conclude that the probability that there exists
$t \in[0, T_2]$ such that $X_t = Y_t $ is less than
$2^{-k+1}$. Since $k$ and $b$ are arbitrarily large,
$\P^{x_1,y_1} (\exists t\in[0,T_1]\dvtx X_t = Y_t)=0$.

\textit{Step} 2.
Suppose that
$r\in(0,1/200)$, $x_1 \in\prt D$, $y_1 \in D_1$ and $x_1\ne y_1$.
In this step, we no longer assume that
$|\<x_1-y_1, \bn(x_1)\>| \leq(r/2) |x_1-y_1|$.
Also, note that we assume that $r\in(0,1/200)$ while in step 1 we
assumed that $r\in(0,1/100)$.

Suppose that $ \P^{x_1,y_1} ( \exists
t\in[0, T_1]\dvtx Y _{t} = X_{t}  ) = p_1
>0$. We will show that this assumption leads to a
contradiction. Let
\[
A =\bigl\{y\in\ol D\dvtx |x_1-y| = |x_1 -
y_1|, \bigl\<x_1-y, \bn(x_1)\bigr\> =
\bigl\<x_1-y_1, \bn(x_1)\bigr\>\bigr\}.
\]
The set $A$ is
a circle, possibly with a zero radius. If the radius of $A$ is
0, that is, if $A$ contains only $y_1$, then $x_1 - y_1$ is
parallel to $\bn(x_1)$. It is easy to see that for any $t_0>0$,
with probability
1, there exists time $t\in(0, t_0 \land T_1)$ such that $X_t \ne Y_t$, $X_t
\in\prt D$, $X_t - Y_t$ is not parallel to $\bn(X_t)$, and $t$ is the
terminal time of an excursion of $X$ from $\prt D$. Let
$U_r$ be the smallest such $t$ greater than $r>0$. We can
apply the strong Markov property at time $U_r$, for every rational time $r>0$,
and the result proved below for the case when $A$ does not
reduce to a single point to show that $X$ and $Y$ will not meet before $T_1$.

Hence, we will assume from now on that the set $A$ is a circle
with a nonzero radius.
Choose $n$
distinct points $y_1, \ldots, y_n$ in $A$, with $n> 2/p_1$. Let
$T_1^{y_j} =
\tau^X_{D_1} \land\tau^{X^{y_j}}_{D_1}$.
By our assumption and symmetry,
$ \P^{x_1,y_j} ( \exists
t\in[0, T_1^{y_j}]\dvtx X_{t} = X^{y_j}_{t}  ) = p_1$. It follows that
for some $j\ne k$,
\[
\P \bigl( \exists t\in\bigl[0, T_1^{y_j}\bigr]\dvtx
X_{t} = X^{y_j}_{t}, \mbox{ and } \exists s\in
\bigl[0, T_1^{y_k}\bigr]\dvtx X_{s} =
X^{y_k}_{s} \bigr) >0.
\]
If the event in the last formula holds, then for $u = s \lor t$ we have
$X^{y_j} _u = X^{y_k}_u$ and $u \leq\tau^X_{D_1} = \tau^{X^{y_j}}_{D_1}
= \tau^{X^{y_k}}_{D_1}$.
In other words, we have shown that if
$T_1^{y_j, y_k} =
\tau^{X^{y_j}}_{D_1} \land\tau^{X^{y_k}}_{D_1}$,
then
$ \P^{y_j,y_k} ( \exists
t\in[0, T_1^{y_j,y_k}]\dvtx X^{y_j} _{t} = X^{y_k}_{t}  ) >0$. We will
prove that this leads to a contradiction. If the processes $X^{y_j}$
and $X^{y_k}$ do not hit $\prt D$
before $T_1^{y_j,y_k}$, then of course they do not meet before
$T_1^{y_j,y_k}$. If
one of them hits $\prt D$ before time $T_1^{y_j,y_k}$, then we can
suppose without loss
of generality that $T_3:= T^{X^{y_j}}_{\prt D} \leq
T^{X^{y_k}}_{\prt D} \land T_1^{y_j,y_k}$. Then $|\<
X^{y_j}_{T_3}-X^{y_k}_{T_3},
\bn(X^{y_j}_{T_3})\>| \leq(r/4) |X^{y_j}_{T_3}-X^{y_k}_{T_3}| $.
Since $T_3 \leq T_1^{y_j,y_k}$, $\B(x_1, r/8) \in\B(X^{y_j}_{T_3},
r/4)$. Let $T_4 = \tau^{X^{y_j}}_{\B(X^{y_j}_{T_3}, r/4)} \land\tau
^{X^{y_k}}_{\B(X^{y_j}_{T_3}, r/4)} $. By step 1, applied with $2r$ in
place of $r$, and the strong Markov property applied at $T_3$,
\begin{eqnarray*}
\P^{y_j,y_k} \bigl( \exists t\in\bigl[0, T_1^{y_j,y_k}
\bigr]\dvtx X^{y_j} _{t} = X^{y_k}_{t}
\bigr) &\leq& \P^{y_j,y_k} \bigl( \exists t\in[0, T_4]\dvtx
X^{y_j} _{t} = X^{y_k}_{t} \bigr)
\\
&=& \P^{y_j,y_k} \bigl( \exists t\in[T_3, T_4]
\dvtx X^{y_j} _{t} = X^{y_k}_{t} \bigr)
=0.
\end{eqnarray*}
This contradicts our earlier assertion and finishes the proof.
\end{pf}

The next lemma is almost the same as a lemma that appeared in \cite
{BCJ}. It says that at the time when the local time reaches a fixed
level, the difference between the processes $X$ and $Y$ is very likely
to be ``almost parallel'' to $\prt D$.

%
%le2.3 #&#
\begin{lemma}\label{lemo302}
For any $b>0$ and $\beta_1 \in(0,1)$ there exist $c_0,\beta_2,
\eps_1>0$ such that if $\eps\leq\eps_1$, $x,y \in\ol D$ and
$|x - y| = \eps$, then
%
%e2.14 #&#
%
\begin{equation}
\label{eqd112} \P^{x, y} \biggl( \frac{ \llvert   \< Y_{\sigma^X_b} - X_{\sigma^X_b},
\n (X_{\sigma^X_b} ) \> \rrvert  } {
\llvert  Y_{\sigma^X_b} - X_{\sigma^X_b} \rrvert } \geq c_0
\eps^{\beta_1} \biggr) \leq\eps^{\beta_2}.
\end{equation}
\end{lemma}

\begin{pf}
The proof is similar to the proof of Lemma 4.6 in \cite{BCJ}, so
we only sketch the main ideas. The paper \cite{BCJ} is
concerned with 2-dimensional domains, but it is easy to see that
the results from that paper that we use here apply to
multidimensional domains.

By Lemma 4.1(ii) of \cite{BCJ}, $\P(L^Y_{\sigma^X_b} \geq a
)\leq c_1 e^{-c_2 a}$. Hence, for any $\beta_3>0$ and
$\beta_4>0$ depending on $\beta_3$,
\[
\P\bigl(L^Y_{\sigma^X_b} \geq\beta_3 |\log\eps|
\bigr) \leq c_1 \exp\bigl(-c_2 \beta_3 |\log\eps|\bigr)
= c_1 \eps^{\beta_4}.
\]
If the event $A_1:= \{L^Y_{\sigma^X_b} \leq\beta_3 |\log
\eps| \}$ holds, then by Lemma 3.8 of \cite{BCJ},
\[
\sup_{t\in[0,\sigma^X_b]}|X_{t}- Y_{t}|
\leq|X_0- Y_0| \exp\bigl( c_4 (1 +
\beta_3 |\log\eps|)\bigr) \leq c_5 \eps^{1-c_4\beta_3} =
c_5 \eps^{1-\beta_5},
\]
where $\beta_5 $ is defined as $c_4 \beta_3$. Choose
$\beta_3>0$ so small that $\beta_5 < \beta_1$, and we can find
$\beta_6$
such that $ \beta_1 < \beta_6 < 1-\beta_5$.

Let $T_1 = \inf\{t\geq0\dvtx X_t \in\prt D\}$ and $\{V_t, 0\leq t
\leq\sigma^X_b - T_1\}:= \{X_{\sigma^X_b-t}, 0\leq t \leq
\sigma^X_b - T_1\} $. If we condition on the values of
$X_{T_1}$ and $X_{\sigma_b^X}$, the process $V$ is a reflected
Brownian motion in $D$ starting from $X_{\sigma_b^X}$ and
conditioned to approach $X_{T_1}$ at its lifetime. It is easy
to see that $\P(|X_{T_1}- X_{\sigma^X_b}| \leq\eps^{\beta_1})
\leq c_6 \eps^{\beta_1}$.

Suppose that the event $A_2:=\{\dist(X_{T_1}, X_{\sigma^X_b})
\geq\eps^{\beta_1}\}$ holds. Conditionally on this event, the
probability that $V$ does not spend at least $\eps^{\beta_6}$ units of
local time on the boundary of $\prt D$ before leaving the ball
$\B(V_0, \eps^{\beta_1})$ is bounded by $ c_7 \eps^{\beta_6-
\beta_1}$. Let $A_3$ be the event that $V$ spends
$\eps^{\beta_6}$ or more units of local time on the boundary of
$\prt D$ before leaving the ball $\B(V_0, \eps^{\beta_1})$. Let
$T_2 = \sup\{t \leq\sigma^X_b\dvtx X_t \notin\B(V_0,
\eps^{\beta_1})\}$. If $A_1$ and $A_3$ hold, then $Y$ must hit $\prt D$
at some time $t\in[T_2, \sigma^X_b]$ because $\eps^{\beta_6}>
c_5 \eps^{1-\beta_5}$ for small $\eps$; that is, the amount of push
given to $X$ exceeds the maximum distance between the two
processes. We also have $X_{\sigma^X_b} \in\prt D$. The
maximum angle between normal vectors at points of $\prt D \cap
\B(V_0, \eps^{\beta_1})$ is less than $c_8 \eps^{\beta_1}$. A
modification of Lemma \ref{lemn31} shows that $|\<
X_{\sigma^X_b} - Y_{\sigma^X_b}, \bn(X_{\sigma^X_b})\>| \leq
|X_{\sigma^X_b} - Y_{\sigma^X_b}| c_9 \eps^{\beta_1}$. Recall
that, by Lemma \ref{lemo311}, $|X_{\sigma^X_b} -
Y_{\sigma^X_b}| >0$, a.s. We have shown that the complement of the
event in \eqref{eqd112}
occurs if $A_1\cap A_2 \cap A_3$ holds. Since $\P((A_1\cap
A_2 \cap A_3)^c) \leq c_1\eps^{\beta_4} + c_6 \eps^{\beta_1} + c_7
\eps^{\beta_6- \beta_1}$, the lemma follows.
\end{pf}

%s3 #&#
\section{The sign of the Lyapunov exponent}\label{secexponent}
This section is devoted to the calculation of the ``Lyapunov
exponent'' for the exterior of a three-dimensional ball. In our
model, the Lyapunov exponent is represented by $1+\lambda_\rho$
where $\lambda_\rho$ is defined in Theorem \ref{tho53}(ii).
This is a three-dimensional analogue of an exponent defined in
\cite{BCJ} for two-dimensional domains. The sign of this
exponent---positive for the domain $D$---has the fundamental
importance for this article.

Recall that $H^x$ is the excursion law for $X$ in $D$, and $\pi_x $
denotes the projection on the plane tangent to $\prt D$ at $x\in\prt
D$. For an
excursion $e$ and nonzero vector $\bv\in\R^3$, we let
$f_\bv(e) = \log|\pi_{e(\zeta-)} (\bv)| - \log|\bv| $. Note
that $f_\bv( e) \leq0$. Let $D_1 = \R^3 \setminus
\ol{\B(0,1)}$, and let $(\wh L_t, \wh H^x)$ be the exit system
for reflected Brownian motion $\wh X$ in $D_1$.

%
%th3.1 #&#
\begin{theorem}\label{tho53}
\textup{(i)} For every $x\in\prt D_1$ and $\bv\in\tngt_x \prt D_1$, $|\bv|>0$,
\[
\wh H^x\bigl( f_\bv(e)\bigr) = \sqrt{2} - 1 - \log(1+
\sqrt2).
\]

\textup{(ii)} Let $\lambda_\rho(x,\bv) = H^{x} (f_\bv( e))$. We have
uniformly in $x\in\prt D$ and $\bv\in\tngt_x \prt D$, $|\bv|>0$,
\[
\lim_{\rho\to\infty} \lambda_\rho(x,\bv) = \lim
_{\rho\to\infty} H^{x} \bigl(f_\bv( e)\bigr)=
\sqrt{2} +\log2 - 2 - \log(1+\sqrt2) \approx-0.774013.
\]
\end{theorem}

The actual value of the Lyapunov exponents comes from a
computation presented in the \hyperref[appendix]{Appendix}, which leads to the following lemma.

%
%le3.2 #&#
\begin{lemma}\label{lemj271}
We have
%
%e3.1 #&#
%
\begin{eqnarray}
\label{eqj276}  &&\int_0^{2\pi} \int
_0^\pi\frac{1}{16\pi} \frac{\sin\alpha}{\sin^3 (\alpha/2)} \log
\bigl(\sin^2 \beta+ \cos^2 \beta\cos^2 \alpha
\bigr)^{1/2} \,d\alpha \,d\beta
\nonumber
\\[-8pt]
\\[-8pt]
\nonumber
&&\qquad = \sqrt{2} - 1 - \log(1+\sqrt2)
\end{eqnarray}
and
%
%e3.2 #&#
%
\begin{equation}\label{eq061}
\int_0^{2\pi} \int_0^\pi
\frac{1}{4\pi} (\sin\alpha) \log\bigl(\sin^2 \beta+
\cos^2 \beta\cos^2 \alpha\bigr)^{1/2} \,d\alpha
\,d\beta
 = \log2 - 1.
\end{equation}
\end{lemma}

\begin{pf*}{Proof of Theorem \ref{tho53}}
(i)
We will derive a formula for the expectation of a random variable under
the excursion law from the well-known formula for the density of the
harmonic measure.

Let $\tau^X_A= \inf\{t\geq0\dvtx X_t \notin A\}$. Recall that
$\P^x_{D_1}$
denotes the distribution of Brownian motion
starting from $x$ and killed at the time $\tau^X_{D_1}$. Let
$\mu_r$ denote the uniform probability distribution on the
sphere $\B(0,r)$; we will abbreviate $\mu_1=\mu$. An explicit
formula for the harmonic measure in $D_1$ is given in
\cite{PS}, Theorem~3.1, page~102. That formula implies that
%
%e3.3 #&#
%
\begin{equation}
\label{j263} \P^x_{D_1} \bigl(X\bigl(\tau^X_{D_1}-
\bigr) \in dy\bigr) = a(x) |x-y|^{-3} \mu(dy)
\end{equation}
for $x\in D_1$ and $y\in\prt\B(0,1)$, where $a(x)$ is such
that for $x,y\in\prt\B(0,1)$,
\[
\lim_{\delta\downarrow0} \frac{\P^{x+ \delta\n(x)}_{D_1}(X(\tau^X_{D_1}-) \in dy) } {
2\delta|x+ \delta\n(x)-y|^{-3} \mu(dy)} =1.
\]
We use this and \eqref{eqM52} to see that for $x,y\in\B(0,1)$,
%
%e3.4 #&#
%
\begin{equation}
\label{eqo81} \wh H^{x}_{D_1}\bigl(e(\zeta-) \in dy\bigr) =
2 |x -y|^{-3} \mu(dy).
\end{equation}

%
%f1 #&#
\begin{figure}

\includegraphics{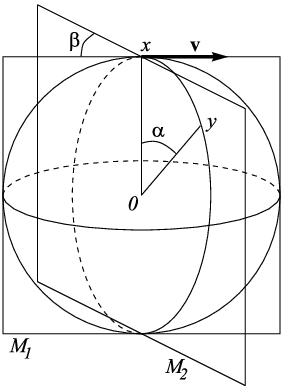}

\caption{The spherical coordinates $\alpha$ and $\beta$ used in the
derivation of length reduction
of the vector~$\bv$.}\label{figearth2}
\end{figure}

Note that, by symmetry, $\wh H^x( f_\bv(e))$ does not depend on
$x\in\prt D_1$ and $\bv\in\tngt_x \prt D_1$, so
we can fix arbitrarily $x\in\prt D_1$ and $\bv\in\tngt_x \prt D_1$
with $|\bv|>0$.
We will express $\mu(dy)$ and $ f_\bv(e)$ using
spherical coordinates. Let $\alpha$ denote the angle between
the radii of $\B(0,1)$ going from $0$ to $x$ and $y$ in $\prt\B(0,1)$.
Let $M_1$ be the
plane that contains $\bv$ and $0$, and let $M_2$ be the plane
that contains $0,x$ and $y$. Let $\beta$ be the angle between
$M_1$ and $M_2$; see Figure~\ref{figearth2}.
The uniform probability measure on the sphere $\prt\B(0,1)$
can be represented as\looseness=-1
%
%e3.5 #&#
%
\begin{equation}
\label{eqo82} \mu(dy) = (2\pi)^{-1}\,d\beta(1/2) \sin\alpha \,d\alpha.
\end{equation}
We have $|x-y| = 2\sin(\alpha/2)$, so
\eqref{eqo81}--\eqref{eqo82} yield
%
%e3.6 #&#
%
\begin{eqnarray}
\label{eqo83} \wh H^{x}_{D_1}\bigl(e(\zeta-) \in dy\bigr) &=&
2 \bigl(2\sin(\alpha/2)\bigr)^{-3} (2\pi)^{-1}\,d\beta(1/2)
\sin\alpha \,d\alpha
\nonumber
\\[-8pt]
\\[-8pt]
\nonumber
&=&\frac{1}{16 \pi} \frac{\sin\alpha} {\sin^3 (\alpha/2)} \,d\alpha \,d\beta.
\nonumber
\end{eqnarray}
It is elementary (although somewhat tedious) to check that
\[
\frac{|\pi_{y} (\bv)|}{|\bv|} = \bigl(\sin^2 \beta+ \cos^2 \beta
\cos^2 \alpha\bigr)^{1/2}.
\]
If $e(\zeta-) = y$, then
%
%e3.7 #&#
%
\begin{equation}
\label{eqo810} f_\bv( e) =\log\bigl|\pi_{e(\zeta-)} (\bv)\bigr| - \log|\bv|
=\log\bigl(\sin^2 \beta+ \cos^2 \beta\cos^2
\alpha\bigr)^{1/2}.
\end{equation}
We combine this formula with \eqref{eqo83} to see that
\[
\wh H^{x}_{D_1}\bigl( f_\bv( e)\bigr) = \int
_0^{2\pi} \int_0^\pi
\frac{1}{16\pi} \frac{\sin\alpha} {\sin^3 (\alpha/2)} \log\bigl(\sin^2 \beta+
\cos^2 \beta\cos^2 \alpha\bigr)^{1/2} \,d\alpha
\,d\beta.
\]
Part (i) of the theorem follows from this formula and Lemma
\ref{lemj271}.

(ii)
We will divide excursions into two families---these that return to
$\prt D$ relatively soon, and those that travel far away from $\prt D$.
The first part of the following argument shows that the excursions
which travel far away are likely to hit $\prt D$ at a random point
distributed almost uniformly over $\prt D$.
Excursions from $\prt D$ which do not travel far away contribute to the
estimate about as much as excursions from $\prt D_1$.

First, we will show that the harmonic measure in a
spherical shell has a density very close to a constant, under
some assumptions. Let $S(r, R) = \B(0,R) \setminus
\ol{\B(0,r)}$ denote the spherical shell with center 0, inner
radius $r$ and outer radius $R$. Let $h(r,R; x,y)$ be the
density of harmonic measure in $S(r,R)$ restricted to $\prt\B(0,r)$;
more precisely, let
\[
h(r,R; x,y) = \frac{\P^x_{S(r,R)} (X_{\tau^X_{S(r,R)}} \in dy)} {
\mu_r(dy)}
\]
for $x\in S(r,R)$ and $y\in\prt\B(0,r)$. For fixed $r,R$ and
$y$, the function $x\to h(r,R; x,y)$ is harmonic in $S(r,R)$.
By the Harnack principle, there exists $c_1>0$ such that for
any positive harmonic function $f$ in $\B(0,1)$, we have $c_1 <
f(v) / f(z) < 1/c_1$ for all $v,z\in\B(0,1/2)$. By scaling,
for any $r>0$ and for any positive harmonic function $f$ in
$\B(0,r)$, we have $c_1 < f(v) / f(z) < 1/c_1$ for all $v,z\in
\B(0,r/2)$. We can find a finite number $N$ such that there
exist $x_k \in\prt\B(0, 2r)$, $k=1,\ldots, N$, such that $\prt
\B(0,2r) \subset\bigcup_{1\leq k \leq N} \B(x_k, r/2)$. Then
the standard chaining argument shows that for $R\geq3r$ and
every positive harmonic function $f$ in $S(r,R)$, we have
$c_1^N < f(v) / f(z) < 1/c_1^N$ for all $v,z\in\B(0,2r)$. Let
$c_2 = c_1^N$. Consider a large integer $m$. As a particular
case of the last formula, we obtain that
%
%e3.8 #&#
%
\begin{equation}
\label{eqo84} c_2 < h\bigl(2^k,2^m; x,y
\bigr) / h\bigl(2^k,2^m; v,y\bigr)< 1/c_2
\end{equation}
for $0\leq k \leq m-2$, $y\in\prt\B(0,2^k)$ and $x,v \in\prt
\B(0, 2^{k+1})$. By the strong Markov property for Brownian
motion applied at the hitting time of $\prt\B(0, 2^{k+1})$,
\[
h\bigl(2^k,2^m; x,y\bigr) = \int_{\prt\B(0, 2^{k+1})}
h\bigl(2^k,2^m; v,y\bigr) h\bigl(2^{k+1},2^m;
x,v\bigr) \mu_{2^{k+1}}(dv)
\]
for $0\leq k \leq m-3$, $y\in\prt\B(0,2^k)$ and $x \in\prt
\B(0, 2^{k+2})$. This, \eqref{eqo84} and Lemma 6.1 of
\cite{BTW} imply, using the same argument as at the end of the
proof of Theorem~6.1 in~\cite{BTW}, that for any $c_3<1$
arbitrarily close to $1$ there exists $m_0$ such that for
$m\geq m_0$,
\[
c_3 < h\bigl(1,2^m; x,y\bigr) / h\bigl(1,2^m;
v,y\bigr)< 1/c_3
\]
for $y\in\prt\B(0,1)$ and $x,v \in\prt\B(0, 2^{m-1})$. By
applying a rotation, we obtain the following variant of the
above result. For any $c_3<1$ arbitrarily close to $1$ there
exists $m_0$ such that for $m\geq m_0$,
%
%e3.9 #&#
%
\begin{equation}
\label{eqo85} c_3 < h\bigl(1,2^m; x,y\bigr) / h
\bigl(1,2^m; x,z\bigr)< 1/c_3
\end{equation}
for $y,z\in\prt\B(0,1)$ and $x \in\prt\B(0, 2^{m-1})$.

Suppose that $\rho$ used in the definition of $D$ satisfies
$2^{m+1} \leq\rho\leq2^{m+2}$ for some $m\geq m_0$. Let
\begin{eqnarray*}
T_1 & =& 0,
\\
U_k & = &\inf\bigl\{t > T_k\dvtx X_{t} \in
\prt S\bigl(1,2^{m}\bigr)\bigr\},\qquad k\geq1,
\\
T_k & = &\inf\bigl\{t > U_{k-1}\dvtx X_t \in
\prt\B\bigl(0, 2^{m-1}\bigr)\bigr\},\qquad k\geq2.
\end{eqnarray*}
Then for $x\in\prt\B(0, 2^{m-1})$ and $y \in\prt D=\prt
\B(0,1)$,
\begin{eqnarray*}
\P^x_D (X_{T^X_{\prt D}} \in dy) &= &\sum
_{k=1}^\infty \P^x_D
\bigl(X_{U_k} \in dy; X_{U_j} \in\prt\B\bigl(0,2^m
\bigr), j < k \bigr)
\\
&= &\sum_{k=1}^\infty \E^x_D
\bigl( \P^{X_{T_k}}_{S(1,2^m)} (X_{U_k} \in dy)
\bone_{\{ X_{U_j} \in\prt\B(0,2^m), j < k\}} \bigr)
\\
&= &\sum_{k=1}^\infty \E^x_D
\bigl( h\bigl(1, 2^m; X_{T_k}, y\bigr) \mu(dy)
\bone_{\{ X_{U_j} \in\prt\B(0,2^m), j < k\}} \bigr).
\end{eqnarray*}
This and \eqref{eqo85} imply that
\[
c_3 < \P^x_D(X_{T^X_{\prt D}} \in dy) /
\P^x_D(X_{T^X_{\prt D}} \in dz) < 1/c_3
\]
for $y,z\in\prt\B(0,1)$ and $x \in\prt\B(0, 2^{m-1})$. The last estimate
and the strong Markov property of excursion laws applied at the
hitting time $T_{\prt\B(0,2^{m-1})}$ of $\prt\B(0,2^{m-1})$ show that
%
%e3.10 #&#
%
\begin{eqnarray}
\label{eqo87} c_3 &< &H^x\bigl( e(\zeta-) \in dy;
T_{\prt\B(0,2^{m-1})}< \zeta\bigr)
\nonumber
\\[-8pt]
\\[-8pt]
\nonumber
&&{} / H^x\bigl( e(\zeta-) \in dz;
T_{\prt\B(0,2^{m-1})}< \zeta\bigr) < 1/c_3
\end{eqnarray}
for $x,y,z\in\prt\B(0,1)$. Informally speaking, for sufficiently
large $m$ (and $\rho$), the density of $H^x( e(\zeta-) \in dy;
T_{\prt
\B(0,2^{m-1})}< \zeta)$ is arbitrarily close to a constant on~$\prt D$.

The probability that 3-dimensional Brownian motion starting
from $x + \delta\n(x)$, $x\in\prt\B(0,1)$, will never return
to $\prt\B(0,1)$ is equal to $1-(1+ \delta) ^{-1}$. This and
\eqref{eqM52} imply that for any $c_4>0$ there exists $m_1$
such that for $m\geq m_1$ and $x\in\prt D$,
\[
1- c_4 < H^x(T_{\prt\B(0,2^{m-1})} < \zeta) < 1+
c_4.
\]
It follows from this and \eqref{eqo87} that for any $c_5>0$
and sufficiently large $\rho$, we have for $x,y \in\prt D$,
%
%e3.11 #&#
%
\begin{equation}
\label{eqo811} 1-c_5 < H^x\bigl( e(\zeta-) \in dy;
T_{\prt\B(0,2^{m-1})}< \zeta\bigr) / \mu(dy) < 1+c_5.
\end{equation}

We have by continuity of probability that
%
%e3.12 #&#
%
\begin{equation}
\label{o91} \lim_{m\to\infty} H^x\bigl( e(\zeta-) \in dy;
T_{\prt\B(0,2^{m-1})}> \zeta\bigr) = \wh H^x\bigl( e(\zeta-) \in dy
\bigr).
\end{equation}
Note that the above limit is monotone.

We have
%
%e3.13 #&#
%
\begin{eqnarray}
\label{e089} H^{x} \bigl(f_\bv( e)\bigr) &=& \int
_{\prt D} \bigl(\log\bigl|\pi_{y} (\bv)\bigr| - \log|\bv|\bigr)
H^x\bigl( e(\zeta-) \in dy\bigr)
\nonumber\\
&=& \int_{\prt D} \bigl(\log\bigl|\pi_{y} (\bv)\bigr| - \log|
\bv|\bigr) H^x\bigl( e(\zeta-) \in dy; T_{\prt\B(0,2^{m-1})}> \zeta\bigr)
\\
&&{} + \int_{\prt D} \bigl(\log\bigl|\pi_{y} (\bv)\bigr| - \log|
\bv|\bigr) H^x\bigl( e(\zeta-) \in dy; T_{\prt\B(0,2^{m-1})}< \zeta\bigr).
\nonumber
\end{eqnarray}
It follows from \eqref{o91}, monotone convergence theorem and
part (i) of this theorem that
%
%e3.14 #&#
%
\begin{eqnarray}
\label{eqo815} &&\lim_{m\to\infty} \int_{\prt D}
\bigl(\log\bigl|\pi_{y} (\bv)\bigr| - \log|\bv|\bigr) H^x\bigl( e(
\zeta-) \in dy; T_{\prt\B(0,2^{m-1})}> \zeta\bigr)
\nonumber
\\
&&\qquad=\int_{\prt D} \bigl(\log\bigl|\pi_{y} (\bv)\bigr| - \log|
\bv|\bigr) \wh H^x\bigl( e(\zeta-) \in dy\bigr) = \wh H^x
\bigl(f_\bv(e)\bigr) \\
&&\qquad= \sqrt{2} - 1 - \log (1+\sqrt2).\nonumber
\end{eqnarray}
We combine \eqref{eqo82}, \eqref{eqo810}, \eqref{eqo811}
and Lemma \ref{lemj271} to obtain
\begin{eqnarray*}
&&\lim_{m\to\infty} \int_{\prt D} \bigl(\log\bigl|
\pi_{y} (\bv)\bigr| - \log|\bv|\bigr) H^x\bigl( e(\zeta-) \in
dy; T_{\prt\B(0,2^{m-1})}< \zeta\bigr)
\\
&&\qquad=\int_{\prt D} \bigl(\log\bigl|\pi_{y} (\bv)\bigr| - \log|
\bv|\bigr) \mu(dy)
\\
&&\qquad= \int_0^{2\pi} \int_0^\pi
\frac{1}{4\pi} \sin\alpha \log\bigl(\sin^2 \beta+
\cos^2 \beta\cos^2 \alpha\bigr)^{1/2} \,d\alpha
\,d\beta= \log2 - 1.
\end{eqnarray*}
Part (ii) of the theorem follows from this formula,
\eqref{e089} and \eqref{eqo815}.
\end{pf*}

%s4 #&#
\section{Recurrence of synchronous couplings in 3-dimensional
torus}\label{secrec}

The natural scale for our arguments is the combination of the local
time scale and the logarithmic scale.
The reason is that when the ``real'' time reaches a fixed level, the
vector between $X$ and $Y$ is not parallel to $\prt D$ in any
reasonable sense. On the contrary, when the local time reaches a fixed
level, the vector between $X$ and $Y$ is approximately parallel to
$\prt D$, in a sense. The last observation is used repeatedly in our arguments.
The following definitions introduce the ``local time scale.''

Let $\sigma^X_t = \inf\{s\geq0\dvtx L^X_s \geq t\}$, $\sigma^Y_t =
\inf\{s\geq0\dvtx L^Y_s \geq t\}$ and $\sigma'_b = \sigma^X_b
\land\sigma^Y_b$. The random variable $\sigma'_b$ was denoted
$\sigma_*$ in Section~\ref{secdiffrbm} for consistency with
the notation of \cite{B3}. The new notation, $\sigma'_b$, is
more appropriate for this paper.
An alternative formula is
$\sigma'_b = \inf\{t\geq0\dvtx L^X_t \lor L^Y_t \geq b\}$. Let
\[
\sigma'_{(k+1)b} = \inf \bigl\{t\geq\sigma'_{kb}
\dvtx \bigl(L^{X}_t - L^{X}_{\sigma'_{kb}}
\bigr) \lor\bigl( L^{Y}_t - L^{Y}_{\sigma'_{kb}}
\bigr) \geq b \bigr\}
\]
for $k\geq1$.
Note that, typically,
$\sigma'_{kb} $ is not equal to $ \inf\{t\geq0\dvtx L^X_t \lor L^Y_t
\geq
kb\}$. Let
%
%e4.1 #&#
%
\begin{eqnarray}
\label{eqstdnot} R_t &=& |X_t - Y_t|,\qquad
M_t = \log R_t,\qquad t\geq0,
\nonumber
\\[-8pt]
\\[-8pt]
\nonumber
V_k &=&M_{\sigma'_{kb}},\qquad k=0,1,\ldots.
\end{eqnarray}

The following lemma shows that over a long time interval, the distance
between~$X$ and~$Y$ is unlikely to decrease.\vspace*{-3pt}

%
%le4.1 #&#
\begin{lemma}\label{lemo303}
For any $c_0>0$, $\beta_1 \in(0,1)$ and $p<1$ there exist
$c_1, b,\eps_1>0$ such that if $\eps\leq\eps_1$, $x_0
\in\prt D$, $y_0\in\ol D$, $|x_0 - y_0| = \eps$, $X_0 = x_0$,
$Y_0=y_0$ and
%
%e4.2 #&#
%
\begin{equation}
\label{eqn710} \frac{ \llvert   \< y_0 - x_0,
\n (x_0 ) \> \rrvert  } {
\llvert  y_0-x_0 \rrvert } \leq c_0 \eps^{\beta_1},
\end{equation}
then
\[
\P^{x_0,y_0} ( V_1- V_0 \geq c_1 )
\geq p.
\]
\end{lemma}

\begin{pf}
It suffices to prove the lemma for $c_0=1$. To see this, choose any
$\beta_1^* \in(0,\beta_1)$ and note that $c_0
\eps^{\beta_1} \leq\eps^{\beta_1^*}$ for some $\eps_*>0$ and all
$\eps
\in(0,\eps_*)$. Hence, if the lemma is proved for $\beta_1^*$ in place
of $\beta_1$, with $1$ in place of $c_0$ and for $\eps< \eps_1$, then
it also holds for $\beta_1$, $c_0$ and $\eps< \eps_1 \land\eps_*$.

\textit{Step} 1.
In this step, the distance between $X$ and $Y$ is approximated by a sum
of increments related to excursions. The rate of increase (or decrease)
of the distance is expressed using excursion theory-based calculations
from Section~\ref{secexponent}.

Recall the results from \cite{B3} reviewed in
Section~\ref{secdiffrbm}. Suppose that $\eps_*>0$, $x_0 \in
\prt D$, $\bv\in\T_{x_0}\prt D$, $|\bv| =1$, $X_0 = x_0$ and
let $e_u$ be the first excursion of $X$ from $\prt D$ with
$|e_u(0) - e_u(\zeta-)| \geq\eps_*$. Let $x_1 = e_u(\zeta-)$
and $\alpha= 3/4$. We will estimate $\P^{x_0}(|x_0 - e_u(0)|
\geq\eps_*^\alpha)$ and $\E^{x_0} [|\log|\pi_{x_1}\bv||
\bone_{\{|x_0 - e_u(0)| \leq\eps_*^\alpha\}} ] $.

Let $U_0 = \varnothing$, $U_1 = \B(x_0, \eps_*)\cap\prt D$, $U_k =
(\B(x_0,
k\eps_*) \setminus\B(x_0, (k-1)\eps_*))\cap\prt D$ for $k\geq
2$ and $T_0=0$.
Set $j_0 = 1$ and for $k \geq0$, set
\begin{eqnarray*}
T_{k+1} &=&\inf\bigl\{t\geq T_{k}\dvtx X_t \in
\prt D \setminus(U_{j_k -1} \cup U_{j_k } \cup U_{j_k +1} )
\bigr\},
\\
j_{k+1} &=& \min\{i\geq0\dvtx X_{T_{k+1}} \in U_i\}.
\end{eqnarray*}

Recall that $u$ denotes the starting time of the first
excursion of $X$ from $\prt D$ with $|e_u(0) - e_u(\zeta-)|
\geq\eps_*$. Let $p_1$ be the probability that $|x_0 -
e_u(0)|<\eps_* $ and note that $p_1>0$. The strong Markov
property applied at $T_k$ shows that $\P^{x_0}(u\leq T_{k+1}
\mid u \geq T_k)\geq p_1$. It follows that $\P^{x_0}(u\geq
T_{k}) \leq(1-p_1)^k$. For the event $\{|x_0 - e_u(0)| \geq
\eps_*^\alpha\}$ to occur, we have to have $u \geq T_k$ with $k \geq
\eps_*^\alpha/(2\eps_*)$. It follows that, setting $c_1 = -
(1/2)\log(1-p_1)>0$,
%
%e4.3 #&#
%
\begin{equation}
\label{n59} \P^{x_0}\bigl(\bigl|x_0 - e_u(0)\bigr| \geq
\eps_*^\alpha\bigr) \leq(1-p_1)^{\eps_*^\alpha/(2\eps_*)} = \exp\bigl(
- c_1 \eps_*^{\alpha-1}\bigr).
\end{equation}

Let $\beta= 5/8$ and note that if $|x_1 - x_0| \leq
\eps_*^\beta$, then $|\log|\pi_{x_1}\bv||  \leq c_2
\eps_*^{2\beta}$. Hence,
%
%e4.4 #&#
%
\begin{eqnarray}
\label{n63} &&\E^{x_0}\bigl [\bigl |\log|\pi_{x_1}\bv|\bigr|
\bone_{\{|x_0
- e_u(0)| \leq\eps_*^\alpha\}} \bigr]\nonumber\\
&&\qquad= \E^{x_0} \bigl[ \bigl|\log|\pi_{x_1}\bv|\bigr|
\bone_{\{|x_0
- e_u(0)| \leq\eps_*^\alpha\}} \bone_{\{ |x_1 - x_0| \leq\eps_*^\beta\}} \bigr]
\nonumber
\\[-8pt]
\\[-8pt]
\nonumber
&&\qquad\quad{} + \E^{x_0} \bigl[ \bigl|\log|\pi_{x_1}\bv|\bigr| \bone_{\{|x_0 -
e_u(0)| \leq\eps_*^\alpha\}}
\bone_{\{ |x_1 - x_0| \geq\eps_*^\beta\}} \bigr]
\\
& &\qquad\leq c_2 \eps_*^{2\beta} + \E^{x_0} \bigl[\bigl|\log|
\pi_{x_1}\bv|\bigr| \bone_{\{|x_0 -
e_u(0)| \leq\eps_*^\alpha\}} \bone_{\{ |x_1 - x_0| \geq\eps_*^\beta\}} \bigr].
\nonumber
\end{eqnarray}

It follows from \eqref{eqo81} and \eqref{eqo811} that for
large $\rho$, small $\eps_*$, $|x_1 - x_0| \geq\eps_*^\beta$
and $|x_0 - x| \leq\eps_*^\alpha$,
%
%e4.5 #&#
%
\begin{equation}
\label{eqn72} \frac{H^{x}(e(\zeta-) \in d x_1)} {
H^{x_0}(e(\zeta-) \in d x_1)} \leq\frac{ |x -x_1|^{-3}} { |x_0 -x_1|^{-3}} \leq\frac{(\eps_*^\beta- \eps_*^\alpha)^{-3}}{\eps_*^{-3\beta}}
\leq1 + 6 \eps_*^{\alpha-\beta}.
\end{equation}

Let $c_* =\sqrt{2} +\log2 - 2 - \log(1+\sqrt2) \approx
-0.77$ be the constant in the statement of Theorem
\ref{tho53}(ii). Theorem \ref{tho53}(ii), the exit system
formula \eqref{exitsyst} and \eqref{eqn72} imply that for
any $c_4 \in(-c_*,1)$ and $c_3 \in(0, c_4 + c_*)$, all large $\rho$
and small $\eps_*>0$,
\begin{eqnarray*}
&&\E^{x_0} \bigl[\bigl |\log|\pi_{x_1}\bv|\bigr| \bone_{\{|x_0 -
e_u(0)| \leq\eps_*^\alpha\}}
\bone_{\{ |x_1 - x_0| \geq\eps_*^\beta\}} \bigr]
\\
&&\quad= \E^{x_0} \frac{ H^{X_u}
(\bigl|\log|\pi_{e(\zeta-)}\bv|\bigr|  \bone_{\{|x_0 -
X_u| \leq\eps_*^\alpha\}}
\bone_{\{ |e(\zeta-) - x_0| \geq\eps_*^\beta\}}
)} {
H^{X_u}
(
\bone_{\{ |e(\zeta-) - X_u| \geq\eps_* \}}
)}
\\
&&\quad\leq\frac{(1 + 6 \eps_*^{\alpha-\beta}) H^{x_0}
(|\log|\pi_{e(\zeta-)}\bv||
\bone_{\{ |e(\zeta-) - x_0| \geq\eps_*^\beta\}}
)} {
H^{x_0}
(
\bone_{\{ |e(\zeta-) - x_0| \geq\eps_* \}}
)}
\\
&&\quad\leq\frac{(1 + 6 \eps_*^{\alpha-\beta})
(c_3+|\sqrt{2} +\log2 - 2 - \log(1+\sqrt2)|)} {
H^{x_0}
(
\bone_{\{ |e(\zeta-) - x_0| \geq\eps_* \}}
)}
\\
&&\quad\leq c_4 / H^{x_0} ( \bone_{\{ |e(\zeta-) - x_0| \geq\eps_* \}} ).
\end{eqnarray*}
We combine the last estimate and \eqref{n63} to obtain
%
%e4.6 #&#
%
\begin{equation}
\label{n64}\qquad  \E^{x_0} \bigl[\bigl |\log|\pi_{x_1}\bv|\bigr|
\bone_{\{|x_0
- e_u(0)| \leq\eps_*^\alpha\}} \bigr] \leq c_2 \eps_*^{2\beta} +
c_4 / H^{x_0} (\bone_{\{ |e(\zeta-) - x_0| \geq\eps_* \}
} ).
\end{equation}

Recall the notation from the paragraph containing
\eqref{eqd111}. Consider an arbitrary $\bv_0 \in\R^3$. Since
$\prt D$ is a sphere with the unit radius, $\sh(x)$ is the
identity operator so $ \I_k = \exp(\Delta\ell^*_k )
\pi_{x^*_k}$ and, therefore,
%
%e4.7 #&#
%
\begin{eqnarray}
\label{n512} \I_{m^*} \circ\cdots\circ\I_0 (
\bv_0) &=&\exp \biggl(\sum_{0 \leq k \leq m^*} \Delta
\ell^*_k \biggr) \pi_{x^*_{m^*}} \circ\cdots\circ
\pi_{x^*_0} (\bv_0)
\nonumber\\
&=&\exp \bigl( \ell^*_{m^*+1} \bigr)
\pi_{x^*_{m^*}} \circ
\cdots\circ\pi_{x^*_0} (\bv_0)
\\
&=&\exp \bigl( L^X_{\sigma'_b} \bigr)
\pi_{x^*_{m^*}} \circ\cdots\circ\pi_{x^*_0} (\bv_0).\nonumber
\end{eqnarray}
We will estimate the above quantity, starting with the composition of
projection operators. We have
%
%e4.8 #&#
%
\begin{eqnarray}
\label{n513} && \log\bigl\llvert \pi_{x^*_{m^*}} \circ\cdots\circ
\pi_{x^*_0} (\bv_0) \bigr\rrvert
\nonumber
\\
&&\qquad= \sum_{1 \leq k \leq m^*} \bigl( \log\bigl\llvert
\pi_{x^*_{k}} \circ\cdots\circ\pi_{x^*_0} (\bv_0) \bigr
\rrvert - \log\bigl\llvert \pi_{x^*_{k-1}} \circ\cdots\circ\pi_{x^*_0}
(\bv_0) \bigr\rrvert \bigr)\\
&&\qquad\quad{} + \log\bigl\llvert \pi_{x^*_0} (
\bv_0)\bigr\rrvert.\nonumber
\end{eqnarray}

By the strong Markov property applied at the excursion endpoint
$s_{k-1}:= t^*_{k-1} + \zeta(e_{t^*_{k-1}})$, the conditional
distribution of
\[
\log\bigl\llvert \pi_{x^*_{k}} \circ\cdots\circ \pi_{x^*_0} (
\bv_0) \bigr\rrvert - \log\bigl\llvert \pi_{x^*_{k-1}} \circ
\cdots\circ\pi_{x^*_0} (\bv_0) \bigr\rrvert
\]
given $\F_{s_{k-1}}$ is the same as that of $|\log
|\pi_{x_1}\bv|| $, introduced at the beginning of the proof. Let
\[
F_k = \bigl\{ \bigl\llvert x^*_{k-1} -
e_{t^*_k}(0) \bigr\rrvert \leq\eps_*^\alpha \bigr\}.
\]
We see that the events $F_k$, $k\geq1$, are independent and so are the
random variables
%
%e4.9 #&#
%
\begin{equation}
\label{eqn83} \bigl\llvert \log\bigl\llvert \pi_{x^*_{k}} \circ\cdots\circ
\pi_{x^*_0} (\bv_0) \bigr\rrvert - \log \bigl| \pi_{x^*_{k-1}}
\circ \cdots\circ\pi_{x^*_0} (\bv_0) \bigr|\bigr\rrvert
\bone_{F_k}.
\end{equation}
It follows from \eqref{n64} that
\begin{eqnarray*}
&&\E^{x_0} \bigl[ \bigl\llvert \log \bigl| \pi_{x^*_{k}} \circ\cdots
\circ \pi_{x^*_0} (\bv_0) \bigr| - \log \bigl| \pi_{x^*_{k-1}}
\circ \cdots\circ\pi_{x^*_0} (\bv_0) \bigr|\bigr\rrvert
\bone_{F_k} \mid\F_{s_{k-1}} \bigr]
\\
&&\qquad \leq c_2 \eps_*^{2\beta} + c_4 /
H^{x_0} (\bone_{\{ |e(\zeta-) - x_0| \geq\eps_* \}
} ).
\end{eqnarray*}
Thus the process
\begin{eqnarray*}
N_n &=& n \bigl(c_2 \eps_*^{2\beta} +
c_4 / H^{x_0} (\bone_{\{ |e(\zeta-) - x_0| \geq\eps_* \}
} )\bigr)
\\
&&{} - \sum_{1 \leq k \leq n} \bigl\llvert \log \bigl|
\pi_{x^*_{k}} \circ\cdots\circ\pi_{x^*_0} (\bv_0)\bigr | -
\log \bigl| \pi_{x^*_{k-1}} \circ\cdots\circ\pi_{x^*_0} (
\bv_0) \bigr| \bigr\rrvert \bone_{F_k}
\end{eqnarray*}
is a submartingale. By the optional stopping theorem, $\E^{x_0}
N_{m^*} \geq0$, so
%
%e4.10 #&#
%
\begin{eqnarray}
\label{n57}\qquad  &&\E^{x_0} \biggl[ \sum_{1 \leq k \leq m^*}
\bigl\llvert \log \bigl| \pi_{x^*_{k}} \circ\cdots\circ\pi_{x^*_0} (
\bv_0) \bigr| - \log \bigl| \pi_{x^*_{k-1}} \circ\cdots\circ
\pi_{x^*_0} (\bv_0) \bigr| \bigr\rrvert \bone_{F_k}
\biggr]
\nonumber
\\[-8pt]
\\[-8pt]
\nonumber
&&\qquad\leq\E^{x_0} m^* \bigl(c_2 \eps_*^{2\beta} +
c_4 / H^{x_0} (\bone_{\{ |e(\zeta-) - x_0| \geq\eps_* \}
} )\bigr).
\end{eqnarray}
Formula \eqref{eqo81} implies that $H^{x_0}  ( \bone_{\{
|e(\zeta
-) -
x_0| \geq\eps_* \}}  ) \leq c_5 /\eps_*$. It follows from
the definition of $m^*$ and the exit system formula
\eqref{exitsyst} that $m^*$ has the Poisson distribution with
the expected value $b H^{x_0}  ( \bone_{\{ |e(\zeta-) -
x_0| \geq\eps_* \}}  )$. These observations and
\eqref{n57} yield for some $c_6>0$, any $c_7 \in(c_4,1)$ and
small $\eps_*$,
%
%e4.11 #&#
%
\begin{eqnarray}
\label{n58} &&\E^{x_0} \biggl[ \sum_{1 \leq k \leq m^*}
\bigl\llvert \log \bigl| \pi_{x^*_{k}} \circ\cdots\circ\pi_{x^*_0} (
\bv_0)\bigr | - \log \bigl| \pi_{x^*_{k-1}} \circ\cdots\circ
\pi_{x^*_0} (\bv_0) \bigr| \bigr\rrvert \prod
_{1 \leq j \leq m^*} \bone_{F_j} \biggr]\nonumber\hspace*{-15pt}
\\
&&\qquad\leq\E^{x_0} \biggl[ \sum_{1 \leq k \leq m^*} \bigl
\llvert \log \bigl| \pi_{x^*_{k}} \circ\cdots\circ\pi_{x^*_0} (
\bv_0) \bigr| - \log \bigl| \pi_{x^*_{k-1}} \circ\cdots\circ
\pi_{x^*_0} (\bv_0) \bigr| \bigr\rrvert \bone_{F_k}
\biggr]\hspace*{-15pt}
\\
&&\qquad\leq\bigl(c_6 \eps_*^{2\beta-1} + c_4 \bigr)b
\leq c_7 b.
\nonumber\hspace*{-15pt}
\end{eqnarray}
In addition, since we are dealing with a sum of i.i.d. random
variables given in~\eqref{eqn83}, and the sum has a Poisson
number $m^*$ of terms with large mean, it is easy to see that
for any $c_8 \in(c_7,1)$ and $p_2>0$ there exist $b_1$ and
$\eps_0$ such that for $b\geq b_1$ and $\eps_*\leq\eps_0$,
%
%e4.12 #&#
%
\begin{eqnarray}
\label{eqn84} &&\P^{x_0} \biggl(\sum_{1 \leq k \leq m^*}
\bigl\llvert \log \bigl| \pi_{x^*_{k}} \circ\cdots\circ\pi_{x^*_0} (
\bv_0) \bigr| - \log\bigl | \pi_{x^*_{k-1}} \circ\cdots\circ
\pi_{x^*_0} (\bv_0) \bigr| \bigr\rrvert \bone_{F_k}
\nonumber
\\[-8pt]
\\[-8pt]
\nonumber
&&\hspace*{272pt}{} \geq
c_8 b \biggr) \leq p_2.
\end{eqnarray}

A similar argument based on the strong Markov property applied
at times $s_k$ and the optional stopping theorem for
submartingales, combined with \eqref{n59}, gives
%
%e4.13 #&#
%
\begin{equation}
\label{eqn88} \qquad\P^{x_0} \biggl( \bigcup_{1\leq k \leq m^*}
F^c_k \biggr) \leq\E^{x_0} m^* \exp\bigl( -
c_1 \eps_*^{\alpha-1}\bigr) \leq c_9 b \exp\bigl(
- c_1 \eps_*^{\alpha-1}\bigr)\eps_*^{-1}.
\end{equation}

\textit{Step} 2.
We will use a result from a different paper to show that the discrete
approximation of the distance between $X$ and $Y$ employed in the
previous step is sufficiently accurate
for our purposes.

Recall the notation from Section~\ref{secdiffrbm}. We copy below \eqref{EM311}--\eqref{M316}
because these estimates are crucial to the present argument.
Fix an arbitrarily small $c_{10}>0$. There exist
$c_{11},c_{12},c_{13},\eps_0>0$, $\beta_1 \in(1,4/3)$ and
$\beta_2 \in(0, 4/3 - \beta_1)$ such that if $X_0=x$, $Y_0=y$,
$|x-y| =\eps< \eps_0$ and $\eps_* = c_{11} \eps$, then
%
%e4.14 #&#
%
\begin{equation}
\label{eqn75} \bigl|(Y_{\sigma'_b} - X_{\sigma'_b}) - \I_{m^*} \circ
\cdots\circ\I_0 (Y_{0} - X_{0})\bigr| \leq|
\Lambda| + \Xi,
\end{equation}
where $|\Lambda| < c_{10} \eps$, $\P^{x,y}$-a.s., and
%
%e4.15 #&#
%
\begin{equation}
\label{eqn76} \P^{x,y}\bigl(|\Xi| > c_{12}
\eps^{\beta_1}\bigr) \leq c_{13} \eps^{\beta_2}.
\end{equation}

We have
%
%e4.16 #&#
%
\begin{eqnarray}
\label{eqn86} && \log\bigl\llvert \pi_{x^*_{m^*}} \circ\cdots\circ
\pi_{x^*_0} (Y_{0} - X_{0}) \bigr\rrvert
\nonumber\\
&&\qquad= \sum_{1 \leq k \leq m^*} \bigl( \log\bigl\llvert
\pi_{x^*_{k}} \circ\cdots\circ\pi_{x^*_0} (Y_{0} -
X_{0}) \bigr\rrvert
\nonumber
\\[-8pt]
\\[-8pt]
\nonumber
&&\hspace*{68pt}{}- \log\bigl\llvert \pi_{x^*_{k-1}} \circ
\cdots\circ\pi_{x^*_0} (Y_{0} - X_{0}) \bigr
\rrvert \bigr)
\\
&&\qquad\quad{} + \log\bigl|\pi_{x^*_0} (Y_{0} - X_{0})\bigr|.\nonumber
\end{eqnarray}
Note that $x_0^* = x_0$.
It follows from \eqref{eqn710} that
%
%e4.17 #&#
%
\begin{equation}
\label{eqn810}\quad  \log\eps-\log\bigl|\pi_{x^*_0} (Y_{0} -
X_{0})\bigr| =\log\eps- \log\bigl|\pi_{x_0} (y_{0} -
x_{0})\bigr| \leq c_{14} \eps^{2\beta_1}.
\end{equation}
We combine this with \eqref{eqn84}, \eqref{eqn88} and
\eqref{eqn86} to see that for any $c_{15} \in(c_7,1)$ and
$p_2>0$, there exists $b_2$ such that for any $b\geq b_2$, there exists
$\eps_1>0$ such that for $\eps\leq\eps_1$,
%
%e4.18 #&#
%
\begin{equation}
\label{eqn85}\quad  \P \bigl(\bigl\llvert \log \bigl| \pi_{x^*_{m^*}} \circ\cdots\circ
\pi_{x^*_0} (Y_{0} - X_{0}) \bigr| -
\log|Y_0 - X_0| \bigr\rrvert \geq c_{15} b
\bigr) \leq p_2.
\end{equation}
A special case of \eqref{n512} is
\[
\I_{m^*} \circ\cdots\circ\I_0 (Y_{0} -
X_{0}) =\exp \bigl( L^X_{\sigma'_b} \bigr)
\pi_{x^*_{m^*}} \circ\cdots\circ\pi_{x^*_0} (Y_{0} -
X_{0}).
\]
This implies that
\[
\log\bigl\llvert \I_{m^*} \circ\cdots\circ\I_0
(Y_{0} - X_{0})\bigr\rrvert = L^X_{\sigma'_b}
+ \log\bigl\llvert \pi_{x^*_{m^*}} \circ\cdots\circ\pi_{x^*_0}
(Y_{0} - X_{0}) \bigr\rrvert.
\]

Recall that $|X_0-Y_0|=|x_0-y_0|=\eps$.
On the event $\{\sigma'_b = \sigma^X_b\}$ we have
$L^X_{\sigma'_b} = b$ so, in view of
\eqref{eqn810} and \eqref{eqn85},
for any $c_{16} \in(c_{15},1)$, $c_{17} = 1 - c_{16} >0$ and
$p_3>0$, there exists $b_3$ such that for any $b\geq b_3$, there exists
$\eps_2>0$ such that for $\eps\leq\eps_2$,
%
%e4.19 #&#
%
\begin{eqnarray}
\label{es224} \quad&&\P^{x_0,y_0} \bigl( \log\bigl\llvert \I_{m^*} \circ
\cdots\circ\I_0 (Y_{0} - X_{0})\bigr\rrvert
-b -\log\eps\leq-c_{16} b \mbox{ and } \sigma'_b
= \sigma^X_b \bigr)
\nonumber
\\
&&\qquad= \P^{x_0,y_0} \bigl( \log\bigl\llvert \I_{m^*} \circ\cdots\circ
\I_0 (Y_{0} - X_{0}) \bigr\rrvert% \\
%&&\hspace*{64pt}
\leq(1-
c_{16}) b + \log\eps\nonumber\\
&&\hspace*{225pt} \mbox{and } \sigma'_b
= \sigma^X_b \bigr)
\\
&&\qquad= \P^{x_0,y_0} \bigl( \bigl\llvert \I_{m^*} \circ\cdots\circ
\I_0 (Y_{0} - X_{0}) \bigr\rrvert %\\
%&&\hspace*{64pt}\leq\eps
\exp(c_{17} b) \mbox{ and } \sigma'_b =
\sigma^X_b \bigr) \nonumber\\
&&\qquad\leq p_3.\nonumber
\end{eqnarray}
Recall from \eqref{eqn75} that we can assume that $|\Lambda|
\leq c_{10} \eps$, a.s.
It follows from \eqref{eqn76} that for small $\eps$,
$\P(|\Xi| \geq c_{10} \eps) < p_3$. These remarks and \eqref{es224}
imply that
\begin{eqnarray*}
&&\P^{x,y} \bigl( \bigl\llvert \I_{m^*} \circ\cdots\circ
\I_0 (Y_{0} - X_{0})\bigr\rrvert -|\Lambda|
-|\Xi| \leq\eps\bigl(\exp(c_{17} b)-2c_{10}\bigr) \mbox{
and } \sigma'_b = \sigma^X_b
\bigr)\\
&&\qquad \leq2 p_3.
\end{eqnarray*}
We combine this estimate with \eqref{eqn75} to see that
%
%e4.20 #&#
%
\begin{eqnarray}
\label{es232} \qquad&&\P^{x_0,y_0} \bigl( \bigl\llvert (Y_{\sigma^X_b} -
X_{\sigma^X_b})\bigr\rrvert \leq\eps\bigl(\exp(c_{17}
b)-2c_{10}\bigr) \mbox{ and } \sigma'_b =
\sigma^X_b \bigr)
\nonumber
\\[-8pt]
\\[-8pt]
\nonumber
&&\qquad=\P^{x_0,y_0} \bigl( \bigl\llvert (Y_{\sigma'_b} - X_{\sigma'_b})
\bigr\rrvert \leq\eps\bigl(\exp(c_{17} b)-2c_{10}\bigr)
\mbox{ and } \sigma'_b = \sigma^X_b
\bigr) \leq2 p_3.
\nonumber
\end{eqnarray}
We choose large $b_4$ so that for $b\geq b_4$, $c_{18} =
c_{18}(b):= \exp(c_{17} b)-2c_{10}>1$.
We can now write \eqref{es232} as
%
%e4.21 #&#
%
\begin{eqnarray}
\label{july191}&& \P^{x_0,y_0} \bigl( \bigl\llvert (Y_{\sigma'_b} -
X_{\sigma'_b})\bigr\rrvert \leq c_{18}\eps \mbox{ and }
\sigma'_b = \sigma^X_b
\bigr)
\nonumber
\\[-8pt]
\\[-8pt]
\nonumber
&&\qquad=\P^{x_0,y_0} \bigl( \bigl\llvert (Y_{\sigma^X_b} - X_{\sigma^X_b})
\bigr\rrvert \leq c_{18}\eps \mbox{ and } \sigma'_b
= \sigma^X_b \bigr) \leq2 p_3.
\end{eqnarray}
Recall that $x_0 \in\prt D$, $y_0\in\ol D$, and
let $T' = \inf\{t\geq0\dvtx |X_t| = |Y_t|\}$. Note that the
distributions of $\{(X_t,Y_t), t\geq T'\}$ and
$\{(Y_t,X_t), t\geq T'\}$ are symmetric. Moreover, $Y_t \notin\prt D$
for $t < T'$
and, therefore, $L^Y_{T'} = 0$. It follows from this and \eqref
{july191} that
\begin{eqnarray*}
&&\P^{x_0,y_0} \bigl( \bigl\llvert (Y_{\sigma'_b} - X_{\sigma'_b})
\bigr\rrvert \leq c_{18}\eps, T' \leq
\sigma'_b \mbox{ and } \sigma'_b
= \sigma^Y_b \bigr)
\\
&&\qquad= \P^{x_0,y_0} \bigl( \bigl\llvert (Y_{\sigma'_b} - X_{\sigma'_b})
\bigr\rrvert \leq c_{18}\eps, T' \leq
\sigma'_b \mbox{ and } \sigma'_b
= \sigma^X_b \bigr) \leq2 p_3
\end{eqnarray*}
and
%
%e4.22 #&#
%
\begin{eqnarray}
\label{july193} &&\P^{x_0,y_0}  \bigl( \bigl\llvert (Y_{\sigma'_b} -
X_{\sigma'_b})\bigr\rrvert \leq c_{18}\eps \bigr)
\nonumber\\
&&\qquad\leq \P^{x_0,y_0} \bigl( \bigl\llvert (Y_{\sigma'_b} -
X_{\sigma'_b})\bigr\rrvert \leq c_{18}\eps, \mbox{ and }
\sigma'_b = \sigma^X_b
\bigr)
\\
&&\qquad\quad{} + \P^{x_0,y_0} \bigl( \bigl\llvert (Y_{\sigma'_b} -
X_{\sigma'_b})\bigr\rrvert \leq c_{18}\eps, T' \leq
\sigma'_b \mbox{ and } \sigma'_b
= \sigma^Y_b \bigr) \leq4 p_3.
\nonumber
\end{eqnarray}
Let $c_{19} = \log
c_{18}>0$. Then
\[
\P^{x_0,y_0}  ( V_1 - V_0\leq c_{19}
) = \P^{x_0,y_0} \bigl( \log\bigl\llvert (Y_{\sigma'_b} -
X_{\sigma
'_b})\bigr\rrvert - \log\eps\leq\log c_{18} \bigr) \leq4
p_3.
\]
Since $p_3>0$ is arbitrarily small and $c_{19} >0$, the lemma is proved.
\end{pf}

The following lemma estimates the distribution of the increment of the
logarithm of the distance between $X$ and $Y$. The assertion of the
lemma has two parts. One part says that the distribution is close to
the distribution of an integrable random variable. The other part shows
the that error of approximation is small in an appropriate sense.
Recall notation from \eqref{eqstdnot}.

%
%le4.2 #&#
\begin{lemma}\label{lemn81}
For any $\beta_1 \in(0,1/2)$ there exist
$\beta_2,b,c_1,\eps_1>0$ and a cumulative distribution function
$G\dvtx \R\to[0,1]$ satisfying $\int_{-\infty}^\infty|a| \,dG(a)
< \infty$ and such that if $\eps\leq\eps_1$, $x_0 \in\prt
D$, $y_0\in\ol D$, $|x_0 - y_0| = \eps$, $X_0 = x_0$,
$Y_0=y_0$ and
%
%e4.23 #&#
%
\begin{equation}
\label{eqd115} \frac{ \llvert   \< y_0 - x_0,
\n (x_0 ) \> \rrvert  } {
\llvert  y_0-x_0 \rrvert } \leq\eps^{\beta_1},
\end{equation}
then there exists an event $F$ such that
%
%e4.24 #&#
%e4.25 #&#
%
\begin{eqnarray}
\label{eqn201} \P^{x_0,y_0}\bigl(F^c\bigr) &\leq&
c_1 \eps^{\beta_2},
\\
\P^{x_0,y_0} \bigl(| V_1- V_0|\bone_{F}
\leq a \bigr) &\leq& G(a),\qquad  a\in \R. \label{eqn202}
\end{eqnarray}
\end{lemma}

\begin{pf}
\textit{Step} 1.
This step is devoted to a review of upper bounds on the rate of growth
of the distance between $X$ and $Y$. It also contains a list of
definitions (notation) used throughout the rest of the proof.

Fix $b$ as in Lemma \ref{lemo303}, some $\beta_3$ and $\beta_4$ such
that $\beta_1 < \beta_3 < \beta_4 < 1/2$, and consider the condition
%
%e4.26 #&#
%
\begin{equation}
\label{eqj201} \frac{ \llvert   \< y_0 - x_0,
\n (x_0 ) \> \rrvert  } {
\llvert  y_0-x_0 \rrvert } \leq c_2 \eps^{\beta_4},
\end{equation}
where $c_2 = 200\cdot2 ^{\beta_4}$. Note that $c_2 \eps^{\beta_4} <
\eps^{\beta_1}$ for small $\eps>0$.
It follows from \eqref{july193} that for some
$p_1\in(0,1)$, $\eps_1>0$ and $c_3=c_3(b)>0$, if $|x_0 -
y_0|=\eps\leq\eps_1$ and either \eqref{eqd115} or \eqref{eqj201}
holds, then
%
%e4.27 #&#
%
\begin{equation}
\label{eqn2230} \P^{x_0,y_0} \bigl( \llvert Y_{\sigma'_b} -
X_{\sigma'_b}\rrvert \leq c_3\eps \bigr) \leq p_1.
\end{equation}

Lemma 3.4 of \cite{B3} and its proof show that there exists
$c_4>0$ such that for all $x,y\in\ol D$ and $t\geq0$, we have $\P
^{x,y}$-a.s.,
%
%e4.28 #&#
%
\begin{equation}
\label{eqn181} |X_{t} - Y_{t} | \leq\exp
\bigl(c_4 \bigl(L^X_t + L^Y_t
\bigr) \bigr) |x-y|.
\end{equation}
By the Markov property, for any fixed $t,s\geq0$, a.s.,
%
%e4.29 #&#
%
\begin{equation}
\label{j271} |X_{t+s} - Y_{t+s} | \leq\exp
\bigl(c_4 \bigl(L^X_{t+s} -
L^X_t + L^Y_{t+s} -
L^Y_t\bigr) \bigr) |X_{t} - Y_{t}
|.
\end{equation}
Since the last formula holds for all rational $t,s\geq0$
simultaneously, a.s., and $X$ and $Y$ are continuous, the inequality
actually holds for all random times $t,s\geq0$ (not necessarily
stopping times).
We obtain from \eqref{eqn181},
%
%e4.30 #&#
%
\begin{equation}
\label{eqn212} \inf_{0 \leq t \leq\sigma'_b } |X_t - Y_t|
\geq\exp(-2c_4 b) |X_{\sigma'_b} - Y_{\sigma'_b}|.
\end{equation}

Let $c_5 = \exp(2c_4 b)$ and $c_6 =c_3 c_5^{-1}$. It follows from
\eqref{eqn2230} and \eqref{eqn212} that
%
%e4.31 #&#
%
\begin{equation}
\label{eqn1113} \P^{x_0,y_0} \Bigl( \inf_{0\leq t \leq\sigma'_b} \llvert
Y_{t} - X_{t}\rrvert \leq c_6 \eps \Bigr)
\leq p_1,
\end{equation}
and for any random time $T \in[0, \sigma'_b]$,
%
%e4.32 #&#
%
\begin{equation}
\sup_{T\leq t \leq\sigma'_b} \llvert Y_{t} - X_{t}
\rrvert \leq c_5 |Y_T - X_T|,\qquad
\P^{x_0,y_0}\mbox{-a.s.} \label{eqd21}
\end{equation}
In particular,
%
%e4.33 #&#
%
\begin{equation}
\label{j203} \sup_{0\leq t \leq\sigma'_b} \llvert Y_{t} -
X_{t}\rrvert \leq c_5 \eps,\qquad \P^{x_0,y_0}
\mbox{-a.s.}
\end{equation}

We set $c_7 = (-1 - 2 c_4 b) \land\log c_6$ and $c_8 =
e^{c_7}$. Hence,
%
%e4.34 #&#
%
\begin{equation}
\label{a111} c_8 c_5 \leq\exp(-1 - 2
c_4 b + 2 c_4 b) < 1/2.
\end{equation}
Obviously, \eqref{eqn1113} implies that, assuming that either \eqref
{eqd115} or \eqref{eqj201} holds,
%
%e4.35 #&#
%
\begin{equation}
\label{eqn1114} \P^{x_0,y_0} \Bigl( \inf_{0\leq t \leq\sigma'_b} \llvert
Y_{t} - X_{t}\rrvert \leq c_8\eps \Bigr)
\leq p_1.
\end{equation}

Let $U_0 = 0$ and
\[
S_1 = \inf\{t\geq0\dvtx M_t - M_0 \leq
c_7\} = \inf\bigl\{t\geq0\dvtx |X_t - Y_t| \leq
c_8 |X_0 - Y_0|\bigr\}.
\]
Here and later, $\inf\varnothing= \infty$. Note that at
least one of the processes $X$ and $Y$ must belong to $\prt D$
at time $S_1$.
We proceed by induction.
First assume that $X_{S_k}\in\prt D$.
Let $z_k \in\prt D$ be the point such that
$\bn(z_k) = \frac{Y_{S_k} - X_{S_k}}{|Y_{S_k} - X_{S_k}|}$, and
for some $c_9>0$ (to be specified later) and $k\geq1$, let
\begin{eqnarray*}
U_k &=& \inf\{t\geq S_k\dvtx Y_t \in\prt D\},
\\
S_{k} &=& \inf\{t\geq U_{k-1}\dvtx M_t -
M_{U_{k-1}} \leq c_7\} \\
&=& \inf\bigl\{ t\geq U_{k-1}\dvtx
|X_t - Y_t| \leq c_8 |X_{U_{k-1}} -
Y_{U_{k-1}}| \bigr\},
\\
F_k &=& \bigl\{S_k < \sigma'_b
\bigr\},
\\
\G_k &=& \sigma( B_t, t\leq U_k),
\\
J_k &=& n \in\Z\mbox{ such that } 2^{-n}
\leq|X_{S_k} -z_k| < 2^{-n+1},
\\
d_k & =& |X_{U_{k-1}} - Y_{U_{k-1}}| \qquad \bigl(\mbox{hence, }
d_0 = |X_0 - Y_0| = \eps\bigr),
\\
I_k &=& \bigl\{2^{-J_k} \geq d_k^{\beta_3}
\bigr\},
\\
S^*_k &=& \inf \bigl\{t\geq S_k\dvtx |X_t -
X_{S_k}| \geq d_k^{\beta
_4} \bigr\},
\\
C_k &=& \bigl\{U_k \leq S^*_k \bigr\},
\\
G_k &= &\bigl\{ |X_{U_k}- Y_{U_k}| \geq
c_9 2^{-J_k} d_k \bigr\},
\\
K_k&=& \biggl\{\frac{ \llvert   \< X_{U_k}- Y_{U_k},
\n (Y_{U_k} ) \> \rrvert  } {
\llvert  X_{U_k}- Y_{U_k} \rrvert } \leq c_2
|X_{U_{k}} - Y_{U_{k}}|^{\beta_4} \biggr\},
\\
A_k &= &F_k \cap I_k \cap C_k
\cap G_k \cap K_k,
\\
A^+_k &= &\bigcap_{j\leq k}
A_j.
\end{eqnarray*}
If $X_{S_k}\notin\prt D$, then we must have $Y_{S_k}\in\prt D$, and
we apply all the above definitions with the roles of $X$ and $Y$
interchanged. In the rest of the proof, we will discuss only the case
when $X_{S_k}\in\prt D$. Our arguments hold in the other case by symmetry.

\textit{Step} 2. In this step we will prove that, for some $c_{10},
c_{11}<\infty$, $k\geq1$ and $m$ such that $2^{-m} \geq
d_k^{\beta_3}$, on $A^+_{k-1}$,
%
%e4.36 #&#
%e4.37 #&#
%
\begin{eqnarray}
\label{eqn1120} \P\bigl( A^c_k \cap F_k
\mid\G_{k-1}\bigr) &\leq& c_{10} d_k^{\beta_3},
\\
\P\bigl(\{J_k \geq m\} \cap F_k \mid\G_{k-1}
\bigr) &\leq& c_{11} 2^{-m}.\label {eqn188}
\end{eqnarray}
Informally speaking, we will show that some events are unlikely. Given
that they do not happen, we will find good estimates for the distance
between $X$ and $Y$. This step is subdivided into further substeps
because we have to analyze several families of ``unusual'' events and
show that they all have small probabilities. The first substep will
show that ``long'' excursions are unlikely.

\textit{Step} 2.1.
Let $U'_k = S^*_k \land U_k$.
Note that $U'_k = U_k$ on $C_k$ and $U'_k= S^*_k$ on $C^c_k$.
If $F_k$ holds, then $d_k \leq c_5 \eps$, by \eqref{j203}.
By the definition of $S_k$, $|X_{S_k} - Y_{S_k}| = c_8 d_k$,
if $S_k < \infty$. We have assumed that $X_{S_k} \in\prt D$ so
$\dist(Y_{S_k}, \prt D) \leq c_8 d_k $.
We apply Lemma 3.2 of
\cite{BCJ} to the process $Y$ at the stopping time $S_k$ to see that
for some $c_{12}>0$,
%
%e4.38 #&#
%
\begin{equation}
\label{j211} \P\bigl( |Y_{S_k} - Y_{U'_k}| \geq
d_k ^{\beta_4} /3\bigr) \leq c_{12}
d_k^{1-\beta_4}.
\end{equation}

We will show that if $\eps_1 >0$ is sufficiently small and $\eps<
\eps
_1$, then $L^X_{U'_k} - L^X_{S_k} \leq3(1+c_5) d_k^{\beta_4}$. Suppose
that the last inequality does not hold, and let $T^*_k = \inf\{t \geq S_k\dvtx L^X_{t} - L^X_{S_k} = 3(1+c_5) d_k^{\beta_4}\}$.
Then by assumption we have $T^*_k \leq U'_k \leq S^*_k$.
Assuming that $F_k$ holds and using \eqref{j203},
%
%e4.39 #&#
%
\begin{equation}
\label{j224} |X_{S_k} - Y_{S_k}| = c_8
d_k \leq c_5 \eps.
\end{equation}
Hence, if $\eps$ is sufficiently small, then $c_8 d_k \leq d_k^{\beta
_4}/3$ and, therefore,
%
%e4.40 #&#
%
\begin{equation}
\label{j221} |X_{S_k} - Y_{S_k}| \leq d_k^{\beta_4}/3.
\end{equation}
It follows from the definitions of $U_k$ and $U'_k$ that $L^Y_{U'_k} -
L^Y_{S_k} =0$.
If $\eps$ is small then $d_k$ is small and $3(1+c_5) d_k^{\beta_4} < b
\land1/100$. So the definition of $T^*_k$, \eqref{j221} and~\eqref
{eqn181} imply that
%
%e4.41 #&#
%
\begin{equation}
\label{j222} |X_{T^*_k} - Y_{T^*_k}| \leq c_5
d_k^{\beta_4}.
\end{equation}
For all $t\in[S_k, T^*_k]\subset[S_k, S^*_k]$ such that $X_t \in\prt D$,
the angle between $\n(X_t)$ and $\n(X_{S_k})$ is less than $2
d_k^{\beta
_4}$. It follows that the angle between $\int_{S_k}^{T^*_k}
\n(X_t) \,dL^X_t$ and $\n(X_{S_k})$ is also smaller than $2 d_k^{\beta_4}
< 1/50$. Moreover, the length of\break $\int_{S_k}^{T^*_k} \n(X_t) \,dL^X_t$ is
greater than $2 (1+c_5) d_k^{\beta_4}$.
Recall that $Y_t \notin\prt D$ for $t\in[S_k, T^*_k]$. Thus $\int_{S_k}^{T^*_k} \n(Y_t)
\,dL^Y_t=0$ and, therefore,
\[
X_{T^*_k} - Y_{T^*_k} = X_{S_k} - Y_{S_k} +
\int_{S_k}^{T^*_k} \n(X_t)
\,dL^X_t.
\]
This relation, the fact that $X_{S_k}\in\prt D$, \eqref{j221},
\eqref
{j222} and our observations about the direction and length of $\int_{S_k}^{T^*_k} \n(X_t) \,dL^X_t$ imply that $Y_{T^*_k}$ must be at least
$(1+c_5) d_k^{\beta_4}$ units inside the ball $\B(0,1)$. This is
impossible so the claim that $L^X_{U'_k} - L^X_{S_k} \leq3(1+c_5)
d_k^{\beta_4}$ is proved.

Recall that, assuming that $F_k$ holds, $c_8 d_k \leq c_5 \eps$. Hence,
if $\eps_1>0$ is sufficiently small and $\eps< \eps_1$, then $3(1+c_5)
d_k^{\beta_4} < b$.
Since $L^X_{U'_k} - L^X_{S_k} \leq3(1+c_5) d_k^{\beta_4}$, $L^Y_{U'_k}
- L^Y_{S_k} =0$
and $|X_{S_k} - Y_{S_k}| = c_8 d_k$, we have by \eqref{j271}, for all
$t\in[S_k, U'_k]$,
%
%e4.42 #&#
%
\begin{equation}
\label{j2210} |X_{t} - Y_{t}| \leq c_5
c_8 d_k < d_k^{\beta_4}/3.
\end{equation}
In particular, $|X_{U'_k} - Y_{U'_k}| < d_k^{\beta_4}/3$.
This, the definitions of $S_k^*$ and $U'_k$ and \eqref{j221} imply
that, assuming that $C_k$ does not hold,
$|Y_{S_k} - Y_{S^*_k }| = |Y_{S_k} - Y_{U'_k }| \geq d_k ^{\beta_4} /3$.
This and \eqref{j211} imply that, on $A^+_{k-1}$,
%
%e4.43 #&#
%
\begin{equation}
\label{eqn211} \P\bigl( C^c_k \cap F_k \mid
\G_{k-1}\bigr) \leq c_{13} d_k^{1-\beta_4}.
\end{equation}

We record, for future reference, the following variants of \eqref
{j2210}. If $C_k\cap F_k$ holds, then $U'_k = U_k$ and for any random
time $R \in[S_k, U_k]$
and all $t\in[R, U_k]$,
%
%e4.44 #&#
%
\begin{equation}
\label{j272} |X_{t} - Y_{t}| \leq c_5
c_8 |X_{R} - Y_{R}|.
\end{equation}
It follows from \eqref{j271} that if $C_k\cap F_k$ holds, then for all
$t\in[U_{k-1}, U_k]$,
%
%e4.45 #&#
%
\begin{equation}
\label{20121} |X_{t} - Y_{t}| \leq c_5
|X_{U_{k-1}} - Y_{U_{k-1}}| = c_5 d_k.
\end{equation}
Since $|X_{S_k} - Y_{S_k}| = c_8 d_k$, if $F_k$ holds, then we have by
\eqref{j271} and \eqref{a111}, for all $t\in[S_k, \sigma'_b]$,
%
%e4.46 #&#
%
\begin{equation}
\label{a31} |X_{t} - Y_{t}| \leq c_5
c_8 d_k < d_k/2.
\end{equation}

\textit{Step} 2.2.
The intuitive meaning of the technical estimate in this step is that if
the vector between $X$ and $Y$ is close to the normal to $\prt D$,
then $Y$ must have traveled a long distance since it last visited $\prt D$.

Assume that $F_k$ holds.
Let $U^*_k = \sup\{t< U_k\dvtx Y_t \in\prt D\}$ and
$\wt U_k = U_{k-1} \lor U^*_k$. It is easy to see
that, a.s., $U^*_k \leq\wt U_k < S_k < U_k$, for $k\geq2$. (We will
limit our discussion to the case $k\geq2$; the case $k=1$
requires minor modifications so we omit the proof.) Random
times $U^*_k $ and $ U_k$ are the endpoints of an excursion of
$Y$ from $\prt D$.
Suppose that $J_k\geq m$ and $2^{-m} \geq
d_k^{\beta_3}$. By \eqref{eqd21}, $d_k \leq c_5 \eps$ so, assuming
that $\eps_1 >0$ is small and $\eps\leq\eps_1$, we have $c_8 d_k
\leq
d_k^{\beta_3} \leq2^{-m}$. We have $|X_{S_k} -z_k| \leq2^{-m+1}$
and, using
\eqref{j224},
%
%e4.47 #&#
%
\begin{equation}
\label{eqn2210}\qquad  |Y_{S_k} -z_k| \leq|X_{S_k}
-z_k| + |X_{S_k} - Y_{S_k} | \leq2^{-m+1}
+ c_8 d_k \leq2^{-m+2}.
\end{equation}

Suppose that $\sup_{\wt U_k \leq t \leq S_k, X_t \in\prt D} |Y_t
- z_k| \leq c_{14}:=1/400$. We will show that this assumption
leads to a contradiction. The assumption and \eqref{j203}
imply that, for small $\eps$, $\sup_{\wt U_k \leq t \leq S_k, X_t
\in\prt D} |X_t - z_k| \leq2 c_{14}$. This in turn implies
that for all $t\in[\wt U_k, S_k]$ such that $X_t \in\prt D$,
the angle between $\n(X_t)$ and $\n(z_k)$ is less than $4
c_{14}$. It follows that the angle between $\int_{\wt U_k} ^{S_k}
\n(X_t) \,dL^X_t$ and $\n(z_k)$ is also smaller than $4 c_{14}$.
Note that $Y_t \notin\prt D$ for $t\in[\wt U_k, S_k]$ by the
definition of $\wt U_k$. Thus $\int_{\wt U_k}^{S_k} \n(Y_t)
\,dL^Y_t=0$ and, therefore,
%
%e4.48 #&#
%
\begin{equation}
\label{eqn2211} X_{S_k} - Y_{S_k} = X_{\wt U_k} -
Y_{\wt U_k} + \int_{\wt U_k}^{S_k} \n
(X_t) \,dL^X_t.
\end{equation}
Recall that $X_{S_k} - Y_{S_k}$ is a positive multiple of
$-\n(z_k)$ and $X_{S_k} \in\prt D$. Assume that $K_{k-1}$ holds.
If $\wt U_k = U^*_k$, then $Y_{\wt U_k} \in\prt D$.
Next consider the case $\wt U_k > U^*_k$.
In this case, $X_{U_{k-1}} \in\prt D$ and, assuming that $\eps>0$ is
small, the vector $Y_{U_{k-1}} - X_{U_{k-1}}$
is almost orthogonal to $\n(X_{U_{k-1}})$. More precisely, $K_{k-1}$
implies that $\dist(Y_{U_{k-1}}, \prt D) \leq2 c_2 d_{k}^{1 + \beta_4}$.
These observations and the fact that the angle between $\int_{\wt U_k} ^{S_k}
\n(X_t) \,dL^X_t$ and $\n(z_k)$ is smaller than $4 c_{14}$ show
that \eqref{eqn2211} cannot be true; see Figure~\ref{figearth1}.
This contradiction
implies that $\sup_{\wt U_k \leq t \leq S_k, X_t \in\prt D} |Y_t -
z_k| \geq c_{14}$.
We combine this with~\eqref{eqn2210} to
see that $\sup_{\wt U_k \leq t \leq S_k, X_t \in\prt D} |Y_t -
Y_{S_k}| \geq c_{15}:= c_{14}/2$, for some $m_1$ and all
$m\geq m_1$. Suppose that $s_1$ is such that $\wt U_k \leq s_1
\leq S_k, X_{s_1} \in\prt D$ and $|Y_{s_1} - Y_{S_k}| \geq
c_{15}$. Then either $|Y_{\wt U_k} - Y_{S_k}| \geq c_{15}/2$ or
$|Y_{\wt U_k} - Y_{s_1}| \geq c_{15}/2$. Since $X_{S_k} \in\prt
D$, it follows that there exists $s_2$ such that $\wt U_k \leq
s_2 \leq S_k, X_{s_2} \in\prt D$ and $|Y_{\wt U_k} - Y_{s_2}|
\geq c_{15}/2$. We record this for future reference. There exists $m_1$
such that if $m\geq m_1$, $F_k \cap K_{k-1}$ holds, $J_k\geq m$ and
$2^{-m} \geq
d_k^{\beta_3}$, then
%
%e4.49 #&#
%
\begin{equation}
\label{eqd22} \sup_{\wt U_k \leq t \leq S_k, X_t \in\prt D} |Y_t - Y_{\wt U_k}|
\geq c_{15} /2.
\end{equation}

%
%f2 #&#
\begin{figure}

\includegraphics{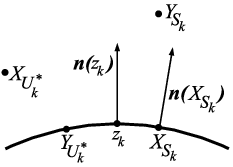}

\caption{In analysis of possible locations of $X$ and $Y$, we can
ignore Brownian oscillations because they are common to
both $X$ and $Y$ and hence do not affect their relative position.
Consider the case $\wt U_k = U^*_k$.
On the interval $[U^*_k,S_k]$, only $X$ gets ``local time push''
on $\prt D$ because $Y$ does not visit $\prt D$ between these times.
The direction of the push is always close to $\n(z_k)$. The picture
represents an impossible configuration---it is impossible
for $X_{U_k^*}$ to be ``above'' $Y_{U_k^*}$ and for $X_{S_k}$
to be ``below'' $Y_{S_k}$ if $X$ is pushed in the ``upward''
direction between times $U^*_k$ and $S_k$.}
\label{figearth1}
\end{figure}

\textit{Step} 2.3.
We will show that if $Y$ comes close to $\prt D$, then it is not likely
to hit $\prt D$ far from this point.

Assume that $C_k \cap F_k$ holds.
Recall notation related to excursions from Section~\ref{secexc}.
We will apply excursion theory to excursions of the Markov process
$(Y,X)$ from $\prt D \times\ol D$. From the intuitive point of view,
the exit system representing these excursions is equivalent to the exit
system for excursions of $Y$ from $\prt D$. We use the ``richer''
version of excursion theory so that we can discuss the relationship of
excursions of $Y$ and the process $X$. We will use the same notation
$H^z$ for excursion laws of the process $(Y,X)$ as for excursion laws
of the process $Y$ since the two families of excursion laws can be
clearly identified with each other.
All estimates of $H^z$-measures of events given in this proof hold
uniformly in $z\in\prt D$, so we will write $H^\cdot$ for such
uniform bounds.

Consider an arbitrary $c_{16} \in
(0,c_{15}/2)$, and for the moment, consider $d_k$ a fixed number. In
the following definitions, $e$ will represent excursions of the second
component of $(X,Y)$ from $\prt D$. Let
\begin{eqnarray*}
T & =& \inf\bigl\{t\geq0\dvtx \dist\bigl(e(t), \prt D\bigr) \leq
c_5 d_k, \bigl|e(0) - e(t)\bigr| \geq c_{16},
X_t \in\prt D\bigr\},
\\
\wt A&=& \Bigl\{T < \zeta, \sup_{T < t < \zeta} \bigl|e(\zeta-) - e(t)\bigr| \geq
d_k^{\beta_4}/4 \Bigr\}.
\end{eqnarray*}
An application of Lemma 3.2 of \cite{BCJ} and \eqref{eqM52}
give
%
%e4.50 #&#
%
\begin{equation}
\label{eqn2212} H^\cdot(T< \zeta) \leq H^\cdot \Bigl(\sup
_{0<t<\zeta}\bigl|e(t) - e(0)\bigr| \geq c_{16} \Bigr) \leq
c_{17}.
\end{equation}
Another application of Lemma 3.2 of \cite{BCJ} and the strong
Markov property applied at the stopping time $T$ yield
\[
H^\cdot(\wt A \mid T < \zeta) = H^\cdot \Bigl(\sup
_{T < t < \zeta}\bigl |e(\zeta-) - e(t)\bigr| \geq d_k^{\beta_4}/4
\mid T<\zeta \Bigr) \leq c_{18} d_k ^{1-\beta_4}.
\]
We combine this and \eqref{eqn2212} to see that
%
%e4.51 #&#
%
\begin{equation}
\label{eqn2213} H^\cdot(\wt A) \leq c_{19}
d_k^{1-\beta_4}.
\end{equation}
Now we go back to the original definition of $d_k$---we treat it again
as a random variable. Note that $t \to|X_t - Y_t|$ is a predictable
process, so\break $\sum_k d_k \bone_{t\in(U_{k-1}, U_k]}$ is a predictable process.
This, \eqref{eqn2213} and the exit system formula \eqref{exitsyst}
imply that the
probability that there exists an excursion of $Y$ belonging to
the set $\wt A$ and starting in the time interval $[U_{k-1}, U_k \land
\sigma'_b]$ is less than $b c_{19} d_k^{1-\beta_4}$.

Assume that $J_k = m$ for some $m\geq m_1$ such that $2^{-m} \geq
d_k^{\beta_3}$. Recall that we have assumed that $C_k \cap F_k$ holds.
Let
\begin{eqnarray*}
S_k^1 &=& \inf \bigl\{t\geq\wt U_k\dvtx
X_t \in\prt D, |Y_t - Y_{\wt U_k}| \geq
c_{16} \bigr\},
\\
S^+_k &=& \inf \bigl\{t\geq S^1_k\dvtx
|X_t - X_{S^1_k}| \geq d_k^{\beta_4}/2 \bigr
\},
\\
C_k^1 &=& \bigl\{U_k \leq S^+_k
\bigr\}.
\end{eqnarray*}
Note that $S_k^1 \leq S_k \leq U_k$ because of
\eqref{eqd22} (recall that $c_{16} < c_{15}/2$).
This implies that $S^+_k \leq S^*_k$.

If $(C^1_k)^c$ holds, then $| X_{S^+_k} - Y_{S^+_k}| \leq
c_5 d_k $, by \eqref{j2210}, \eqref{20121} and the fact that $S^+_k
\leq S^*_k$.
Under the same assumptions, we also have $| X_{S^1_k} - Y_{S^1_k}| \leq
c_5 d_k $ because $S^1_k \leq S^+_k$.
Suppose that $\eps_1>0$ is so small that for $\eps<\eps_1$ we have
$c_5 d_k <
d_k^{\beta_4}/8$. Then $| Y_{S^1_k} - Y_{S^+_k}| \geq
d_k^{\beta_4}/4$, by the definition of $S^+_k$ and the triangle inequality.

Suppose that $\wt U_k = U^*_k$.
Since $| Y_{S^1_k} - Y_{S^+_k}| \geq
d_k^{\beta_4}/4$, the excursion of $Y$ starting at $U^*_k$ belongs to
the set $\wt A$. We have proved that the
probability that there exists an excursion of $Y$ belonging to
the set $\wt A$ and starting in the time interval $[U_{k-1}, U_k \land
\sigma'_b]$ is less than $b c_{19} d_k^{1-\beta_4}$. Thus, on $A^+_{k-1}$,
\[
\P\bigl(\bigl\{\wt U_k = U^*_k\bigr\} \cap
\{J_k = m\} \cap\bigl(C^1_k
\bigr)^c \cap C_k \cap F_k \mid
\G_{k-1}\bigr) \leq c_{19} b d_k^{1-\beta_4}.
\]

Now suppose that $\wt U_k = U_{k-1}$.
If we replace $\wt U_k$ with $U_{k-1}$ in the definition of $S^1_k$,
then this random time becomes a stopping time, and we can apply the
strong Markov property at such modified $S^1_k$.
By Lemma 3.2 of \cite{BCJ} and the strong
Markov property applied at the modified $S^1_k$, the probability of $\{
\sup_{s,t\in[S^1_k, U_k]} | Y_{s} - Y_{t}| \geq
d_k^{\beta_4}/4\}$ is bounded by $c_{20} d_k^{1-\beta_4}$. Since
$(C^1_k)^c \cap C_k \cap F_k$ implies $\{| Y_{S^1_k} - Y_{S^+_k}| \geq
d_k^{\beta_4}/4\}$, we obtain on $A^+_{k-1}$,
\[
\P\bigl(\{\wt U_k = U_{k-1}\} \cap\{J_k = m\}
\cap\bigl(C^1_k\bigr)^c \cap C_k
\cap F_k \mid\G_{k-1}\bigr) \leq c_{20}
d_k^{1-\beta_4}.
\]
Combining this with the previous case yields
\[
\P\bigl( \{J_k = m\} \cap\bigl(C^1_k
\bigr)^c \cap C_k \cap F_k \mid
\G_{k-1}\bigr) \leq c_{21} d_k^{1-\beta_4}.
\]
If we apply this estimate with $m$
defined by $2^{-m} < d_k^{\beta_3} \leq2^{-m+1}$, then we obtain, on
$A^+_{k-1}$,
%
%e4.52 #&#
%
\begin{equation}
\label{eqn2215} \P\bigl(I_k^c \cap\bigl(C^1_k
\bigr)^c \cap C_k \cap F_k \mid
\G_{k-1}\bigr) \leq c_{21} d_k^{1-\beta_4}.
\end{equation}
There is no $b$ on the right-hand side of the last estimate because
$b$ was fixed, so
it can be absorbed into the constant $c_{21}$. There will be some other places
in the proof where we absorb $b$ into the constant.

\textit{Step} 2.4.
We will show that the process $Y$ is unlikely to hit the boundary close
to the point where the normal vector is parallel to the original vector
from $X$ to $Y$, assuming that $X$ is in $\prt D$ at the initial time.

It is elementary to check that if $\bv,\bw\in
\R^3$ are nonzero vectors,
then the angle $\angle(-\bv,\bv- \bw)$ is greater than the angle
$\angle(\bv, \bw)$.

Suppose that the event $ \{J_k = m\}
\cap C_k\cap C_k^1 \cap F_k$ occurred for some $m$ such that $2^{-m}
\geq d_k^{\beta_3}$. Note that $S_k^1 \leq S_k$ because of
\eqref{eqd22}. Let $\alpha_k$ be the angle between $
X_{S^1_k} - Y_{S^1_k}$ and $ X_{S_k} - Y_{S_k}$.

Suppose that $\alpha_k > 16\cdot2^{-J_k}$. We will show that this
assumption leads to a contradiction. For all
$t\in[S_k^1, S_k]$ such that $X_t \in\prt D$, the angle
between $\n(X_t)$ and $\n(X_{S_k})$ is smaller than $d_k^{\beta_4}$
because $C^1_k$ holds so $S_k \leq U_k \leq S_k^+$, and therefore, the
definition of $S_k^+$ implies that $\sup_{t\in[S_k^1, S_k]} |X_{t} -
X_{S^1_k}|
\leq d_k^{\beta_4}/2$.
It follows that the angle
between $\int_{S_k^1} ^{S_k} \n(X_t) \,dL^X_t$ and $\n(X_{S_k})$
is also smaller than $ d_k^{\beta_4}$. Since $J_k =
m$, the angle between $\n(z_k)$ and $\n(X_{S_k})$ is smaller
than or equal to $2^{-m+2}$. This is equivalent to saying that
the angle between $ Y_{S_k} - X_{S_k}$ and $\n(X_{S_k})$ is
smaller than or equal to $2^{-m+2}$. It follows that the angle
between $\int_{S_k^1} ^{S_k} \n(X_t) \,dL^X_t$ and $Y_{S_k} - X_{S_k}$ is
smaller than $2^{-m+2} + d_k^{\beta_4} \leq
2^{-m+2} + d_k^{\beta_3} \leq2^{-m+2} + 2^{-m} <
2^{-m+3}$. Note that $Y_t \notin
\prt D$ for $t\in[S_k^1, S_k]$. Thus $\int_{S_k^1} ^{S_k}
\n(Y_t) \,dL^Y_t=0$ and therefore,
%
%e4.53 #&#
%
\begin{equation}
\label{eqn2216} X_{S_k} - Y_{S_k} = X_{S_k^1} -
Y_{S_k^1} + \int_{S_k^1} ^{S_k} \n
(X_t) \,dL^X_t.
\end{equation}
We will identify some elements of the above formula with vectors $\bv$
and $\bw$ in the opening remark in this step, namely, $\bv= X_{S_k} -
Y_{S_k}$ and $\bw= X_{S_k^1} - Y_{S_k^1}$. Then $ \int_{S_k^1} ^{S_k}
\n(X_t) \,dL^X_t = \bv- \bw$ and
$\alpha_k = \angle(\bv, \bw) = \angle( X_{S^1_k} - Y_{S^1_k},
X_{S_k} -
Y_{S_k}) > 16\cdot2^{-J_k} = 2^{-m+4}$. This and the fact
that $\angle(-\bv,\bv- \bw)= \angle( Y_{S_k} - X_{S_k}, \int_{S_k^1}
^{S_k} \n(X_t) \,dL^X_t) < 2^{-m+3}$ yield a
contradiction. Hence we must have $\alpha_k \leq2^{-J_k+4}$ if $ \{
J_k = m\} \cap
C_k\cap C_k^1 \cap F_k$ holds.

If $C^1_k$ occurred, then $\sup_{S^1_k \leq t \leq U_k} |X_{t} -
X_{S^1_k}| \leq d_k^{\beta_4}/2$. Since $S_k \in[S^1_k, U_k]$, it
follows that $|X_{U_k} - X_{S_k}| \leq d_k^{\beta_4}$.
Recall that $|X_{U_k} - Y_{U_k}| \leq c_5 d_k$ by \eqref{eqd21}.
This
implies that, for small $\eps_1>0$,
%
%e4.54 #&#
%
\begin{equation}
\label{j261}\quad  |Y_{U_k} - X_{S_k}| \leq|X_{U_k} -
X_{S_k}| + |X_{U_k} - Y_{U_k}| \leq
d_k^{\beta_4} + c_5 d_k \leq2
d_k^{\beta_4}.
\end{equation}
Let $z^1_k\in\prt D$ be defined by $\bn(z^1_k) = (Y_{S^1_k} -
X_{S^1_k})/ |Y_{S^1_k} - X_{S^1_k}|$. Since
$\alpha_k \leq2^{-J_k+4}$, we have $|z_k - z^1_k|
\leq c_{22} 2^{-J_k} = c_{22} 2^{-m}$. We have assumed that $2^{-m}
\geq d_k^{\beta_3}$, so \eqref{j261} implies that
\begin{eqnarray*}
\bigl|Y_{U_k} - z^1_k\bigr| &\leq& |Y_{U_k} -
X_{S_k}| + |X_{S_k} - z_k| + \bigl|z_k -
z^1_k\bigr|
\\
&\leq&2 d_k^{\beta_4} + 2^{-J_k +1} + c_{22}
2^{-m} = 2 d_k^{\beta_4} + 2^{-m +1} +
c_{22} 2^{-m} \leq c_{23} 2^{-m}.
\end{eqnarray*}
We have shown that the event $ \{J_k = m\} \cap C_k\cap C_k^1 \cap F_k$ implies
%
%e4.55 #&#
%
\begin{equation}
\label{j262} \bigl\{\bigl|Y_{U_k} - z^1_k\bigr| \leq
c_{23} 2^{-m}\bigr\}.
\end{equation}

Suppose that $\wt U_k = U^*_k$.
In this case, we will estimate the probability of the event in \eqref
{j262} using excursion theory.
Recall the remarks and conventions from the beginning of step 2.3. Let
$T^1 = \inf \{t\geq0\dvtx |e(t) - e(0)| \geq
c_{16}  \}$, $z^2 \in\prt D$ be the point such that $\bn(z^2) =
(e(T^1) - X_{T^1})/ |e(T^1) - X_{T^1}|$
and
\[
\wh A = \bigl\{e\dvtx T^1 < \zeta,\bigl |e(\zeta-) - z^2\bigr|
\leq c_{23} 2^{-m} \bigr\}.
\]
The number of excursions starting before $\sigma'_b$ and
such that $T^1 < \zeta$ is Poisson with the mean bounded
by $c_{24}b$, by \eqref{exitsyst} and the right-hand side of \eqref
{eqn2212}.
We can assume that $c_{16}>0$ is
arbitrarily small. If $c_{16}$ is sufficiently small, then it is
easy to see that the angle between $e(T^1) - X_{T^1}$ and
$\n(e(0))$ must be bounded below by a strictly positive
constant, and therefore the distance between $z^2$ and
$e(T^1)$ must be bounded below by $c_{25}>0$.
By the strong Markov property applied at
$T^1$, given the values of
$e(T^1)$ and $z^2$ and assuming that $|e(T^1)- z^2|
\geq c_{25}$, the probability that $e(\zeta-) \in\B(z^2,
c_{23}2^{-m})$ is smaller than $c_{26} 2^{-m}$, by \eqref{j263}.
Hence the expected number of excursions in $\wh A$ starting before
$\sigma'_b$ is bounded by $c_{26}b 2^{-m}$.
This implies that the probability that such an excursion will
occur is less than or equal to $c_{26} b 2^{-m}$. We have shown that $
\{J_k = m\} \cap C_k\cap C_k^1 \cap F_k$
implies $ \{|Y_{U_k} - z^1_k| \leq c_{23} 2^{-m}\}$, so if $\{\wt U_k =
U^*_k\} \cap\{J_k = m\} \cap C_k\cap C_k^1 \cap F_k$ occurs, then the
excursion of $Y$ starting at $U_k^*$
belongs to $\wh A$.
We conclude that, on $A^+_{k-1}$,
%
%e4.56 #&#
%
\begin{equation}
\label{201210} \P\bigl(\bigl\{\wt U_k = U^*_k\bigr\} \cap
\{J_k =m\} \cap C_k\cap C_k^1
\cap F_k \mid \G_{k-1}\bigr) \leq c_{26} b
2^{-m}.
\end{equation}

Next suppose that $\{\wt U_k = U_{k-1}\}\cap K_{k-1}$ holds. It is easy
to see that if $c_{16} >0$ is sufficiently small, then we can find
$\eps
_1>0$ such that for $\eps< \eps_1$, the angle between $Y_{S^1_k} -
X_{S^1_k}$ and
$\n(X_{S^1_k})$ is bounded below by a strictly positive constant.
Then the distance between $z^1_k$ and
$Y_{S^1_k}$ is bounded below by $c_{26}>0$.
By the strong Markov property applied at
$S^1_k$, given the values of
$Y_{S^1_k}$ and $z^1_k$ and assuming that $|Y_{S^1_k}- z^1_k|
\geq c_{26}$, the probability that $Y_{U_k} \in\B(z^1_k,
c_{23}2^{-m})$ is smaller than $c_{27} 2^{-m}$, by~\eqref{j263}.
We have shown that $\{J_k = m\} \cap C_k\cap C_k^1 \cap F_k$
implies $ \{|Y_{U_k} - z^1_k| \leq c_{23} 2^{-m}\}$ so, on $A^+_{k-1}$,
\[
\P\bigl(\{\wt U_k = U_{k-1}\} \cap\{J_k =m\}
\cap C_k\cap C_k^1 \cap F_k \mid
\G_{k-1}\bigr) \leq c_{27} 2^{-m}.
\]
We combine this with \eqref{201210} to see that
\[
\P\bigl( \{J_k =m\} \cap C_k\cap C_k^1
\cap F_k \mid\G_{k-1}\bigr) \leq c_{28}
2^{-m}.
\]
Summing over $m\geq m'$, we obtain, on $A^+_{k-1}$,
%
%e4.57 #&#
%
\begin{equation}
\label{eqn2217} \P\bigl( \bigl\{J_k \geq m'\bigr\} \cap
C_k\cap C_k^1 \cap F_k \mid
\G_{k-1}\bigr) \leq c_{29} 2^{-m'}.
\end{equation}
We combine \eqref{eqn211},
\eqref{eqn2215} and \eqref{eqn2217} to see that
if $2^{-m} \geq d_k^{\beta_3}$, then on $A^+_{k-1}$,
\begin{eqnarray*}
\P\bigl(\{J_k \geq m\} \cap F_k \mid\G_{k-1}
\bigr) &\leq&\P\bigl( C^c_k \cap F_k \mid
\G_{k-1}\bigr) + \P\bigl(I_k^c \cap
\bigl(C^1_k\bigr)^c \cap C_k
\cap F_k \mid\G_{k-1}\bigr)
\\
&&{} + \P\bigl( \{J_k \geq m\} \cap C_k\cap
C_k^1 \cap F_k \mid\G_{k-1}\bigr)
\\
&\leq& c_{13} d_k^{1-\beta_4} + c_{21}
d_k^{1-\beta_4} +c_{29} 2^{-m} \leq
c_{30} 2^{-m};
\end{eqnarray*}
that is, \eqref{eqn188} holds.

The following follows from \eqref{eqn2217}, with $m'$
defined by $2^{-m'-1} \leq d_k^{\beta_3} < 2^{-m'}$. We have on $A^+_{k-1}$,
%
%e4.58 #&#
%
\begin{equation}
\label{eqn223} \P\bigl( I_k^c \cap C_k\cap
C_k^1 \cap F_k \mid\G_{k-1}\bigr)
\leq c_{29} d_k^{\beta_3}.
\end{equation}

\textit{Step} 2.5.
We will show that the vector from $X$ to $Y$ is very likely to be
almost parallel to $\prt D$ at the time $U_k$.

Assume that $C_k\cap F_k$ holds. Let
\begin{eqnarray*}
\wh S^j_k &=& \inf\bigl\{t \geq S_k\dvtx
|X_t - Y_t| \leq2^{-j}\bigr\},
\\
\wh U^j_k &=& \inf\bigl\{t\geq\wh S^j_k
\dvtx |X_t - X_{\wh S^j_k}| \geq2^{-j \beta_4}\bigr\},
\\
\wh C^j_k &=& \bigl\{U_k \leq\wh
U^j_k \bigr\}.
\end{eqnarray*}

The following argument is very similar to that in step 2.1.
By the definition of $\wh S^j_k$, for large $j$,
%
%e4.59 #&#
%
\begin{equation}
\label{j2710} |X_{\wh S^j_k} - Y_{\wh S^j_k}| = 2^{-j}.
\end{equation}
Suppose that $\{\wh S^j_k \leq U_k\} \cap(\wh C^j_k )^c$ holds. Then
$X_{\wh S^j_k} \in\prt D$ and
$\dist(Y_{\wh S^j_k}, \prt D) \leq2^{-j} $.
By Lemma 3.2 of \cite{BCJ},
%
%e4.60 #&#
%
\begin{equation}
\label{j274} \P \Bigl( \sup_{\wh S^j_k\leq t \leq U_k} |Y_{\wh S^j_k} -
Y_{t}| \geq2 ^{-j \beta_4} /3 \Bigr) \leq c_{31}
2^{-j(1-\beta_4)}.
\end{equation}
It follows from \eqref{j272} that
for all $t\in[\wh S^j_k, U_k]$,
%
%e4.61 #&#
%
\begin{equation}
\label{j275} |X_{t} - Y_{t}| \leq c_5
|X_{\wh S^j_k} - Y_{\wh S^j_k}|.
\end{equation}
In particular, for large $j$, $|X_{\wh U^j_k} - Y_{\wh U^j_k}| \leq c_5
2^{-j} < 2^{-j \beta_4}/3$.
This, \eqref{j2710} and the definitions of $\wh S^j_k$ and $\wh C^j_k$
imply that, assuming that $\wh C^j_k$ does not hold,
$ |Y_{\wh S^j_k} - Y_{\wh U^j_k }| \geq2^{-j \beta_4}/3$.
This and \eqref{j274} imply that, on $A^+_{k-1}$,
%
%e4.62 #&#
%
\begin{equation}
\label{eqd55} \P\bigl( \bigl(\wh C_k^j
\bigr)^c \cap\bigl\{\wh S^j_k \leq
U_k\bigr\} \cap C_k \cap F_k\mid
\G_{k-1}\bigr) \leq c_{32} 2^{-j(1-\beta_4)}.
\end{equation}

Assume that $\wh C_k^j \cap\{\wh S^j_k \leq U_k\}$ holds.
Since
$X_{\wh S^{j}_k} \in\prt D$ and $Y_{U_k} \in\prt D$, there
is $t\in[\wh S^{j}_k, U_k]$ such that $\dist(X_t, \prt D) =
\dist(Y_t, \prt D)$. Let $\wt S^{j}_k$ be the smallest $t\geq
\wh S^{j}_k$ with this property. Let $\wt z\in\prt D$ be the point
closest to $X_{\wt S^{j}_k}$ among all points equidistant from $X_{\wt
S^{j}_k}$ and $Y_{\wt S^{j}_k}$.
By the definition of $\wh C_k^j$, for all
$t\in[\wh S^j_k,U_k]$, we have $ |X_t - X_{\wh S^j_k}|
\leq2^{-j \beta_4}$. By \eqref{j275},
for all
$t\in[\wh S^j_k,U_k]$, we have $ |X_t - Y_{t}|
\leq c_5 2^{-j}$. This implies that, for large $j$, $ |\wt z - X_{\wt
S^j_k}| = |\wt z - Y_{\wt S^j_k}| \leq10\cdot2^{-j \beta_4}$. We
also have
for $t\in[\wt S^j_k,U_k]$,
$ |\wt z - X_{t}| \leq20\cdot2^{-j \beta_4}$
and
$ |\wt z - Y_{t}| \leq20\cdot2^{-j \beta_4}$.
Hence
we can apply Lemma \ref{lemn31} with $c_1/4 =
20\cdot2^{-j \beta_4}$ at the stopping time $\wt S^j_k$ to see that
%
%e4.63 #&#
%
\begin{equation}
\label{j2711} \bigl\llvert \< X_{U_k}- Y_{U_k}, \n (\wt z )
\> \bigr\rrvert \leq80\cdot2^{-j \beta_4} \llvert X_{U_k}-
Y_{U_k} \rrvert.
\end{equation}
Since
$ |\wt z - Y_{U_k}| \leq20\cdot2^{-j \beta_4}$, the angle between
$\n
(\wt z)$ and $\n(Y_{U_k})$ is less than $40\cdot2^{-j \beta_4}$ for
large $j$. This and
\eqref{j2711} imply that, for large $j$,
%
%e4.64 #&#
%
\begin{equation}
\label{j2712} \bigl\llvert \bigl\< X_{U_k}- Y_{U_k}, \n
(Y_{U_k} ) \bigr\> \bigr\rrvert \leq200\cdot2^{-j \beta_4} \llvert
X_{U_k}- Y_{U_k} \rrvert.
\end{equation}
Let $j_0$ be the largest $j$ such that
$\wh S^j_k \leq U_k$. Then $| X_{U_k}- Y_{U_k}| \geq2^{-j_0-1}$, and
if the event in \eqref{j2712} holds with $j=j_0$, then the following
event holds:
\[
K_k = \biggl\{\frac{ \llvert   \< X_{U_k}- Y_{U_k},
\n (Y_{U_k} ) \> \rrvert  } {
\llvert  X_{U_k}- Y_{U_k} \rrvert } \leq c_2
|X_{U_{k}} - Y_{U_{k}}|^{\beta_4} \biggr\},
\]
with $c_2 = 200\cdot2 ^{\beta_4}$.
It follows from the definitions of $S_k$ and $\wh S^j_k$ that $2^{-j_0}
\leq2c_8 d_k$.
Thus
\[
K_k^c \cap C_k \cap F_k \subset
\bigcup_{j\dvtx 2^{-j} \leq2 c_8 d_k} \bigl(\wh C_k^j
\bigr)^c \cap\bigl\{\wh S^j_k \leq
U_k\bigr\} \cap C_k \cap F_k.
\]
This and
\eqref{eqd55} imply that, on $A^+_{k-1}$,
%
%e4.65 #&#
%
\begin{eqnarray}
\label{eqn214} &&\P\bigl( K_k^c \cap C_k \cap
F_k\mid\G_{k-1}\bigr)\nonumber\\
&&\qquad \leq\P \biggl( \bigcup
_{j\dvtx 2^{-j} \leq2 c_8 d_k} \bigl(\wh C_k^j
\bigr)^c \cap\bigl\{\wh S^j_k \leq
U_k\bigr\} \cap C_k \cap F_k \mid
\G_{k-1} \biggr)
\\
&&\qquad\leq \sum_{j\dvtx 2^{-j} \leq2 c_8 d_k} c_{32}
2^{-j(1-\beta_4)} \leq c_{33} d_k^{1-\beta_4}.
\nonumber
\end{eqnarray}

\textit{Step} 2.6.
We will find a lower bound for the distance from $X$ to $Y$ at the time~$U_k$.

Suppose that $I_k \cap C_k \cap F_k$ holds. Recall that
$\beta_4> \beta_3$.
Since $I_k$ holds,
we have $2^{-J_k} \geq d_k^{\beta_3}$. Assume for now that $d_k^{\beta
_3} \leq\eta$, where $\eta>0$ is so small that
$d_k^{\beta_4} < (1/100) \land d_k^{\beta_3}/(4\pi) \leq(1/\pi)
2^{-J_k-1}$.
Since $C_k$ is assumed to hold, we have
$|X_t - X_{S_k}| \leq d_k^{\beta_4}$
for
all $t\in[S_k, U_k]$ such that $X_t \in\prt D$, and therefore, for
such~$t$, the angle
between $\n(X_t)$ and $\n(X_{S_k})$ is smaller than $\pi
d_k^{\beta_4}$. It follows that the angle between $\int_{S_k}
^{U_k} \n(X_t) \,dL^X_t$ and $\n(X_{S_k})$ is also smaller than
$\pi d_k^{\beta_4}$. The angle between $\n(X_{S_k})$ and
$\n(z_k)$ is greater than $2^{-J_k}$. This implies that the
angle between $\int_{S_k} ^{U_k} \n(X_t) \,dL^X_t$ and $Y_{S_k} -
X_{S_k}$, which is the same as
the
angle between $\int_{S_k} ^{U_k} \n(X_t) \,dL^X_t$ and $\n(z_k)$,
is greater than $2^{-J_k} - \pi d_k^{\beta_4} > 2^{-J_k} - 2^{-J_k-1} =
2^{-J_k-1}$. Note that $Y_t \notin\prt D$
for $t\in[S_k, U_k]$ by the definition of $U_k$. Thus
$\int_{S_k} ^{U_k} \n(Y_t) \,dL^Y_t=0$ and, therefore,
%
%e4.66 #&#
%
\begin{equation}
\label{eqn215} X_{U_k} - Y_{U_k} = X_{S_k} -
Y_{S_k} + \int_{S_k} ^{U_k}
\n(X_t) \,dL^X_t.
\end{equation}
If $\bv,\bw\in\R^3$ are nonzero vectors
and the angle $\angle(\bv, \bw)$ is greater than $\alpha$ then the
length of $\bv- \bw$ is at least $|\bw| \sin\alpha$. In view of
\eqref
{eqn215}, we can apply this observation to $\bw=
Y_{S_k} - X_{S_k}$ and $\bv= \int_{S_k} ^{U_k} \n(X_t) \,dL^X_t$, and
conclude that
$|X_{U_k}
- Y_{U_k}| \geq c_{34} 2^{-J_k} |X_{S_k} - Y_{S_k}| = c_{34}
2^{-J_k} c_8 d_k$. We now specify the value of the constant in the
definition of $G_k$ to be $c_9 = c_{34}c_8$. With this definition of
$G_k$, we see that we have shown that $G_k$ holds.
Hence, assuming that
$d_k^{\beta_4} < (1/100) \land d_k^{\beta_3}/(4\pi)$, we have on $A^+_{k-1}$,
%
%e4.67 #&#
%
\begin{equation}
\label{eqn216} \P\bigl( G_k^c \cap I_k \cap
C_k \cap F_k \mid\G_{k-1}\bigr) =0.
\end{equation}
Since $\inf\{a\geq0\dvtx a^{\beta_4} \geq a^{\beta_3}/(4\pi)\}>0$ and the
probability of any event is bounded by 1,
we have for some constant $c_{35}<\infty$, on the event
$\{d_k^{\beta_4} \geq(1/100) \land d_k^{\beta_3}/(4\pi)\} \cap A^+_{k-1}$,
%
%e4.68 #&#
%
\begin{equation}
\label{a11} \P\bigl( G_k^c \cap I_k \cap
C_k \cap F_k \mid\G_{k-1}\bigr) \leq
c_{35} d_k^{\beta_3}.
\end{equation}
In view of \eqref{eqn216}, we see that \eqref{a11} holds on $A^+_{k-1}$.

It follows from \eqref{eqn211},
\eqref{eqn2215}, \eqref{eqn223}, \eqref{eqn214} and
\eqref{a11} that on $A^+_{k-1}$,
\begin{eqnarray*}
\P\bigl(A^c_k \cap F_k \mid
\G_{k-1}\bigr) &=&\P\bigl(\bigl(I^c_k \cup
C^c_k \cup G^c_k \cup
K^c_k\bigr) \cap F_k \mid\G _{k-1}
\bigr)
\\
&\leq& \P\bigl( C_k^c \cap F_k \mid
\G_{k-1}\bigr) + \P\bigl(I_k^c \cap
\bigl(C_k^1\bigr)^c \cap C_k
\cap F_k \mid\G_{k-1}\bigr)
\\
&&{} + \P\bigl( I_k^c \cap C^1_k
\cap C_k \cap F_k \mid\G_{k-1}\bigr) + \P\bigl(
K_k^c \cap C_k \cap F_k \mid
\G_{k-1}\bigr)
\\
&&{} + \P\bigl( G_k^c \cap I_k \cap
C_k \cap F_k \mid\G_{k-1}\bigr)
\\
& \leq& c_{13} d_k^{1-\beta_4} + c_{21}
d_k^{1-\beta_4} +c_{29} d_k^{\beta_3}
+ c_{33} d_k^{1-\beta_4} + c_{35}
d_k^{\beta_3} \\
&\leq& c_{36} d_k^{\beta_3}.
\end{eqnarray*}
This completes the proof of \eqref{eqn1120}.

\textit{Step} 3.
The last step of the proof combines the estimates obtained above.
Although this part of the proof looks complicated, its beginning
consists mostly of elementary combinatorial arguments. The second part
is a more or less straightforward
translation of the earlier estimates into the language of distributions
and stochastic domination.

If $ F_k$ holds, then \eqref{a31} shows that
$\sup_{t\in[S_k, \sigma'_b]} \llvert Y_{t} - X_{t}\rrvert  < d_k/2$.
It follows that if $F_k \cap F_{k+1}$ holds, then
$U_k \in[S_k, \sigma'_b]$ and, therefore,
\[
d_{k+1} = |X_{U_k} - Y_{U_k}| \leq\sup
_{t\in[S_k, \sigma'_b]} \llvert Y_{t} - X_{t}\rrvert <
d_k/2.
\]
Hence if the event $\bigcap_{j\leq k-1} F_j $ occurred, then $d_k \leq
d_0 2^{-k+1} = \eps2^{-k+1}$.
This, the fact that $ A^+_{k-1} \subset\bigcap_{j\leq k-1} F_j $ and
\eqref{eqn1120} imply that
%
%e4.69 #&#
%
\begin{equation}
\label{eqn183} \P \bigl( A^c_{k} \cap F_k
\cap A^+_{k-1} \bigr) \leq c_{10} d_k^{\beta_3}
\leq c_{10} \eps^{\beta_3} 2^{-(k-1)\beta_3}.
\end{equation}

Let
\[
F = F_1^c \cup\bigcup_{k=1}^\infty
\bigl(F_k \cap F_{k+1}^c \cap
A^+_{k} \bigr).
\]
If $\bigcap_{k=1}^\infty F_k$ holds, then
%
%e4.70 #&#
%
\begin{equation}
\label{a141} \liminf_{k\to\infty}|X_{U_k} -
Y_{U_k}| =\liminf_{k\to\infty} d_{k-1} \leq\lim
_{k\to\infty} \eps2^{-k-2} =0,
\end{equation}
so
$\inf_{0\leq t \leq
\sigma'_b } |X_t - Y_t| \leq\lim\inf_{k\to\infty}|X_{U_k} - Y_{U_k}|
= 0$. The last event has probability
0, according to Lemma \ref{lemo311}, so $\P ( \bigcap_{k=1}^\infty F_k  ) =0$. Since $F_{k+1} \subset F_k$,
there exists at most one $N_1$ such that $F_{N_1}^c \cup
F_{N_1+1}$ fails (in other words, $F_{N_1} \cap
F_{N_1+1}^c$ holds). We will write $\bigcap_{k=1}^\infty F_k = \{N_1
=\infty\}$ so $\P(N_1 = \infty) =0$. There exists at most one $N_2$
such that
$A_j$ holds for all $j<N_2$, and $A_{N_2}$ does not hold. Using these
definitions of $N_1$ and $N_2$, and \eqref{eqn183}, we obtain
\begin{eqnarray*}
\P\bigl(F^c\bigr) &=& \P \Biggl( F_1 \cap\bigcap
_{k=1}^\infty \biggl(F_k^c
\cup F_{k+1} \cup \bigcup_{j\leq k}
A_j^c \biggr) \Biggr)
\\
&\leq&\P \bigl( F_1 \cap\{N_1 =\infty\} \bigr) + \P
\Biggl( \bigcap_{k=1}^\infty
\biggl(F_k^c \cup F_{k+1}\cup \bigcup
_{j\leq k} A_j^c \biggr)\cap\{
N_1<\infty\} \Biggr)
\\
&=& 0 + \P \Biggl( \bigcap_{k=1}^\infty
\biggl(F_k^c \cup F_{k+1}\cup \bigcup
_{j\leq k} A_j^c \biggr)\cap
\{N_1<\infty\} \Biggr)
\\
&\leq& \P \Biggl( \bigcup_{n=1}^\infty
\biggl( \biggl(F_{n}^c \cup F_{n+1}\cup \bigcup
_{j\leq n} A_j^c \biggr) \cap
\{N_1 = n\} \biggr) \Biggr)
\\
&= &\P \Biggl( \bigcup_{n=1}^\infty \biggl(
\biggl( \bigcup_{j\leq n} A_j^c
\biggr) \cap\{N_1 = n\} \biggr) \Biggr)
\\
&=& \P \Biggl( \bigcup_{n=1}^\infty\bigcup
_{m=1}^n \biggl( \biggl( \bigcup
_{j\leq m} A_j^c \biggr) \cap
\{N_1 = n, N_2 =m\} \biggr) \Biggr)
\\
&=& \P \Biggl( \bigcup_{m=1}^\infty \biggl(
\biggl( \bigcup_{j\leq m} A_j^c
\biggr) \cap\{N_1 \geq m, N_2 =m\} \biggr) \Biggr)
\\
&\leq& \P \Biggl( \bigcup_{m=1}^\infty
\bigl( A^+_{m-1} \cap A_m^c \cap
F_m \bigr) \Biggr)
\\
&\leq&\sum_{m=1}^\infty \P \bigl(
A^+_{m-1} \cap A_m^c \cap F_m
\bigr)
\\
&\leq&\sum_{m=1}^\infty c_{10}
\eps^{\beta_3} 2^{-(m-1)\beta_3} \leq c_{37} \eps^{\beta_3}.
\end{eqnarray*}
This proves \eqref{eqn201}.

Since $F_{k+1} \subset F_k$ and $ (F_{k+1}\cap
A^+_{k+1} ) \subset (F_{k}\cap
A^+_{k} )$, we have
%
%e4.71 #&#
%
\begin{eqnarray}
\label{eqd121} F &=& F_1^c \cup\bigcup
_{n=1}^\infty \bigl(F_n \cap
F_{n+1}^c \cap A^+_{n} \bigr)
\nonumber
\\
&=&F_1^c \cup\bigcup_{n=1}^{k-2}
\bigl(F_n \cap F_{n+1}^c \cap
A^+_{n} \bigr) \cup \bigl(F_{k-1} \cap F_{k}^c
\cap A^+_{k-1} \bigr)
\nonumber
\\
&&{} \cup \bigcup_{n=k}^\infty
\bigl(F_n \cap F_{n+1}^c \cap
A^+_{n} \bigr)
\\
&&{}\subset F_1^c \cup\bigcup
_{n=1}^{k-2} F_{n+1}^c \cup
\bigl(F_{k-1} \cap F_{k}^c \cap
A^+_{k-1} \bigr) \cup \bigcup_{n=k}^\infty
\bigl(F_n \cap A^+_{n} \bigr)
\nonumber
\\
&&{} \subset F^c_{k-1} \cup\bigl( F_{k-1} \cap
F_{k}^c \cap A^+_{k-1}\bigr) \cup
\bigl(F_k\cap A^+_{k} \bigr).\nonumber
\end{eqnarray}

Let $T_k = U_k \land\sigma'_b$. We make the following three claims:
%
%e4.72 #&#
%
\begin{eqnarray}
\label{eqn1810} &&\bigl|\log|X_{T_k}- Y_{T_k}| -
\log|X_{T_{k-1}} - Y_{T_{k-1}}| \bigr|
\nonumber
\\[-8pt]
\\[-8pt]
\nonumber
&&\qquad{} \cases{ %
 =0, & \quad$\mbox{if $F^c_{k-1}$ holds;}$
\vspace*{2pt}\cr
\leq c_{38}:= c_7 \lor\log c_5, &\quad
$\mbox{if $F_{k-1} \cap F_{k}^c \cap
A^+_{k-1}$ holds;}$
\vspace*{2pt}\cr
\leq c_{39} m, & \quad$\mbox{if $\{J_k =m\} \cap F_k
\cap A^+_{k}$ holds.}$ }
\end{eqnarray}
The first claim follows from the definitions of $T_{k-1}$, $S_{k-1}$
and $F_{k-1}$. The second claim follows from the definition of
$S_k$ and \eqref{eqd21} applied with $T = U_{k-1}$. The last claim
follows from the fact that $G_k \subset
A_k$.

If $ A_k$ holds, then $K_k$ holds. Then condition \eqref{eqj201} is
satisfied with $x_0 = X_{U_k}$ and $y_0 = Y_{U_k}$. By the strong
Markov property applied at the stopping time $U_k$, we obtain a formula
analogous to
\eqref{eqn1114} which implies that
\[
\P\bigl(F_{k} \cap A^+_{k} \mid\G_{k-1}\bigr)
\leq \P(F_{k} \mid\G_{k-1}) \leq p_1
\]
on $A^+_{k-1}$. By the repeated
application of the strong Markov property at $U_1, U_2, \ldots,$ we obtain
%
%e4.73 #&#
%
\begin{equation}
\label{a121} \P\bigl(F_{k} \cap A^+_{k}\bigr) \leq
p_1^{k}.
\end{equation}
This and the second claim in \eqref{eqn1810} imply that
%
%e4.74 #&#
%e4.75 #&#
%e4.76 #&#
%
\begin{eqnarray}
\label{a122}\qquad  &&\P \bigl( \bigl|\log|X_{T_k}- Y_{T_k}| -
\log|X_{T_{k-1}} - Y_{T_{k-1}}| \bigr| \bone_{F_{k-1} \cap F_{k}^c \cap A^+_{k-1}} >
c_{38} \bigr) =0,
\\
\label{a145}&&\P\bigl( \bigl|\log|X_{T_k}- Y_{T_k}| -\log|X_{T_{k-1}} -
Y_{T_{k-1}}| \bigr| \bone_{F_{k-1} \cap F_{k}^c \cap A^+_{k-1}} \in(0, c_{38}]
\bigr)
\nonumber
\\[-8pt]
\\[-8pt]
\nonumber
&&\qquad\leq
p_1^{k-1},
\\
\label{a146}&&\P \bigl( \bigl|\log|X_{T_k}- Y_{T_k}| -\log|X_{T_{k-1}} -
Y_{T_{k-1}}| \bigr| \bone_{F_{k-1} \cap F_{k}^c \cap A^+_{k-1}} =0 \bigr)
\nonumber
\\[-8pt]
\\[-8pt]
\nonumber
&&\qquad\geq1- p_1^{k-1}.
\end{eqnarray}
Since $D$ is bounded, there exists $m_0 > -\infty$ such that $J_k \geq
m_0$, a.s.
It follows from~\eqref{eqn188} that on $A^+_{k-1}$,
\[
\P \bigl(\{J_k \geq m\} \cap F_k \cap A_k
\mid\G_{k-1} \bigr) \leq c_{40} 2^{-m},
\]
so we obtain for $m\geq m_0$, using \eqref{a121} and the third claim
in \eqref{eqn1810},
%
%e4.77 #&#
%e4.78 #&#
%e4.79 #&#
%
\begin{eqnarray}
\label{a123}\quad \qquad&& \P \bigl(\bigl\llvert \log|X_{T_k}- Y_{T_k}| -
\log|X_{T_{k-1}} - Y_{T_{k-1}}| \bigr\rrvert \bone_{\{J_k =m\} \cap F_k\cap A^+_{k}} >
c_{39} m \bigr) =0,
\\
\label{a147}&&\P\bigl (\bigl\llvert \log|X_{T_k}- Y_{T_k}| -\log|X_{T_{k-1}}
- Y_{T_{k-1}}| \bigr\rrvert \bone_{\{J_k =m\} \cap F_k\cap A^+_{k}} \in(0, c_{39} m] \bigr)
\nonumber
\\[-8pt]
\\[-8pt]
\nonumber
&&\qquad
\leq p_1^{k-1}c_{40} 2^{-m},
\\
\label{a148}&&\P \bigl(\bigl\llvert \log|X_{T_k}- Y_{T_k}| -
\log|X_{T_{k-1}} - Y_{T_{k-1}}| \bigr\rrvert \bone_{\{J_k =m\} \cap F_k\cap A^+_{k}} =0
\bigr)
\nonumber
\\[-8pt]
\\[-8pt]
\nonumber
&&\qquad\geq 1 - p_1^{k-1}c_{40} 2^{-m}.
\end{eqnarray}

Recall from the paragraph following \eqref{a141} that only a finite
number of events $F_k$, $k\geq1$, hold, a.s. Hence, for some random
$k_0<\infty$ and all $k\geq k_0$, we have $U_k = \sigma'_b$. It follows
that $T_n = T_{k_0}= \sigma'_b$ for all $n\geq k_0$, and therefore,
\[
|V_1 - V_0| = \Biggl\llvert \sum
_{k=1}^\infty \log|X_{T_k}-
Y_{T_k}| -\log|X_{T_{k-1}} - Y_{T_{k-1}}| \Biggr\rrvert.
\]
This, \eqref{eqd121} and the first claim in \eqref{eqn1810} imply that
%
%e4.80 #&#
%
\begin{eqnarray}
\label{a149} |V_1 - V_0| \bone_F & =&
\Biggl\llvert \sum_{k=1}^\infty
\log|X_{T_k}- Y_{T_k}| -\log|X_{T_{k-1}} -
Y_{T_{k-1}}| \Biggr\rrvert \bone_F
\nonumber\\
& \leq& \sum_{k=1}^\infty \bigl\llvert
\log|X_{T_k}- Y_{T_k}| -\log|X_{T_{k-1}} -
Y_{T_{k-1}}|\bigr \rrvert \bone_{F^c_{k-1}}
\nonumber
\\
&&{} + \sum_{k=1}^\infty \bigl\llvert
\log|X_{T_k}- Y_{T_k}| -\log|X_{T_{k-1}} -
Y_{T_{k-1}}|\bigr \rrvert \bone_{F_{k-1} \cap F_{k}^c \cap A^+_{k-1}}
\nonumber
\\[-8pt]
\\[-8pt]
\nonumber
&&{} + \sum_{k=1}^\infty \bigl\llvert
\log|X_{T_k}- Y_{T_k}| -\log|X_{T_{k-1}} -
Y_{T_{k-1}}| \bigr\rrvert \bone_{F_k\cap A^+_{k}}
\\
& =& \sum_{k=1}^\infty\bigl \llvert
\log|X_{T_k}- Y_{T_k}| -\log|X_{T_{k-1}} -
Y_{T_{k-1}}| \bigr\rrvert \bone_{F_{k-1} \cap F_{k}^c \cap A^+_{k-1}}
\nonumber
\\
&&{} + \sum_{k=1}^\infty\sum
_{m\geq m_0} \bigl\llvert \log|X_{T_k}- Y_{T_k}| -
\log|X_{T_{k-1}} - Y_{T_{k-1}}| \bigr\rrvert \bone_{ \{J_k =m\} \cap F_k\cap A^+_{k}}.
\nonumber
\end{eqnarray}

Let $k_0$ be such that $p_1^{k-1} + p_1^{k-1} \sum_{m\geq m_0}
c_{40} 2^{-m} \leq1$ for $k\geq k_0$, and let $m_1$ be such
that $\sum_{m\geq m_1} c_{40} 2^{-m} \leq1$. Let $q'\geq0$ be
such that $q' + \sum_{ m\geq m_1} c_{40} 2^{-m} =1$, and let
$q_k\geq0$ be such that $q_k +p_1^{k-1} + p_1^{k-1}
\sum_{m\geq m_0} c_{40} 2^{-m} =1$, for $k\geq k_0$. Let $Z_k$,
$k\geq1$, be independent random variables with the following
distributions; for $1\leq k \leq k_0 -1$,
\[
Z_k =\cases{ c_{38} +
c_{39} m_1, & \quad$\mbox{with probability $q'
$;}$
\vspace*{2pt}\cr
c_{39} m, &\quad $\mbox{with probability $ c_{40}
2^{-m}$ for $m\geq m_1$,}$}
\]
and for $k\geq k_0$,
\[
Z_k =\cases{ %
0, & \quad$\mbox{with
probability $q_k$;}$
\vspace*{2pt}\cr
c_{38}, & \quad$\mbox{with probability $p_1^{k-1}
$;}$
\vspace*{2pt}\cr
c_{39} m, & \quad$\mbox{with probability $c_{40}
p_1^{k-1} 2^{-m}$ for $m\geq m_0$.}$}
\]

By \eqref{a122}--\eqref{a146}, \eqref{a123}--\eqref{a148} and
\eqref
{a149}, the
random variable
\[
|V_1 - V_0| \bone_F = \Biggl\llvert \sum
_{k=1}^\infty \log|X_{T_k}-
Y_{T_k}| -\log|X_{T_{k-1}} - Y_{T_{k-1}}| \Biggr\rrvert
\bone_{F}
\]
is stochastically dominated by $Z_*:=\sum_{k\geq1} Z_k$. We have
\begin{eqnarray*}
\E Z_* &= &\sum_{ 1 \leq k \leq k_0-1} \biggl( (c_{38} +
c_{39} m_1)q' + \sum
_{m\geq m_1} c_{40} 2^{-m}c_{39} m
\biggr)
\\
&&{} +\sum_{k\geq k_0} \biggl( q_k \cdot0
+p_1^{k-1} c_{38} + p_1^{k-1}
\sum_{m\geq m_0} c_{39} 2^{-m}c_{39}
m \biggr)< \infty.
\end{eqnarray*}
If we take $G(a)$ to be the cumulative distribution function of $Z_*$,
then the last estimate shows that \eqref{eqn202} is satisfied.
\end{pf}

The next result is an elementary lemma involving distributions and expectations.
Recall the notation from \eqref{eqstdnot}.

%
%le4.3 #&#
\begin{lemma}\label{lemn113}
For any $c_0>0$, $\beta_1 \in(0,1/2)$ there exist $\beta_2,c_1, c_2,
b,\eps_1>0$ such that if $\eps\leq\eps_1$, $x_0 \in\prt D$,
$y_0\in\ol D$, $|x_0 - y_0| = \eps$, $X_0 = x_0$, $Y_0=y_0$
and
%
%e4.81 #&#
%
\begin{equation}
\label{eqn114} \frac{ \llvert   \< y_0 - x_0,
\n (x_0 ) \> \rrvert  } {
\llvert  y_0-x_0 \rrvert } \leq c_0 \eps^{\beta_1},
\end{equation}
then there exists an event $F$ such that
%
%e4.82 #&#
%e4.83 #&#
%
\begin{eqnarray}
\label{eqn115} \P^{x_0,y_0}\bigl(F^c\bigr) &\leq&
c_1 \eps^{\beta_2},
\\
\E^{x_0,y_0} \bigl[( V_1- V_0)
\bone_{F} \bigr]& \geq &c_2.\label{eqn116}
\end{eqnarray}
\end{lemma}

\begin{pf}
It suffices to prove the lemma for $c_0=1$, by the same argument as the
one at the beginning of the proof of Lemma \ref{lemo303}.

First we prove a general claim.
Suppose that a cumulative distribution
function $G\dvtx \R\to[0,1]$ satisfies $\int_{-\infty}^\infty
|a| \,dG(a) < \infty$. Then for every $c_3>0$ there exists
$p_1>0$ such that if $W$ is a random variable which satisfies $\P
(|W|\leq a) \leq
G(a)$ for $a\in\R$ and $\P(W \leq c_3) \leq p_1$, then
$\E W \geq c_3/2$. To see this, let $a_1>- \infty$ be such that
$\int_{(-\infty,a_1]} |a| \,dG(a) < c_3/8$. We choose $p_1>0$ so small
that $|a_1| p_1 < c_3/8$ and $c_3 (1-p_1) > 3c_3/4$. Then
%
%e4.84 #&#
%
\begin{eqnarray}
\label{a151} \E W &\geq&\int_{(-\infty,a_1]} a \,dG(a) -
|a_1| p_1 + c_3 (1- p_1)
\nonumber
\\[-8pt]
\\[-8pt]
\nonumber
&\geq&-
c_3/8 - c_3/8 + 3 c_3/4 = c_3 /2.
\end{eqnarray}
We will apply this observation to $W=(V_1 -
V_0)\bone_F$. By Lemma \ref{lemn81}, there exists an event $F$ such
that $\P(F^c) \leq\eps^{\beta_2}$ and $\P(|V_1-V_0|\bone_F \leq a)
\leq G(a)$ for $a\in\R$ for some $G$ with
$\int_{-\infty}^\infty|a| \,dG(a) < \infty$.
We can choose small $\eps_1, c_3>0$ and apply Lemma \ref{lemo303}
to obtain
\[
\P\bigl((V_1 - V_0)\bone_F \leq
c_3\bigr) \leq\P(V_1 - V_0 \leq
c_3) + \P\bigl(F^c\bigr) \leq p_1/2 +
\eps^{\beta_2} \leq p_1.
\]
We now apply \eqref{a151} to $W=(V_1 -
V_0)\bone_F$ to see that $\E[ (V_1 - V_0)\bone_F]
\geq c_3/2$. We take $c_2 = c_3/2 $ to finish the proof of the lemma.
\end{pf}

\begin{pf*}{Proof of Theorem \ref{eqj135}}
\textit{Step} 1.
In this step, we will define, using induction, a pair of stochastic
processes similar to $X$ and $Y$ on a sequence of random intervals. At
the end of each interval, we check whether the processes have a typical
(and desirable) behavior. If so, we let them continue according to the
original stochastic differential equations. Otherwise, we insert a jump
which brings the processes to a convenient position. We will later
argue that the probability of inserting even a single jump is very
small. We note that this part of the proof could have been presented in
a different way. Instead of inserting jumps, we could have killed the
processes at the time when we insert the first jump. This would have
made the first step of the argument more natural, but it would make the
remaining part of the proof more awkward to present.

Recall that $\sigma'_b = \sigma^X_b \land\sigma^Y_b$ and
\[
\sigma'_{(k+1)b} = \inf \bigl\{t\geq\sigma'_{kb}
\dvtx \bigl( L^{X}_t - L^{X}_{\sigma'_{kb}}
\bigr) \land \bigl( L^{Y}_t - L^{Y}_{\sigma'_{kb}}
\bigr) \geq b \bigr\}
\]
for $k\geq1$. Fix $c_0,\eps_1,b,\beta_1>0$ and $p<1$
such that Lemmas \ref{lemo302}, \ref{lemo303} and \ref{lemn113}
hold with this choice of parameters. Below, the constant $c_0$ will be
denoted~$c_2$.

We will define processes $X^*_t$ and $Y^*_t$ for $t\geq0$ in
an inductive way. Let $X^*_t = X_t$ and $Y^*_t = Y_t$ for $
t\in[0,\sigma'_{b})$.
By Lemma \ref{lemo311}, $Y_{\sigma'_{b}} \ne
X_{\sigma'_{b}}$, $\P^{x,y}$-a.s., for any $x,y\in\ol D$ such that
$x\ne y$. Fix an arbitrary $p_1>0$ and choose $c_1>0$ such that
%
%e4.85 #&#
%
\begin{equation}
\label{eqo315} \P^{x,y}\bigl(|Y_{\sigma'_{b}} - X_{\sigma'_{b}}| \leq
c_1\bigr) < p_1.
\end{equation}
Let $F_1 = \{|Y_{\sigma'_{b}} - X_{\sigma'_{b}}| \geq c_1\}$. Recall
from Section~\ref{secdiffrbm} that $\pi_x$ denotes the projection on
the plane tangent to $\prt D$ at $x\in\prt D$.
Suppose that $\sigma'_b = \sigma_b^X$, recall $c_2=c_0$ defined above
and let
\[
A_1 = \biggl\{ \frac{ \llvert   \< Y_{\sigma'_{b}} - X_{\sigma'_{b}},
\n (X_{\sigma'_{b}} ) \> \rrvert  } {
\llvert  Y_{\sigma'_{b}} - X_{\sigma'_{b}} \rrvert } \leq c_2 |
Y_{0} - X_{0}|^{\beta_1} \biggr\},
\]
%
%
%e4.86 #&#
%
\begin{equation}
\label{eqo306}\qquad \cases{ Y^*_{\sigma'_{b}} = Y_{\sigma'_{b}}, \qquad\mbox{if
$A_1 \cap F_1$ holds,} \vspace*{2pt}
\cr
\displaystyle Y^*_{\sigma'_{b}} = X_{\sigma'_{b}} + \pi_{X_{\sigma'_{b}}}(Y_{\sigma'_{b}-}-X_{\sigma'_{b}})
\frac{\llvert X_{\sigma'_{b}}-Y_{\sigma'_{b}}\rrvert
\lor c_1} {
\llvert
\pi_{X_{\sigma'_{b}}}(Y_{\sigma'_{b}}-X_{\sigma'_{b}})\rrvert }, \vspace*{2pt}\cr
\hspace*{71pt}\mbox{otherwise}.}
\end{equation}
Let $\{Y^*_t, t\in[\sigma'_{b}, \sigma'_{2b})\}$ be the
solution to \eqref{eqj132} with the initial condition given
by~\eqref{eqo306} and driven by Brownian motion $\{B_t, t\in
[\sigma'_{b}, \sigma'_{2b})\}$. Let $X^*_t = X_t$ for $ t\in
[\sigma'_{b}, \sigma'_{2b})$.
Note that, no matter which part of the definition \eqref{eqo306} is
applied, we
have $|Y^*_{\sigma'_{b}} - X^*_{\sigma'_{b}}| \geq
|Y_{\sigma'_{b}} - X_{\sigma'_{b}}|$ and
%
%e4.87 #&#
%
\begin{equation}
\label{eqo307} \frac{ \llvert   \< Y^*_{\sigma'_{b}} - X^*_{\sigma'_{b}},
\n (X^*_{\sigma'_{b}} ) \> \rrvert  } {
\llvert  Y^*_{\sigma'_{b}} - X^*_{\sigma'_{b}} \rrvert } \leq c_2 | Y_{0} -
X_{0}|^{\beta_1}.
\end{equation}
We have
%
%e4.88 #&#
%
\begin{equation}
\label{eqo312} \E^{x,y}\log\bigl\llvert Y^*_{\sigma'_{b}} -
X^*_{\sigma'_{b}}\bigr\rrvert \geq\log c_1 > -\infty.
\end{equation}
If $\sigma'_b = \sigma_b^Y$, then we exchange the roles of $X$
and $Y$ in the above definitions.

The following formulas are a part of the inductive definition, to be
continued below. Let
\begin{eqnarray*}
\sigma^{*}_{0} &=&0,
\\
\sigma^{*}_{kb} &=& \inf \bigl\{t\geq\sigma^{*}_{(k-1)b}
\dvtx \bigl( L^{X^*}_t - L^{X^*}_{\sigma^{*}_{(k-1)b}}
\bigr) \land \bigl( L^{Y^*}_t - L^{Y^*}_{\sigma^{*}_{(k-1)b}}
\bigr) \geq b \bigr\},\qquad k\geq1,
\\
R^*_t &=& \bigl|X^*_t - Y^*_t\bigr|,\qquad M^*_t
= \log R^*_t,\qquad t\geq0,
\\
V^*_k &=&M^*_{\sigma^{*}_{kb}},\qquad k=0,1,\ldots
\end{eqnarray*}
In view of \eqref{eqo307}, we can apply Lemma \ref{lemn113}
to the process $\{(X^*_t,Y^*_t), t\in[\sigma^{*}_{b},
\sigma^{*}_{2b})\}$ to conclude that there exist $c_3 >0$ and an event
$F_{2} \in\sigma(B_t, t\in[\sigma^{*}_{b},
\infty))$ such that, on the event $\{
R^*_{\sigma^{*}_{b}} \leq\eps_1\} $,
\begin{eqnarray*}
\P \bigl(F_{2}^c \mid X^*_{\sigma^{*}_{b}},Y^*_{\sigma
^{*}_{b}}
\bigr) &\leq&\bigl(R^*_{\sigma^{*}_{b}}\bigr)^{\beta_2},
\\
\E \bigl[ \bigl(V^*_2 - V^*_1 \bigr)
\bone_{F_2} \mid X^*_{\sigma^{*}_{b}},Y^*_{\sigma^{*}_{b}} \bigr]
&\geq&
c_3.
\end{eqnarray*}

We proceed with the inductive definition. Suppose that $F_k$,
$X^*_t$ and $Y^*_t$ are already defined for some $k\geq2$ and
$t\in[0, \sigma^{*}_{kb})$. Suppose that $\sigma^{*}_{kb} =
\inf \{t\geq\sigma^{*}_{(k-1)b}\dvtx L^{X^*}_t -
L^{X^*}_{\sigma^{*}_{(k-1)b}} \geq b \}$, and let
\[
A_k = \biggl\{ \frac{ \llvert   \< Y^*_{\sigma^{*}_{kb}-} - X^*_{\sigma^{*}_{kb}-},
\n (X^*_{\sigma^{*}_{kb}-} ) \> \rrvert  } {
\llvert  Y^*_{\sigma^{*}_{kb}-} - X^*_{\sigma^{*}_{kb}-} \rrvert } \leq c_2 \llvert
Y_{\sigma^{*}_{(k-1)b}} - X_{\sigma
^{*}_{(k-1)b}}\rrvert ^{\beta_1} \biggr\},
\]
%
%
%e4.89 #&#
%
\begin{eqnarray}
\label{eqo3016}\qquad \cases{ Y^*_{\sigma^{*}_{kb}} = Y^*_{\sigma^{*}_{kb}-}, &\quad$\mbox{on
$A_k \cap F_k$,}$ \vspace*{2pt}
\cr
Y^*_{\sigma^{*}_{kb}} =
X^*_{\sigma^{*}_{kb}-} %& \vspace*{2pt}
+ \pi_{X^*_{\sigma^{*}_{kb}-}}\bigl(Y^*_{\sigma^{*}_{kb}-}-X^*_{\sigma
^{*}_{kb}-}
\bigr) &\vspace*{2pt}\cr
\displaystyle\qquad\hspace*{8pt}{}\times\frac{\llvert X^*_{\sigma^{*}_{kb}-}-Y^*_{\sigma^{*}_{kb}-}\rrvert
\lor\llvert X^*_{\sigma^{*}_{(k-1)b}}-Y^*_{\sigma^{*}_{(k-1)b}}\rrvert } {
\llvert
\pi_{X^*_{\sigma^{*}_{kb}-}}(Y^*_{\sigma^{*}_{kb}-}-X^*_{\sigma
^{*}_{kb}-})\rrvert }, &\quad$\mbox{otherwise}.$}
\end{eqnarray}
Let $\{(X^*_t,Y^*_t), t\in[\sigma^{*}_{kb},
\sigma^{*}_{(k+1)b})\}$ be the solution to
\eqref{eqj131}--\eqref{eqj132} with the initial conditions
given by $X^*_{\sigma^{*}_{kb}} = X^*_{\sigma^{*}_{kb}-}$ and
\eqref{eqo3016}, and driven by Brownian motion $\{B_t, t\in
[\sigma^{*}_{kb}, \sigma^{*}_{(k+1)b})\}$.
No matter which part of the definition \eqref{eqo3016} is applied,
we have
%
%e4.90 #&#
%
\begin{equation}
\label{a161} \bigl|Y^*_{\sigma^{*}_{kb}} - X^*_{\sigma^{*}_{kb}}\bigr| \geq
\bigl|Y^*_{\sigma^{*}_{kb}-} - X^*_{\sigma^{*}_{kb}-}\bigr|
\end{equation}
and
%
%e4.91 #&#
%
\begin{equation}
\label{eqo3017} \frac{ \llvert   \< Y^*_{\sigma^{*}_{kb}} - X^*_{\sigma^{*}_{kb}},
\n (X^*_{\sigma^{*}_{kb}} ) \> \rrvert  } {
\llvert  Y^*_{\sigma^{*}_{kb}} - X^*_{\sigma^{*}_{kb}} \rrvert } \leq c_2 \bigl\llvert
Y_{\sigma^{*}_{(k-1)b}} - X^*_{\sigma
^{*}_{(k-1)b}}\bigr\rrvert ^{\beta_1}.
\end{equation}
If $\sigma^{*}_{kb} = \inf \{t\geq\sigma^{*}_{(k-1)b}\dvtx L^{Y^*}_t - L^{Y^*}_{\sigma^{*}_{(k-1)b}} \geq b \}$, then
we exchange the roles of $X$ and $Y$ in the above definitions.

In view of \eqref{eqo3017}, we can apply Lemma
\ref{lemn113} to the process $\{(X^*_t,Y^*_t), t\in
[\sigma^{*}_{kb}, \sigma^{*}_{(k+1)b})\}$ to conclude that
there exists an event $F_{k+1} \in\sigma(B_t, t\in
[\sigma^{*}_{kb}, \infty))$ such that, on
the event $\{ R^*_{\sigma^{*}_{kb}} \leq\eps_1\} $,
%
%e4.92 #&#
%
\begin{eqnarray}
\label{eqo317} \P \bigl(F_{k+1}^c \mid
X^*_{\sigma^{*}_{kb}},Y^*_{\sigma
^{*}_{kb}} \bigr)& \leq&\bigl(R^*_{\sigma^{*}_{kb}}
\bigr)^{\beta_2},
\nonumber
\\[-8pt]
\\[-8pt]
\nonumber
\E \bigl[ \bigl(V^*_{k+1} - V^*_k \bigr)
\bone_{F_{k+1}} \mid X^*_{\sigma^{*}_{kb}},Y^*_{\sigma^{*}_{kb}} \bigr]&
\geq&
c_3.
\end{eqnarray}

\textit{Step} 2.
We will show that the probability of the undesirable events $F_k^c$ and
$A_k^c$ is very small.

Definition \eqref{eqo3016} implies that on $F_{k+1}^c$, we
have $V^*_{k+1} \geq V^*_k$. This and the strong Markov property imply
that on the event $\{
R^*_{\sigma^{*}_{kb}} \leq\eps_1\} $,
%
%e4.93 #&#
%
\begin{eqnarray}
\label{eqo313} &&\E \bigl[ V^*_{k+1} - V^*_k \mid\sigma\bigl(
\bigl(X^*_{t},Y^*_t\bigr), t \leq\sigma^{*}_{kb}
\bigr) \bigr]\nonumber\\
&&\qquad = \E \bigl[ V^*_{k+1} - V^*_k \mid
X^*_{\sigma^{*}_{kb}},Y^*_{\sigma^{*}_{kb}} \bigr]
\nonumber\\
&&\qquad=
\E \bigl[ \bigl(V^*_{k+1} - V^*_k \bigr)
\bone_{F_{k+1}} \mid X^*_{\sigma^{*}_{kb}},Y^*_{\sigma^{*}_{kb}} \bigr]
\\
&&\qquad\quad{}+ \E
\bigl[ \bigl(V^*_{k+1} - V^*_k \bigr) \bone_{F_{k+1}^c}
\mid X^*_{\sigma^{*}_{kb}},Y^*_{\sigma^{*}_{kb}} \bigr]
\nonumber\\
&&\qquad\geq c_3+ 0=c_3>0.
\nonumber
\end{eqnarray}

Let $K_1 = \inf\{k\geq1\dvtx \sup_{t\in[\sigma^{*}_{kb},
\sigma^{*}_{(k+1)b}]} R^*_t \geq\eps_1\}$ and $\wt V_k =
V^*_{k\land K_1}$.
It follows from the definition of $V^*_k$'s that all these
random
variables are bounded above by a finite constant because $D$ has a
finite diameter. The estimate \eqref{eqo312}
implies that $\E V^*_1 > -\infty$. It follows from this and
\eqref{eqo313} that
$\E|\wt V_k| < \infty$ for all $k$
and $\{\wt V_k, k\geq1\}$ is a
submartingale. Thus, $\wt V_k$ cannot converge to $-\infty$
with positive probability.

For any fixed $j$, we will estimate the number of $k$ such that
$\wt V_k \in[j, j+1]$.

Let $c_4 = \sup_{x,y \in\ol D} \log|x-y|$ and note that $c_4
< \infty$. We will argue that for any $c_5 \in(-\infty, c_4)$, one can
choose $\eps_1>0$ so small that if $|x-y|\leq\eps_1$, then
$\sup_k \wt V_k \leq c_5$, $\P^{x,y}$-a.s. Let $S = \inf\{t\geq0\dvtx R^*_t
\geq\eps_1\}$ and note that $S \in[\sigma^{*}_{K_1 b}, \sigma
^{*}_{(K_1+1) b}]$. By~\eqref{j271} and the remark following it, for
some $c_6 < \infty$,
\begin{eqnarray*}
\sup_{t\in[0,\sigma^{*}_{(K_1+1) b}]} R^*_t &\leq&\eps_1 \exp
\bigl(c_6 \bigl(L^{X^*}_{\sigma^{*}_{(K_1+1) b}}
-L^{X^*}_S + L^{Y^*}_{\sigma^{*}_{(K_1+1) b}} -
L^{Y^*}_S\bigr)\bigr) \\
&\leq&\eps_1 \exp
\bigl(c_6 (b + b)\bigr).
\end{eqnarray*}
It follows that, for small $\eps_1$, a.s.,
%
%e4.94 #&#
%
\begin{equation}
\label{eqd92} \sup_k \wt V_{k} \leq\log
\eps_1 + 2c_6 b \leq c_5.
\end{equation}
Consider any $c_5 \in(-\infty, c_4)$, assume that $\log\eps_1 +
2c_6 b
\leq c_5$
and fix an integer $j\leq c_5$. Let $U_1 = 0$ and
\begin{eqnarray*}
\wh U_{k} &=& \inf\bigl\{n \geq U_k\dvtx \wt
V_n \notin[j-1, j+2]\bigr\},\qquad k\geq1,
\\
U_{k} &=& \inf\bigl\{n \geq\wh U_k\dvtx \wt
V_n \in[j, j+1]\bigr\},\qquad k \geq2,
\\
K_2^j &=& \sup\{k\dvtx U_k < \infty\},
\end{eqnarray*}
with the convention that $\inf\varnothing=\infty$. The random
variable $K^j_2$ is bounded above by the sum of the number of
upcrossings of the interval $[j-1, j]$ and the number of
downcrossings of the interval $[j+1, j+2]$ by the process $\wt
V_k$. By the upcrossing inequality, in
view of \eqref{eqd92},
%
%e4.95 #&#
%
\begin{equation}
\label{eqo314} \qquad\E K^j_2 \leq \E\bigl(\wt
V_\infty- (j -1)\bigr)^+ + \E\bigl(\wt V_\infty- (j +1)\bigr)^+
+1 \leq2(c_5 - j +2).
\end{equation}

Suppose that $\wt V_{U_k} \in[j, j+1]$ for some $k$.
Let $k_0$ be the smallest integer greater than $3/c_7$, where
$c_7$ has the same value as $c_1$ in Lemma
\ref{lemo303}. Let $p_2$ have the same value as $p$
in Lemma \ref{lemo303}.
We will apply Lemma \ref{lemo303} to estimate
$\wt V_{n+1} - \wt V_n$; this can be done because of \eqref{a161} and
\eqref{eqo3017}.
By Lemma \ref{lemo303} and the strong
Markov property applied at the stopping times $\sigma^{*}_{nb}$,
$n=U_k, U_k+1, \ldots,$ we see that for $c_7,p_2>0$ as chosen above and
$p_3:=
p_2^{k_0+1}$,
\[
\P \bigl( \wt V_{n+1} - \wt V_n \geq c_7, n
= U_k, U_{k}+1,\ldots, U_{k}+
k_0 \mid X^*_{\sigma^{*}_{U_k b}}, Y^*_{\sigma^{*}_{U_k b}} \bigr) \geq
p_2^{k_0+1} = p_3.
\]
If the event in the last formula occurs, then the process $\wt
V$ will leave the interval $[j-1, j+2]$ in at most $k_0+1$
steps, so $\wh U _k - U_k \leq k_0+1$ in this case. If
the\vadjust{\goodbreak}
process $\{\wt V_m, m\geq k\}$ does not leave $[j-1, j+2]$ in $k_0+1$ steps,
then we apply the same argument again, this time using stopping
times $U_{k}+k_0+1, \ldots, U_{k}+ 2k_0+1$. By induction,
the probability that the
process $\{\wt V_m, m\geq k\}$ does not leave $[j-1, j+2]$ in
$r(k_0+1)$ steps is at most $(1-p_3)^r$. It follows that
$(\wh
U _k - U_k )/( k_0+1)$ is majorized by a geometric random
variable with mean $1/p_3$ and, therefore, $\E [\wh U _k - U_k
\mid
X^*_{\sigma^{*}_{U_k b}}, Y^*_{\sigma^{*}_{U_k b}}  ]\leq
(k_0+1)/p_3$. Let $K^j_3$ be the number of $k$ such that $\wt
V_k \in[j,j+1]$. We combine the last estimate with
\eqref{eqo314} to see that
%
%e4.96 #&#
%
\begin{equation}
\label{eqd93} \E K^j_3 \leq2(c_5 - j +2)
(k_0+1)/p_3.
\end{equation}
This, \eqref{eqo315}, \eqref{eqo317} and \eqref{eqd92} yield
\begin{eqnarray*}
\P \biggl( \bigcup_{k\geq1} F^c_k
\biggr) &\leq&\E \biggl[ \sum_{k\geq1}
\bone_{F^c_k} \biggr] = \E\bone_{F^c_1}+ \sum
_{k\geq2} \E\bone_{F^c_k}
\\
&\leq& p_1 + \sum_{j \leq c_5} \E \biggl[ \sum
_{k\dvtx \wt V_{k-1} \in[j,j+1]} \E\bigl(\bone_{F^c_k} \mid\wt
V_{k-1} \in [j,j+1]\bigr) \biggr]
\\
&\leq& p_1 + \sum_{j \leq c_5} \E \biggl[ \sum
_{k\dvtx \wt V_{k-1} \in[j,j+1]} e^{(j+1) \beta_2} \biggr]
\\
&\leq &p_1 + \sum_{j \leq c_5}
e^{(j+1) \beta_2} 2(c_5 - j +2) (k_0+1)/p_3.
\end{eqnarray*}
By \eqref{eqd93} and Lemma \ref{lemo302}, for some
$\beta_3>0$,
\begin{eqnarray*}
\P \biggl( \bigcup_{k\geq1} A^c_k
\biggr) &\leq&\E \biggl[ \sum_{k\geq1}
\bone_{A^c_k} \biggr] = \sum_{k\geq1} \E
\bone_{A^c_k}
\\
&=& \sum_{j \leq c_5} \E \biggl[ \sum
_{k\dvtx \wt V_{k-1} \in[j,j+1]} \E\bigl(\bone_{A^c_k} \mid\wt V_{k-1}
\in [j,j+1]\bigr) \biggr]
\\
&\leq&\sum_{j \leq c_5} \E \biggl[ \sum
_{\wt V_{k-1} \in[j,j+1]} e^{(j+1) \beta_3} \biggr]
\\
&\leq&\sum_{j \leq c_5} e^{(j+1) \beta_3}
2(c_5 - j +2) (k_0+1)/p_3.
\end{eqnarray*}
We combine the last two estimates to obtain
%
%e4.97 #&#
%
\begin{eqnarray}
\label{eqd94}&& \P \biggl( \bigcup_{k\geq1}
A^c_k \cup F^c_k \biggr)
\nonumber
\\[-8pt]
\\[-8pt]
\nonumber
&&\qquad\leq
p_1 + \sum_{j \leq c_5} \bigl(e^{(j+1) \beta_2}
+ e^{(j+1) \beta_3}\bigr) 2(c_5 - j +2) (k_0+1)/p_3.
\end{eqnarray}
Consider an arbitrarily small $p_4>0$. The probability $p_1$ in
\eqref{eqo315} may be chosen to be smaller than $p_4/2$. We make the
sum in \eqref{eqd94} smaller than $p_4/2$
by taking $c_5>-\infty$ sufficiently small.
Then, assuming that $\log\eps_1 + 2c_6 b \leq c_5$,
%
%e4.98 #&#
%
\begin{equation}
\label{eqd91} \P \biggl( \bigcup_{k\geq1}
A^c_k \cup F^c_k \biggr) \leq
p_4.
\end{equation}

\textit{Step} 3.
This step contains soft arguments translating estimates that show that
the distance between $X$ and $Y$ has a tendency to grow into a
statement about the almost sure behavior of the distance process.

Recall that $R_t = |X_t - Y_t|$ and let
$T^R_a = \inf\{t\geq0\dvtx R_t = a\}$.
Recall that $\wt V_k$ does not converge to $-\infty$ at a
finite or infinite time, a.s. If all events $A_k \cap F_k$,
$k\geq1$, hold, then $X^*_t = X_t$ and $Y^*_t = Y_t$ for all
$t\geq0$. This and \eqref{eqd91} imply that for any $p_4>0$,
there exists $\eps_1>0$ such that for any $x,y \in\ol D$, $x
\ne y$, we have $\P^{x,y}(T^R_{\eps_1} < T^R_{0})\geq
1-p_4$.

The process $R_t$ is continuous for all $t\geq0$, a.s.
because the processes $X_t$ and $Y_t$ are continuous.

Suppose that for some $x\ne y$, $p_5:= \P^{x,y}(
T^R_{0}< \infty) >0$. We will show that this assumption
leads to a contradiction. For $j\geq1$, let $S_j = \inf\{t\geq
0\dvtx R_t \leq2^{-j}\}$ and
\[
G_j = \bigl\{ \inf\{t \geq S_j\dvtx R_t =
\eps_1\} < \inf\{t \geq S_j\dvtx R_{t} = 0\}
\bigr\}.
\]
Fix any $j_0$ such that $0< 2^{-j_0} < R_0 \land\eps_1$.
If $T^R_{0}< \infty$, then $S_j < \infty$ for all $j \geq j_0$. It
follows from the strong Markov property applied at $S_j$ that
$\P^{x,y}(\{S_j < \infty\} \cap G_j) \geq p_5(1-p_4)$ for $j \geq
j_0$. Since
$\{S_{j+1} < \infty\} \cap G_{j+1} \subset\{S_j < \infty\}
\cap G_j$, we have $\P^{x,y} (\bigcap_{j\geq j_0} (\{S_j <
\infty\} \cap G_j) ) \geq p_5(1-p_4)>0$. If the event
$\bigcap_{j\geq j_0} (\{S_j < \infty\} \cap G_j)$ holds, then $R$
has a discontinuity at $T^R_{0}$. Since $R$ is continuous
a.s., we have a contradiction which proves that for any $x\ne
y$, $\P^{x,y}( T^R_{0}< \infty)=0$.

Now suppose that $p_6:=
\P(\lim_{t\to\infty} R_t =
0)>0$. If $\lim_{t\to\infty} R_t = 0$, then $S_j <
\infty$ for all $j\geq j_0$. We can argue as above to show that
\[
\P^{x,y} \biggl( \Bigl\{ \lim_{t\to\infty} R_t
= 0 \Bigr\} \cap \bigcap_{j\geq j_0} \bigl(
\{S_j < \infty\} \cap G_j\bigr) \biggr) \geq
p_6(1-p_4)>0.
\]
If the events
$\{T^R_{0}< \infty\}^c$ and
$\bigcap_{j\geq j_0} (\{S_j < \infty\} \cap G_j)$
hold, then $\limsup_{t\to\infty} R_t >0$. Hence, $\P ( \lim_{t\to
\infty} R_t = 0
\mbox{ and } \limsup_{t\to\infty} R_t >0  )>0$. We have a
contradiction which proves that for any $x\ne y$, $\P^{x,y}(
\lim_{t\to\infty} R_t = 0)=0$.
\end{pf*}

\begin{appendix}
%s5 #&#
\section*{Appendix}\label{appendix}

\begin{pf*}{Proof of Lemma \ref{lemj271}}
We have
%
%e5.1 #&#
%
\setcounter{equation}{0}
\begin{eqnarray}
\label{j151} &&\int_0^{2\pi} \log \bigl( \bigl(
\sin^2 \beta+ \cos^2 \beta\cos^2 \alpha
\bigr)^{1/2} \bigr) \,d \beta \nonumber\\
&&\qquad= \int_0^{2\pi}
\log \bigl( \bigl( \cos^2 \beta+ \sin^2 \beta
\cos^2 \alpha \bigr)^{1/2} \bigr) \,d \beta
\nonumber\\
& &\qquad= \int_0^\pi\int_1
^{\cos^2 \alpha} \frac{\sin^2 \beta}{ \cos^2 \beta+ u \sin^2 \beta} \,du \,d\beta
\nonumber
\\
&&\qquad= \int_1 ^{\cos^2 \alpha} \int_{-\infty}^\infty
\frac{1}{1-u} \biggl( \frac{1}{x^2 + u} - \frac{1}{x^2 + 1} \biggr) \,dx
\,du (x = \cot\beta)
\\
&&\qquad= \int_1 ^{\cos^2 \alpha} \frac{1}{\sqrt{u} (1-u)} \biggl[
\arctan \biggl(\frac{x}{\sqrt{u}} \biggr) - \sqrt{u} \arctan(x) \biggr]
_{-\infty}^\infty \,du
\nonumber
\\
&&\qquad = \pi\int_1 ^{\cos^2 \alpha} \frac{1}{\sqrt{u} (1+\sqrt{u})} \,du
\nonumber
\\
&&\qquad= 2 \pi\log \biggl( \frac{1}{2} + \frac{1}{2} |\cos\alpha|
\biggr),
\nonumber
\end{eqnarray}
which implies
\begin{eqnarray*}
&&\int_0^{2\pi} \int_0^\pi
\frac{1}{4\pi} \sin\alpha \log \bigl(\bigl(\sin^2 \beta+
\cos^2 \beta\cos^2 \alpha\bigr)^{1/2} \bigr) \,d
\alpha \,d\beta
\\
&&\qquad = \int_0^{\pi/2} \sin\alpha \log \bigl( (1 +
\cos\alpha)/2 \bigr) \,d\alpha
\\
&&\qquad = \int_{1/2}^1 2 \log y \,dy \bigl(y=(1+\cos
\alpha)/2\bigr)
\\
&&\qquad = \log2 -1.
\end{eqnarray*}
This proves \eqref{eq061}.

We use \eqref{j151} again to see that
\begin{eqnarray*}
&&\int_0^{2\pi} \int_0^\pi
\frac{1}{16\pi} \frac{\sin\alpha} { \sin^3 (\alpha/2)} \log \bigl(\bigl(\sin^2 \beta+
\cos^2 \beta\cos^2 \alpha\bigr)^{1/2} \bigr) \,d
\alpha \,d\beta
\\
&&\qquad = \frac14 \int_0^\pi \frac{\cos(\alpha/2)}{ \sin(\alpha/2) ^2}
\log \biggl( \frac12 + \frac12 |\cos\alpha| \biggr) \,d\alpha
\\
&&\qquad = \int_0^{\pi/4} \frac{ \cos u}{ \sin^2 u} \log(\cos
u) \,du + \int_{\pi/4}^{\pi/2} \frac{ \cos u}{ \sin^2 u} \log(
\sin u) \,du (u = \alpha/2)
\\
&&\qquad = - \biggl[ \frac{\log(\cos u)}{ \sin u} + \log \biggl( \frac{ 1 + \sin u}{ \cos u} \biggr)
\biggr] _0 ^{\pi/4} - \biggl[ \frac{\log(\sin u)}{ \sin u} + \frac1{
\sin u} \biggr] _{\pi/4} ^ {\pi/2}
\\
&&\qquad = \sqrt{2} - 1 - \log(1+\sqrt2).
\end{eqnarray*}
This proves \eqref{eqj276}.
\end{pf*}
\end{appendix}

%s6 #&#
\section*{Acknowledgment}
We are grateful to the referee for many suggestions for improvement, in
particular,
for a short proof of Lemma \ref{lemj271}.

% imsref loaded by akundreckaite, 2013-07-10 13:15:15
%


%suskaldyti doi

\printaddresses

\end{document}